\documentclass[11pt,a4paper, twoside]{article}
\usepackage[utf8]{inputenc}

\usepackage[all]{xy}
\usepackage{latexsym}
\usepackage[margin=1in]{geometry}
\usepackage[dvipdfmx]{graphicx}
\usepackage{overpic}
\usepackage{hyperref}
\usepackage{breakurl}
\usepackage{mymacro}
\usepackage{tikz-cd}
\usepackage{mathtools}
\usepackage{comment}

\usepackage{pinlabel}

\usepackage[export]{adjustbox}

\usepackage{tabularx,booktabs}

\usepackage[
backend=biber,
style=ext-alphabetic,
maxnames = 5,
maxalphanames = 5,
maxcitenames  = 5,
minalphanames = 5,
minnames      = 5,
sorting=nyt,
giveninits,
articlein=false,
url=false,  
doi=false,
eprint=true,
isbn=false
]{biblatex}
\usepackage[whole]{bxcjkjatype}

\usepackage{fancyhdr}

\newcommand\shorttitle{\footnotesize{Non-contractible loops of Legendrian tori from families of knots}}
\newcommand\authors{\footnotesize{Yukihiro Okamoto and Mari\'{a}n Poppr}}

\title{Non-contractible loops of Legendrian tori from families of knots}
\author{Yukihiro Okamoto \and Mari\'{a}n Poppr}
\date{}

\begin{document}

\maketitle

\begin{abstract}
\noindent
In the unit cotangent bundle of $\mathbb{R}^3$, we consider loops of Legendrian tori arising as families of the unit conormal bundles of smooth knots in $\mathbb{R}^3$.
In this paper, using the cord algebra of knots, we give a topological method to compute the monodromy on the Legendrian contact homology in degree $0$ induced by those loops.
As an application, we obtain an infinite family of non-contractible loops of Legendrian tori which are contractible in the space of smoothly embedded tori.
\end{abstract}


\section{Introduction}

\noindent
\textbf{Background.}
Consider the $n$-dimensional Euclidean space $\R^n$ with a standard metric.
Its unit cotangent bundle $U^*\R^n$ has a canonical contact structure.
Given a compact $C^{\infty}$ submanifold $K$ of $\R^n$, its unit conormal bundle is defined by
\[  \Lambda_K \coloneqq \{ (q,p) \in U^*\R^n \mid q\in K,\ p(v)=0 \text{ for all }v\in T_qK \},  \]
and it is a Legendrian submanifold of $U^*\R^n$.
Hence, if we consider the spaces
\[ \begin{array}{rl} \mathcal{K}_{n,k}&\coloneqq \{ k\text{-dimensional  compact }C^{\infty}\text{ submanifolds of }\R^n \}, \\
\mathcal{L}_n&\coloneqq \{\text{compact  Legendrian submanifolds of }U^*\R^n\},
\end{array}\]
for $1\leq k \leq n-1$, we have a map $ \mathcal{K}_{n,k}\to \mathcal{L}_n\colon K\mapsto \Lambda_K$.

It is immediate that for any $K_0,K_1\in \mathcal{K}_{n,k}$,
\[ K_0\text{ is }C^{\infty}\text{ isotopic to }K_1 \text{ in }\R^n 
\Rightarrow \Lambda_{K_0} \text{ is Legendrian isotopic to }\Lambda_{K_1} \text{ in }U^*\R^n.
\]
As a problem of distinguishing Legendrian submanifolds, the following has been studied:
Does the opposite implication ($\Leftarrow$) hold in general? 
This is equivalent to asking whether the induced map $\pi_0(\mathcal{K}_{n,k})\to \pi_0(\mathcal{L}_n)$ on the $0$-th homotopy group is injective.
When $n=3$ and $k=1$, this problem is largely solved.
Shende \cite{Shende} and Ekholm-Ng-Shende \cite{ENS} showed that for any knots $K_0,K_1\in \mathcal{K}_{3,1}$, 
\[ \Lambda_{K_0} \text{ is Legendrian isotopic to }\Lambda_{K_1} \text{ in }U^*\R^n \Rightarrow K_0 \text{ is isotopic to }K_1 \text{ or the mirror of }K_1. \]
There are also studies \cite{Asp, O24} on distinguishing Legendrian non-isotopic unit conormal bundles when $n\geq 4$.
Then, it would be natural to proceed to studying the induced map
\[\pi_m(\mathcal{K}_{n,k},K)\to \pi_m(\mathcal{L}_n,\Lambda_K)\]
on the homotopy groups for $m\geq 1$.

In particular for $m=1$, given a loop $(K_t)_{t\in S^1}$ in $\mathcal{K}_{n,k}$, we ask whether $(\Lambda_{K_t})_{t\in S^1}$ is contractible in $\mathcal{L}_{n,k}$.
This is motivated by preceding researches on loops of Legendrian knots in $\R^3$ with respect to a standard contact structure.
K\'{a}lm\'{a}n \cite{Kalman} found non-contractible loops of Legendrian torus knots.
Casals-Gao \cite{CG} and
Casals-Ng \cite{Casals} gave loops of Legendrian knots which has infinite order under multiplications of the loop.
They computed the monodromy on certain algebraic objects induced by Legendrian loops: \cite{Kalman, Casals} considered the monodromy on the Chekanov-Eliashberg DGA of Legendrian knots, and
\cite{CG} considered the monodromy on the moduli of framed constructible sheaves associated with Legendrian torus links.
Similar to the present paper, Ng \cite{Ng-plane} studied the conormal of oriented immersed curves in $\R^2$ and their loops.
In higher dimension, Sabloff-Sullivan \cite{SS} gave loops of Legendrian spheres in $\R^{2n-1}$ and proved that they are non-contractible by using generating function invariants.
In addition, Golovko \cite{Golovko} considered loops of spherical spuns, based on the construction of \cite{Casals}.
For studies on loops of Legendrian knots from a point of view of h-principle, see \cite{FMP20}.

\

\noindent
\textbf{Main results.}
First, let us fix some notations.
\begin{definition}\label{def-Pi-Leg}
Let $\Lambda_0$ and $\Lambda_1$ be compact Legendrian submanifolds of the $5$-dimensional contact manifold $U^*\R^3$.
We define $\Pi_{\Leg}(\Lambda_0,\Lambda_1)$ as the set of smooth families $(\Lambda_t)_{t\in [0,1]}$ of Legendrian submanifolds of $U^*\R^3$ connecting $\Lambda_0$ and $\Lambda_1$, modulo smooth homotopy equivalence $\sim$ such that
\[ \begin{array}{rl} & (\Lambda^0_t)_{t\in [0,1]}\sim (\Lambda^1_t)_{t\in [0,1]} \vspace{5pt} \\
\Leftrightarrow & \text{there is a smooth family }(\hat{\Lambda}^s_t)_{(s,t)\in [0,1]^2} 
\text{ of Legendrian submanifolds of }U^*\R^3 \text{ satisfying} \\
& \hat{\Lambda}^s_0=\Lambda_0 \text{ and } \hat{\Lambda}^s_1=\Lambda_1 \text{ for any }s\in [0,1], 
\text{ and }\hat{\Lambda}^0_t=\Lambda^0_t \text{ and }\hat{\Lambda}^1_t=\Lambda^1_t \text{ for any }t\in[0,1].
\end{array}\]
%
In addition, forgetting the contact structure and replacing $(\Lambda_t)_{t\in[0,1]}$ and $(\hat{\Lambda}^s_t)_{(s,t)\in [0,1]^2}$ in the above definition with smooth families of embedded surfaces,
we define a set $\Pi_{\sm}(\Lambda_0,\Lambda_1)$ together with a natural map 
$\iota \colon \Pi_{\Leg}(\Lambda_0,\Lambda_1) \to \Pi_{\sm}(\Lambda_0,\Lambda_1)$.
\end{definition}
By concatenating $1$-parameter families, we can naturally define a map
$ \Pi_{\Leg}(\Lambda_1,\Lambda_2) \times \Pi_{\Leg}(\Lambda_0,\Lambda_1) \to \Pi_{\Leg}(\Lambda_0,\Lambda_2)$.
This gives a groupoid $\Pi_{\Leg}$ such that the objects are compact Legendrian submanifolds of $U^*\R^3$ and the sets of morphisms are given by Definition \ref{def-Pi-Leg}.
When $\Lambda_0=\Lambda_1=\Lambda$, we denote the group
$ \Pi_{\Leg}(\Lambda,\Lambda)$ by $\pi_{\Leg} (\Lambda)$.
Similarly, $\pi_{\sm}(\Lambda) \coloneqq \Pi_{\sm}(\Lambda,\Lambda)$ has a group structure.

Analogous to $\Pi_{\sm}(\Lambda_0,\Lambda_1)$, for any smooth knots $K_0$ and $K_1$ in $\R^3$, we define the set $\Pi_{\sm}(K_0,K_1)$ consisting of homotopy classes of smooth families $(K_t)_{t\in [0,1]}$ of smooth knots in $\R^3$ connecting $K_0$ and $K_1$.
Then, we can define a map
\[ \bLambda \colon \Pi_{\sm}(K_0,K_1) \to \Pi_{\Leg}(\Lambda_{K_0},\Lambda_{K_1}) \colon [ (K_t)_{t\in [0,1]} ] \mapsto [(\Lambda_{K_t})_{t\in [0,1]}]. \]


The goal of this paper is to give a sufficient condition, which is described in a topological way, to be $\bLambda(x)\neq \bLambda(y)$ for given $x,y\in \Pi_{\sm}(K_0,K_1)$.
We will also give examples of applying this condition.

To state the main result, we use the \textit{cord algebra} of knots.
This was introduced by Ng \cite[Definition 2.1]{Ng} as a $\Z[H_1(\Lambda_K)]$-algebra for an oriented knot $K$ with a framing.
In this paper, we mainly deal with its reduction to $\Z_2$-coefficient and denote it by $\Cord(K)$.
Taking a parallel copy $K'$ of $K$ with respect to a framing, one can define $\Cord(K)$ as a $\Z_2$-algebra generated by the set of homotopy classes of paths in $\R^3\setminus K$ with endpoints in $K'$,
modulo certain `skein relations' illustrated in Figure \ref{fig-cord-Z2}.
Up to natural isomorphisms, the algebra does not depend on the choice of framing.

Given any $x\in \Pi_{\sm}(K_0,K_1)$ represented by $(K_t)_{t\in [0,1]}$, we take an ambient isotopy $(F_t)_{t\in [0,1]}$ on $\R^3$ such that $F_0=\id_{\R^3}$ and $F_{1-t}(K_t) = K_1$. In particular, $ F_1(K_0) = K_1$.
Then, $F_1$ induces an isomorphism from $\Cord(K_0)$ to $ \Cord(K_1)$, and we denote it by $\rho(x)$ (see Section \ref{subsubsec-isom-on-cord}).
In this way, we construct a map
\[ \rho \colon \Pi_{\sm}(K_0,K_1) \to \Hom (\Cord(K_0),\Cord (K_1)) \colon x\mapsto \rho(x) ,\]
and, in fact, it is well-defined.
Similarly to $\Pi_{\Leg}$, we can define a groupoid $\Pi_{\sm}$ from $\Pi_{\sm}(K_0,K_1)$ for all knots $K_0$ and $K_1$. Then, $\rho$ is regarded as a functor from $\Pi_{\sm}$ to the category of $\Z_2$-algebras.

\begin{theorem}\label{thm-1}
The map $\rho$ factors through $\boldsymbol{\Lambda}$, that is, there is a map
\[ \Theta \colon \Pi_{\Leg}(\Lambda_{K_0},\Lambda_{K_1}) \to \Hom (\Cord(K_0),\Cord (K_1)) \]
such that the following diagram commutes:
\[\xymatrix{
& \Pi_{\Leg}(\Lambda_{K_0},\Lambda_{K_1}) \ar[d]^-{\Theta} \\
\Pi_{\sm}(K_0,K_1)\ar[ur]^-{\bLambda} \ar[r]_-{\rho} & \Hom(\Cord(K_0), \Cord (K_1)) .
}\]
\end{theorem}

Let us explain the idea of the proof.
First, the map $\Theta$ is constructed in the following way:
For a compact Legendrian submanifold $\Lambda$ of $U^*\R^3$ whose Maslov class vanishes, its \textit{Legendrian contact homology} $\LCH_*(\Lambda)$ is defined with coefficients in $\Z_2$ by using pseudo-holomorphic curves in the symplectization $\R\times U^*\R^3$.
Moreover, given a Legendrian isotopy $(\Lambda_t)_{t\in [0,1]}$ from $\Lambda_{K_0}$ to $\Lambda_{K_1}$, we can construct an isomorphism from $\LCH_*(\Lambda_{K_0})$ to $\LCH_*(\Lambda_{K_1})$ and this depends only on the equivalence class of $(\Lambda_t)_{t\in [0,1]}$ in $\Pi_{\Leg}(\Lambda_{K_0},\Lambda_{K_1})$ \cite{EHK, DR}.
We should note that when $\Lambda_{K_0}=\Lambda_{K_1}$, the automorphism on $\LCH_*(\Lambda_{K_0})$ is considered as the monodromy for the loop $(\Lambda_t)_{t\in[0,1]}$ with $\Lambda_0=\Lambda_1=\Lambda_{K_0}$.

The key fact is that there is an isomorphism
\[ \LCH_0(\Lambda_K) \cong  \Cord(K)\]
for any knot $K$.
For the construction of this isomorphism, we refer to \cite{CELN}.
%
%
Then, for $x\in \Pi_{\Leg}(\Lambda_{K_0},\Lambda_{K_1})$ represented by $(\Lambda_t)_{t\in [0,1]}$, $\Theta(x)$ is defined by the composition of isomorphisms
\[ \Cord(K_0) \cong \LCH_0(\Lambda_{K_0})\to \LCH_0(\Lambda_{K_1}) \cong \Cord(K_1). \]

Next, in order to show the equality $\rho =\Theta\circ \bLambda$, we need to compute $\Theta(x)$ when $x$ is represented by a family $(\Lambda_{K_t})_{t\in [0,1]}$ for $[(K_t)_{t\in [0,1]}] \in \Pi_{\sm}(K_0,K_1)$.
The isomorphism $\LCH_0(\Lambda_K) \to  \Cord(K)$ from \cite{CELN} is described by using moduli spaces of pseudo-holomorphic curves in $T^*\R^3$ with boundary in $L_K\cup \R^3$ and punctures asymptotic to Reeb chords of $\Lambda_K$. 
Here, $L_K$ is the conormal bundle of $K$ and $\R^3$ is identified with the zero section of $T^*\R^3$.
More precisely, in \cite[Theorem 1.2]{CELN}, the moduli space is used to construct an isomorphism from $\LCH_0(\Lambda_K)$ to the $0$-th degree part of the \textit{string homology} of $K$ (see Section \ref{subsec-string}). 
It is generated by $0$-chains in a space of tuples of paths in $\R^3$ with endpoints in $K$, and naturally isomorphic to $\Cord(K)$.

Take a generator $\xi\in \Cord(K_0)\cong \LCH_0(\Lambda_{K_0})$ corresponding to a Reeb chord of $\Lambda_{K_0}$.
We want to show that the images of $\xi$ under the two maps $\Theta(x), \rho\left( [(K_t)_{t\in [0,1]}]\right) \colon \Cord(K_0)\to \Cord(K_1)$ coincide.
They are represented by $0$-chains of the string homology of $K_1$, and we need to show that these $0$-chains are homologous via a $1$-chain.

In Section \ref{subsubsec-proof-main},
we construct the required $1$-chain by using parametrized moduli spaces and their compactification.
In more detail, we construct in $T^*\R^3$ a certain family of exact Lagrangian fillings $(L_b)_{b\in [0,\infty)}$ of $\Lambda_{K_0}$ which intersect $\R^3$ cleanly along knots: $K_0$ if $b=0$ and $K_1$ if $b\gg 1$ (see Section \ref{subsubsec-family-of-Lag}). Then, we consider moduli spaces of pseudo-holomorphic curves with boundary in $L_b \cup \R^3$ for some $b\in [0,\infty)$ and punctures asymptotic to Reeb chords of $\Lambda_{K_0}$.
The $1$-dimensional moduli spaces are compactified by using the ideas in \cite{EHK, CELN}.
(In practice, we deal with moduli spaces parametrized over $[0,b_*]$ for $b_* \gg 0$.)


When $K_0=K_1=K$, the above argument amounts to computing the monodromy on $\LCH_0(\Lambda_K) \cong \Cord(K)$ for a loop $(\Lambda_{K_t})_{t\in [0,1]}$.
See Theorem \ref{thm-monodromy}.

\begin{remark}
For any knot $K$,
the Legendrian contact homology of $\Lambda_K$ can be defined with coefficients in $\Z_2[\pi_1(\Lambda_K)]\cong \Z_2[\lambda^{\pm},\mu^{\pm}]$ (and in $\Z[\pi_1(\Lambda_K)]$ if a spin structure of $\Lambda_K$ is fixed.)
The isomorphism $\LCH_0(\Lambda_K) \cong \Cord(K)$ over $\Z_2$ is a reduction of the isomorphism in \cite[Theorem 1.2]{CELN} over $\Z_2[\pi_1(\Lambda_K)]$ (see also Appendix \ref{app-isom}).
Thus, we expect that is is possible to generalize Theorem \ref{thm-1} by using the cord algebras over $\Z_2[\pi_1(\Lambda_K)]$, but it would not be so straightforward for some reasons in Remark \ref{rem-simplify} below.
\end{remark}

Furthermore, Theorem \ref{thm-1} can be used to give an infinite family of non-contractible loops of Legendrian tori in $U^*\R^3$.
Consider the $(2,3)$-torus knot $T_{2,3}$ and its mirror $T_{2,-3}$.
\begin{theorem}\label{thm-2}
Let $K$ be one of the following knots:
\begin{itemize}
    \item[(i)] the square knot $T_{2, 3}\#T_{2, -3}$,
    \item[(ii)] $T_{2, 3}\#(T_{2, -3}\# L)$ for an arbitrary knot $L$,
    \item[(iii)] $(2k+1, 1)$-cable of the square knot $T_{2, 3}\#T_{2, -3}$, where $k\in \Z_{\geq 0}$. 
\end{itemize}
Then, there exists a non-trivial element $x\in \pi_{\Leg}(\Lambda_K)$ such that $\iota(x)$ is the unit of $\pi_{\sm}(\Lambda_K)$.
\end{theorem}

This result follows from Theorem \ref{thm-1} by finding loops of knots $(K_t)_{t\in [0,1]}$ with $K_0=K_1=K$ such that $\rho ([(K_t)_{t\in [0,1]}])\in \Aut(\Cord(K))$ is not the identity map, and the loop $(\Lambda_{K_t})_{t\in [0,1]}$ is contractible as a loop of smooth submanifolds of $U^*\R^3$.
The idea of such loops in this paper is similar to `purple-box Legendrian loops' introduced in \cite[Section 2.4]{Casals}.
See Section \ref{sec:mapping_spaces}.

\begin{remark}\label{rem-simplify}
There are three points simplified by reducing the coefficient ring from $\Z_2[\pi_1(\Lambda_K)]$ to $\Z_2$:
(1) The isomorphism $\LCH_0(\Lambda_K)\to \Cord(K)$ over $\Z_2$ depends only on $K$,
while the isomorphism of \cite[Theorem 1.2]{CELN} over $\Z_2[\pi_1(\Lambda_K)]$, including the definition of the Legendrian contact homology, depends on the choice a base point of $\Lambda_K$. See Section \ref{app-isom}.
(2) In the proof of Theorem \ref{thm-1}, we can avoid the discussion about `spikes' as in \cite[Section 2.1, 4.3]{CELN}, which requires some local analysis on pseudo-holomorphic curves.
This makes the argument in Section \ref{subsubsec-proof-main} less complicated.
(3) A combinatorial description of the cord algebra over $\Z_2$ is significantly simpler (but still non-trivial) than the case of $\Z_2[\pi_1(\Lambda_K)]$. See Example \ref{exmpl_nontriv1}.

\end{remark}

\noindent
\textbf{Possible approach using microlocal sheaf theory.}
Fix a field $\bold{k}$.
For a topological space $X$, let $\mathrm{loc}(X)$ denote the category of locally constant sheaves on $X$ of $\bold{k}$-vector spaces.
Note that there is a functor $\mathrm{loc}(X) \to \mathrm{loc}(\{x\})$ taking the stalks at $x\in X$.
Let $K_0$ and $K_1$ be knots in $\R^3$ and consider a contact Hamiltonian isotopy $\Phi=(\Phi_t)_{t\in [0,1]}$ on $U^*\R^3$ with compact support.
Fix $x_0\in \R^3$ such that its fiber $U_{x_0}^*\R^3$ is sufficiently far from the support of $\Phi$.
In \cite[Theorem 7]{Shende}, Shende proved that if $\Phi_1(\Lambda_{K_0})=\Lambda_{K_1}$, the Guillermou-Kashiwara-Schapira kernel of $\Phi$ can be used to give an equivalence of categories $\mathrm{loc}(\R^3\setminus K_0) \to \mathrm{loc}(\R^3\setminus K_1)$ which commutes with the functors taking the stalks at $x_0$.
Furthermore, from this equivalence and the left-orderablity of the group $\pi_1(\R^3\setminus K_0,x_0)$, an isomorphism
$ \pi_1(\R^3\setminus K_1,x_0) \to  \pi_1(\R^3 \setminus K_0,x_0)$
is deduced \cite[Theorem 8]{Shende}.
Let us denote it by $f_{\Phi}$.
We expect to construct a well-defined map
\[ \Pi_{\Leg}(\Lambda_K)\to \Hom ( \pi_1(\R^3\setminus K_1,x_0),  \pi_1(\R^3\setminus K_0,x_0) ) \colon [(\Lambda_{t})_{t\in [0,1]}] \mapsto f_{\Phi} , \]
where $\Phi$ is an extension of a Legendrian isotopy $(\Lambda_{t})_{t\in [0,1]}$.
Using this map when $K_0=K_1$, one might be able to detect non-trivial loops of Legendrian tori.



\

\noindent
\textbf{Organization of paper.}
In Section \ref{sec-LCH}, the Legendrian contact homology for Legendrian submanifolds in $U^*\R^3$ are defined with coefficients in $\Z_2$.
We also study in detail the isomorphisms induced by  Legendrian isotopies.
In Section \ref{sec-moduli}, referring to \cite{CEL,CELN}, we introduce moduli spaces of pseudo-holomorphic curves in $T^*\R^3$  such that the boundary condition is parametrized by a manifold.
In Section \ref{sec-string}, referring to \cite[Section 2]{CELN}, we introduce the string homology and cord algebra for knots in $\R^3$ with coefficients in $\Z_2$.
In Section \ref{subsubsec-isom-LCH-str}, for a knot $K$, we describe how to construct an isomorphism from the Legendrian contact homology of $\Lambda_K$ to the string homology of $K$.
In Section \ref{sec-coincidence}, we prove Theorem \ref{thm-1} by using the moduli spaces in Section \ref{sec-moduli} and the isomorphism in Section \ref{subsubsec-isom-LCH-str}.
In Section \ref{sec-examples-of-detection}, we apply Theorem \ref{thm-1} to concrete examples and prove Theorem \ref{thm-2}.
Referring to the proof of \cite[Theorem 1.2]{CELN}, Appendix \ref{subsec-proof-isom} shows that the map constructed in Section \ref{subsubsec-isom-LCH-str} is an isomorphism of $\Z_2$-algebra.

\noindent
\textbf{Notations.} In this paper, for any finite set $S$, let $\#_{\bZ_2}S\in \Z_2$ denote the cardinality of $S$ modulo $2$.
In addition, for any paths $\gamma_i\colon [0,T_i]\to \R^3$ ($i=0,1$) such that $\gamma_1(T_1)=\gamma_2(0)$, let $\gamma_1\bullet \gamma_2\colon [0,T_1+T_2]\to \R^3$ denote the concatenation of them.

\

\noindent
\textbf{Acknowledgements.}
This research started while the first author was staying at Universit\"{a}t Augsburg.
We thank Kai Cieliebak, the former supervisor of the second author, for making time for discussion and giving valuable comments.
We also thank Tomohiro Asano, Kei Irie, Tamas K\'{a}lm\'{a}n, Javier Mart\'{i}nez-Aguinaga and  Lenhard Ng for stimulating conversations.
Y.O. is supported by JSPS KAKENHI Grant Number JP25KJ0270.

\section{Legendrian contact homology and Legendrian isotopies}\label{sec-LCH}

\subsection{Preliminaries}
Consider the $3$-dimensional Euclidean space $\R^3$ with the coordinates $(q_1,q_2,q_3)$. 
On the cotangent bundle $T^\ast\mathbb{R}^3$, we have canonical coordinates $(q,p)=((q_1,q_2,q_3), (p_1,p_2,p_3))$, where $p=(p_1, p_2, p_3)$ are the coordinates of fibers.
The canonical Liouville $1$-form on $T^*\R^3$ is defined by $\lambda_{\bR^3}\coloneqq \sum_{i=1}^3p_idq_i$.

Fix a standard metric $\la \cdot,\cdot \ra = \sum_{i=1}^3 dq_i\otimes dq_i$ on $\R^3$.
The unit cotangent bundle $U^\ast\mathbb{R}^3$ is a $5$-dimensional contact manifold with a contact form $\alpha\coloneqq \rest{\lambda_{\R^3}}{U^*\R^3}$.
Let $R_{\alpha}$ denote the Reeb vector field associated with $\alpha$.
This is written as $R_{\alpha}=\sum_{i=1}^3 p_i \partial_{q_i}$.
\begin{remark}\label{rem-1-jet}
Let $S^2$ be the unit sphere in $\R^3$.
It is often useful to identify $U^*\R^3$ with the $1$-jet bundle $J^1S^2$ by a diffeomorphism
\[ U^*\R^3 \to J^1S^2 = \R\times T^*S^2 \colon (q,p) \mapsto (\la q,p \ra, (p, q- \la q,p \ra p)), \]
by which $\alpha$ corresponds to $dz -\lambda_{S^2}$.
Here, $z$ is the coordinate of $\R$ and $\lambda_{S^2}$ is the canonical Liouville form on $T^*S^2$.
Let $\pi_{T^*S^2}\colon U^*\R^3 \cong J^1S^2 \to T^*S^2 $ denote 
the projection.
\end{remark}

Let $\Lambda$ be a compact connected Legendrian submanifold of $U^*\bR^3$ whose Maslov class vanishes. 
The set of its Reeb chords is defined by
\[ \mathcal{R}(\Lambda) \coloneqq \{ a\colon [0,T]\to U^*\bR^3 \mid T>0, \dot{a}(t)= (R_{\alpha})_{a(t)} \text{ for every }t\in [0,T], \text{ and }a(0),a(T) \in \Lambda \}.\]
For any $(a\colon [0,T]\to U^*\bR^3)\in \mathcal{R}(\Lambda)$, let us denote $T_a \coloneqq T = \int a^*\alpha$.

Let $(\varphi_{\alpha}^t)_{t\in \bR}$ be the flow of $R_{\alpha}$.
Denote the contact distribution by $\xi = \ker\alpha$. 
Then, a Reeb chord $a\in \mathcal{R}(\Lambda)$ is called \textit{non-degenerate} if $(d\varphi_{\alpha}^{T_a})_{a(0)}(T_{a(0)}\Lambda)$ is transverse to $T_{a(T_a)}\Lambda$ in $\xi_{a(T_a)}$.
For every $a\in \mathcal{R}(\Lambda)$ which is non-degenerate, we choose a capping path $\gamma_a\colon [0,1]\to \Lambda$ such that $\gamma_a(0)=a(T_a)$ and $\gamma_a(1)= a(0)$.
Using the family of Lagrangian planes $(T_{\gamma_a(t)}\Lambda)_{t\in [0,1]}$ in $\gamma_a^*\xi$, the Maslov index $\mu(a)\in \bZ$ is defined. 
For details, see \cite[Section 2.2]{EES-R}.
From the assumption that the Maslov class of $\Lambda$ vanishes, $\mu(a)$ is independent of the choice of $\gamma_a$.
We set
\[|a| \coloneqq \mu(a) -1 \in\bZ.\]

Let $D\coloneqq\{z\in \bC \mid |z|\leq 1\}$.
For any $m\in\bZ_{\geq0}$, we fix boundary points $p_0,p_1,\dots ,p_m\in \partial D$ ordered counterclockwise.
Then we define $D_{m+1}\coloneqq D\setminus \{p_0,\dots,p_m\}$.
Let $\mathcal{C}_{m+1}$ be the space of conformal structures on $D_{m+1}$ which can be smoothly extended to $D$.
For each $\kappa\in \mathcal{C}_{m+1}$, we take $j_{\kappa}$ as a complex structure on $D_{m+1}$ which associates $\kappa$.
On $\bR\times [0,1]$ with the coordinates $(s,t)$, a complex structure $j_0$ is determined by $j_0(\partial_s) = \partial_t$.
Near each boundary puncture of the Riemann surface $(D_{m+1},j_{\kappa})$, we fix strip coordinates by biholomorphic maps
\begin{align*}
\psi_0 &\colon [0,\infty)\times [0,1] \to D_{m+1}, \\
\psi_k & \colon (-\infty ,0] \times [0,1] \to D_{m+1} \text{ for }k=1,\dots ,m,
\end{align*}
with respect to $j_0$ and $j_{\kappa}$ such that $\lim_{s\to \infty}\psi_0(s,t) = p_0$ and $\lim_{s\to -\infty} \psi_k(s,t) = p_k$ uniformly on $t\in [0,1]$.

\subsection{$\Z_2$-coefficient Legendrian contact homology}

\subsubsection{Moduli space of pseudo-holomorphic curves in $\bR\times U^*\bR^3$}\label{subsubsec-symplectization-moduli}

Hereafter, we identify $\R^3$ with the zero section of $T^*\R^3$. In addition, we will frequently use a diffeomorphism
\begin{align}\label{diffeo-symp-cot}
E\colon \bR\times U^*\bR^3 \to  T^*\bR^3 \setminus \bR^3 \colon (r, (q,p)) \mapsto (q,e^r \cdot p), 
\end{align}
for which $E^* (\lambda_{\bR^3}) = e^r \alpha$ holds.

An almost complex structure $J_0$ on $\bR\times U^*\bR^3$ is defined via $E$ such that
\[ (E_*J_0)_{(q,p)} \colon T_{(q,p)}(T^*\bR^3) \to T_{(q,p)}(T^*\bR^3) \colon \partial_{q_i} \mapsto -|p|\partial_{p_i},\ \partial_{p_i}\mapsto |p|^{-1} \partial_{q_i}\ (i=1,2,3) \]
for every $(q,p)\in T^*\bR^3\setminus \bR^3$.
For every $s\in \bR$, we set a map
\[ \tau_s \colon \bR\times U^*\bR^3 \to \bR\times U^*\bR^3 \colon (r, x) \mapsto (r+s,x).\]
Then, one can check that $\tau_s^*J_0=J_0$.

Let $\Lambda_+$ and $\Lambda_-$ be compact connected Legendrian submanifolds of $U^*\bR^3$.
Let $L$ be an exact Lagrangian cobordism in $\bR\times U^*\bR^3$ whose positive end is $\Lambda_+$ and negative end is $\Lambda_-$.
In the present case where $\Lambda_-$ is connected,
this means that $\rest{(e^r \alpha)}{L}$ is an exact $1$-form on $L$ and there exist $r_-<r_+$ such that
\[\begin{array}{cc} L\cap \left( [r_+,\infty)\times U^*\bR^3\right) = [r_+,\infty)\times \Lambda_+ , & L\cap \left( (-\infty,r_-]\times U^*\bR^3\right) = (-\infty,r_-]\times \Lambda_-, \end{array}\]
and $L\cap ([r_-,r_+]\times U^*\bR^3)$ is compact.
For such an exact Lagrangian cobordism $L$,
we define a set $\mathcal{J}_L$ consisting of almost complex structures $J$ on $\bR\times U^*\bR^3$ satisfying:
\begin{itemize}
\item $d(e^r \alpha) (\cdot, J\cdot)$ is a Riemannian metric on $\bR\times U^*\bR^3$.
\item $\tau_s^*J=J$ on $[r_+,\infty)\times U^*\bR^3$ for every $s\geq 0$, and $\tau_s^*J=J$ on $(-\infty,r_-]\times U^*\bR^3$ for every $s\leq 0$.
\item Outside $[r_-,r_+]\times U^*\bR^3$, $J(\xi) \subset \xi$ and $J(\partial_r)=R_{\alpha}$.
\item There exists $d>0$ such that for $U^*B_d\coloneqq \{(q,p)\in U^*\R^3 \mid |q|<d\}$, $L$ is contained in $ \R\times U^*B_d$ and $J$ coincides with $J_0$ outside $\R\times  U^*B_d$.
\end{itemize}
$\mathcal{J}_L$ is equipped with the $C^{\infty}$ topology.

Let $a\in \mathcal{R}(\Lambda_+)$ and $a_1,\dots ,a_m\in \mathcal{R}(\Lambda_-)$.
For $J\in\mathcal{J}_L$, we define $\calM_{L,J}(a;a_1,\dots ,a_m)$ as the moduli space (modulo conformal automorphisms on $D_{m+1}$) consisting of pairs $(u,\kappa)$
of $\kappa\in \mathcal{C}_{m+1}$ and a $C^{\infty}$ map $u \colon D_{m+1}\to \bR\times U^*\bR^3$
such that:
\begin{itemize}
\item $du + J\circ du \circ j_{\kappa}=0$.
\item $u(\partial D_{m+1})\subset L$.
\item There exist $s_0,s_1,\dots ,s_m \in \bR$ such that
\begin{align*}
&\lim_{s\to \infty} \tau_{-T_a s} \circ u \circ \psi_0 (s ,t) = (s_0,a(T_a t)), \\
&\lim_{s\to -\infty} \tau_{-T_{a_k}s}\circ u \circ \psi_k (s,t) = (s_k,a_k(T_{a_k}t)) \text{ for }k=1,\dots ,m,
\end{align*}
$C^{\infty}$ uniformly on $t\in[0,1]$.
\end{itemize}
Note that the boundary of $\R\times U^*B_d$ is $J_0$-convex, that is,
the image of $u$ is contained in the closure of $\R\times U^*B_d$ for any $(u,\kappa)\in \calM_{L,J}(a;a_1,\dots ,a_m)$.
When $m=0$, we denote this moduli space by $\calM_{L,J}(a;\emptyset)$.

If $\Lambda_+=\Lambda_-=\Lambda$ and $L=\bR\times \Lambda$, we consider a subset $\mathcal{J}_{\bR\times \Lambda}^{\mathrm{cyl}} \subset \mathcal{J}_{\bR\times \Lambda}$ consisting of $J\in \mathcal{J}_{\bR\times \Lambda}$ such that $\tau_s^* J =J$ for every $s\in \bR$.
When $J\in \mathcal{J}_{\bR\times \Lambda}^{\mathrm{cyl}}$,
the group $\bR$ acts on the moduli space by $(u,\kappa)\cdot s = (\tau_s\circ u,\kappa)$ for every $s\in \bR$, so
we define the quotient space
\[ \bar{\calM}_{\Lambda,J}(a;a_1,\dots ,a_m) \coloneqq \calM_{\bR\times \Lambda,J}(a;a_1,\dots ,a_m) /\bR. \]

\subsubsection{Definition of Chekanov-Eliashberg DGA and its homology}\label{subsubsec-CE-DGA}

Let $\Lambda$ be a compact connected Legendrian submanifold of $U^*\bR^3$ such that its Maslov class vanishes and every Reeb chord of $\Lambda$ is non-degenerate.
For generic $J\in \mathcal{J}_{\bR\times \Lambda}^{\mathrm{cyl}}$, we can associate a DGA over $\Z_2$ as follows.
For the construction, we refer to \cite{DR, DR-lift}. See also \cite{EES-R,EES}.

By \cite[Proposition 3.13]{DR}, for generic $J\in \mathcal{J}_{\bR\times \Lambda}^{\mathrm{cyl}}$,
the moduli space $\bar{\calM}_{\Lambda,J}(a;a_1,\dots ,a_m)$ is cut out transversely for every $a,a_1,\dots ,a_m\in \mathcal{R}(\Lambda)$ such that $(a_1,\dots ,a_m)\neq (a)$.
It is a manifold of dimension $|a|-1-\sum_{k=1}^m|a_k|$ (see \cite[Section 4.2.4]{DR-lift} and \cite[Proposition 2.3]{EES}).
Moreover, by Gromov-Floer compactness \cite[Section 3.3.2]{DR}, it admits a compactification consisting of $J$-holomorphic buildings with boundary in $\bR\times \Lambda$. In particular, when $|a|-1-\sum_{k=1}^m|a_k|=0$, $\bar{\calM}_{\Lambda,J}(a;a_1,\dots ,a_m)$ is a compact $0$-dimensional manifold.

Let $\calA_*(\Lambda)$ be the unital non-commutative $\bZ_2$-algebra freely generated by $\mathcal{R}(\Lambda)$ such that $a\in \calA_{|a|}(\Lambda)$ for every $a\in \mathcal{R}(\Lambda)$.
As a vector space, it is generated by words of Reeb chords of $\Lambda$, including the empty word $1$.
We define a derivation $d_J \colon \calA_*(\Lambda) \to \calA_{*-1}(\Lambda)$ of degree $-1$ such that for every $a\in \mathcal{R}(\Lambda)$
\[ d_J (a) \coloneqq \sum_{|a_1|+\dots +|a_m|=|a|-1} \left( \#_{\bZ_2} \bar{\calM}_{\Lambda,J}(a;a_1,\dots ,a_m) \right) a_1\cdots a_m \]
and extend it by the Leibniz rule.
Then, $d_J\circ d_J=0$ holds. For the proof, see \cite[Section 3.3.2]{DR}.
The differential graded algebra $(\mathcal{A}_*(\Lambda),d_J)$ is called the \textit{Chekanov-Eliashberg DGA} of $\Lambda$. 

We define the $\bZ_2$-coefficient \textit{Legendrian contact homology} of $\Lambda$ by
\[ \LCH_*(\Lambda) \coloneqq \Ker d_J / \Image d_J .  \]
Its isomorphism class as a unital graded $\bZ_2$-algebra is independent of $J$ and invariant under Legendrian isotopies of $\Lambda$.
For the proof, we refer to \cite[Theorem 3.5]{DR} using exact Lagrangian cobordisms constructed from Legendrian isotopies.
See also Section \ref{subsec-isom-cobordism} below.


\subsubsection{DGA map associated with exact Lagrangian cobordism}\label{subsubsec-DGA-map}

Let $L$ be an exact Lagrangian cobordism in $\bR\times U^*\bR^3$ whose positive end is $\Lambda_+$ and negative end is $\Lambda_-$. We assume that the Maslov class of $L$ vanishes, $\Lambda_{\pm}$ is connected and every Reeb chord of $\Lambda_{\pm}$ is non-degenerate.
We take almost complex structures $J_{+}\in \mathcal{J}^{\mathrm{cyl}}_{\bR\times \Lambda_+}$ and $J_- \in \mathcal{J}^{\mathrm{cyl}}_{\bR\times \Lambda_-}$ as in Section \ref{subsubsec-CE-DGA} so that the Chekanov-Eliashberg DGAs
$
(\mathcal{A}_*(\Lambda_+),d_{J_+})$ and $ (\mathcal{A}_*(\Lambda_-),d_{J_-})$
are defined.

Let us take an almost complex structure $J\in \mathcal{J}_{L}$ such that $J=J_+$ on $[r_+,\infty)\times U^*\bR^3$ and $J=J_-$ on $(-\infty,r_-]\times U^*\bR^3$, where $r_-<r_+$.
In the way of \cite[Proposition 3.13]{DR},
by a generic perturbation of $J_+$, $\calM_{L,J}(a;a_1,\dots ,a_m)$ is cut out transversely for every $a\in \mathcal{R}(\Lambda_+)$ and $a_1,\dots ,a_m\in\mathcal{R}(\Lambda_-)$.
Moreover,
it is a manifold of dimension $|a|-\sum_{k=1}^m |a_k|$.
Again, it follows from Gromov-Floer compactness that $\calM_{L,J}(a;a_1,\dots ,a_m)$ is a compact $0$-dimensional manifold when $|a|-\sum_{k=1}^m |a_k|=0$.

We define a graded $\bZ_2$-algebra map 
$\Phi'_L\colon \mathcal{A}_*(\Lambda_+)\to \mathcal{A}_*(\Lambda_-)$ such that
\[ \Phi'_L(a)=\sum_{|a_1|+\dots +|a_m|=|a|} \left( \#_{\bZ_2} \calM_{L,J}(a;a_1,\dots ,a_m) \right) a_1\cdots a_m \]
for every $a\in \mathcal{R}(\Lambda_+)$.
In particular, note that $\#_{\Z_2}\calM_{L,J}(a,\emptyset)$ is the coefficient of $1\in \calA_0(\Lambda_-)$.
The next theorem \cite[Theorem 3.3]{DR} follows from \cite[Lemma 3.13 and 3.14]{EHK}.
\begin{theorem}
$\Phi'_L$ is a DGA map, that is, $d_{J_-}\circ \Phi'_L = \Phi'_L\circ d_{J_+}$.
Moreover, its homotopy class is invariant under compactly supported Hamiltonian isotopy of $L$.
\end{theorem}
It follows that $\Phi'_L$ induces a graded $\Z_2$-algerba map
\[  \Phi_L \colon \LCH_*(\Lambda_+)\to \LCH_*(\Lambda_-), \]
which is invariant under compactly supported Hamiltonian isotopy of $L$.

When $\Lambda=\Lambda_+=\Lambda_-$,  $L=\bR\times \Lambda$ and $J\in \mathcal{J}^{\mathrm{cyl}}_{\bR\times \Lambda}$, the moduli space $\calM_{\bR\times \Lambda,J}(a,a_1, \dots, a_m)$ is the empty set if $|a|=\sum_{k=1}^m|a_k|$ and $(a_1,\dots ,a_m)\neq (a)$, since $\dim \bar{\calM}_{ \Lambda,J}(a,a_1, \dots, a_m)<0$.
For any $a\in\mathcal{R}(\Lambda)$, the moduli space $\calM_{\bR\times \Lambda,J}(a,a)$ consists of a single element given by the trivial strip whose image is $\bR\times a([0,T_a])$.
Therefore, $\Phi_{\bR\times \Lambda}$ is the identity map on $\LCH_*(\Lambda)$.

\subsection{Isomorphism induced by an isotopy of Legendrian submanifolds}\label{subsec-isom-cobordism}

Let us discuss the map $\Phi_L$ when $L$ arises from an isotopy of Legendrian submanifolds of $U^*\R^3$.
For geometric arguments, we will refer to \cite[Section 2]{Cha}. (Although \cite{Cha} focuses on the case of $3$-dimensional contact manifolds, the arguments in \cite[Section 2]{Cha} can be generalized to arbitrary odd dimensions.)

For any time-dependent $C^{\infty}$ Hamiltonian on $\R\times U^*\R^3$
\[H\colon (\bR\times U^*\R^3) \times [0,1] \to \bR \colon ((r,x),t)\mapsto H_t(r,x),\]
a time-dependent vector field $(X^t_H)_{t\in [0,1]}$ is defined by $d(e^r\alpha)(\cdot , X^t_H) = dH_t$.
If it is an integrable vector field, we let $(\varphi^t_H)_{t\in [0,1]}$ denote the flow on $\bR\times U^*\R^3$ generated by $(X^t_H)_{t\in [0,1]}$.

Given exact Lagrangian cobordisms $L$ and $L'$ such that the positive end of $L$ (resp. $L'$) is $\Lambda_1$ (resp. $\Lambda_2$) and the negative end of $L$ (resp. $L'$) is $\Lambda_2$ (resp. $\Lambda_3$), we denote by $ L' \circ_{\Lambda_2} L$ an exact Lagrangian cobordism obtained by concatenating $L$ and $L'$ along $\Lambda_2$.
It has $\Lambda_1$ as the positive end and $\Lambda_3$ as the negative end.

\begin{lemma}\label{lem-isotopy-cobordism}
Let $(\Lambda_t)_{t\in [0,1]}$ be an isotopy of compact Legendrian submanifolds of $U^*\R^3$.
Then, there exist
smooth families of exact Lagrangian submanifolds $(L_t)_{t\in [0,1]}$ and $(L'_t)_{t\in [0,1]}$ of $\bR\times U^*\R^3$ satisfying:
\begin{itemize}
\item[(\#)] $L_0=L'_0= \bR\times \Lambda_0$ and there exist $r_-<r_+$ such that
\[\begin{array}{cc}
L_t\cap \left( [r_+,\infty)\times U^*\R^3\right) = [r_+,\infty) \times \Lambda_0, & L_t\cap \left( (-\infty,r_-]\times U^*\R^3\right) = (-\infty,r_-] \times \Lambda_t , \\
L'_t\cap \left( [r_+,\infty)\times U^*\R^3\right) = [r_+,\infty) \times \Lambda_t, & L'_t\cap \left( (-\infty,r_-]\times U^*\R^3\right) = (-\infty,r_-] \times \Lambda_0 .
\end{array}\]
\end{itemize}
Moreover, for any pair of $(L_t)_{t\in [0,1]}$ and $(L'_t)_{t\in [0,1]}$ satisfying (\#), the concatenation $L'_1\circ_{\Lambda_1}L_1$ (resp. $L_1\circ_{\Lambda_0}L'_1$) is isotopic to $\bR\times \Lambda_0$ (resp. $\bR\times \Lambda_1$) by a compactly supported Hamiltonian isotopy on $\bR\times U^*\R^3$.
\end{lemma}

\begin{proof}
By \cite[Proposition 2.2]{Cha} and its proof, there exists a time-dependent Hamiltonian $H$ on $\bR\times U^*\R^3$ such that  $\varphi^t_H(\bR\times \Lambda_0) = \bR\times \Lambda_t$ for every $t\in [0,1]$.
Take a $C^{\infty}$ function $\mu \colon \bR\to [0,1]$ and $r'_-<r'_+$ such that $\mu(r)=0$ if $r\geq r'_+$ and $\mu(r) =1$ if $r\leq r'_-$.
We define
\begin{align*}
G\colon & (\bR\times U^*\R^3)\times [0,1]\to \bR\colon ((r,x),t) \mapsto \mu(r)\cdot H_t(r,x) , \\
G'\colon & (\bR\times U^*\R^3)\times [0,1]\to \bR\colon ((r,x),t) \mapsto (1-\mu(r))\cdot H_t(r,x).
\end{align*}
Then, there exist $r_-<r'_-$ and $r_+>r'_+$ such that
\[\begin{cases}
\varphi^t_{G}(r,x)= (r,x) \text{ and }\varphi^t_{G'}(r,x)= \varphi^t_H(r,x) & \text{ if }r\geq r+ , \\
\varphi^t_G(r,x)=\varphi^t_H(r,x) \text{ and }\varphi^t_{G'}(r,x)=(r,x) & \text{ if } r\leq r_-.
\end{cases}\]
Therefore, if we define $L_t\coloneqq \varphi^t_G(\bR\times \Lambda_0)$ and $L'_t\coloneqq \varphi^t_{G'}(\bR\times \Lambda_0)$ for every $t\in [0,1]$,
 $(L_t)_{t\in [0,1]}$ and $(L'_t)_{t\in [0,1]}$ satisfy the condition (\#).

Next, suppose that we have a pair of $(L_t)_{t\in [0,1]}$ and $(L'_t)_{t\in [0,1]}$ satisfying (\#).
Consider for every $t\in [0,1]$
\[ \begin{array}{cc} S_t\coloneqq  L'_t \circ_{\Lambda_t} L_t , &  S'_t\coloneqq  \psi^t( L_t \circ_{\Lambda_0} L'_t ), \end{array}\]
where $\psi^t\coloneqq  \varphi^1_H \circ (\varphi^t_H)^{-1} $.
Then,
\[\begin{array}{cccc} S_0=\bR\times \Lambda_0, &  S_1 =  L'_1 \circ_{\Lambda_1} L_1, &  S'_0 = \bR\times \Lambda_1 , & S'_1=  L_1 \circ_{\Lambda_0} L'_1.
\end{array}\]
Moreover,
$(S_t)_{t\in [0,1]}$ (resp. $(S'_t)_{t\in [0,1]}$) is a smooth family of exact Lagrangian cobordisms whose positive and negative ends are fixed at $\Lambda_0$ (resp. $\Lambda_1$).
For such smooth families of exact Lagrangian submanifolds, there exist Hamiltonians $K,K'$ on $\bR\times U^*\R^3$ with compact support such that $\varphi^t_K(\bR\times \Lambda_0) = S_t$ and $\varphi^t_{K'}(\bR\times \Lambda_1) = S'_t$.
In particular, $\varphi^1_K(\R\times \Lambda_0)=L'_1\circ_{\Lambda_1}L_1$ and $\varphi^1_{K'}(\R\times \Lambda_1)=L_1\circ_{\Lambda_0}L'_1$.
\end{proof}

Moreover, using $1$-parameter families of Hamiltonians on $\bR\times U^*\R^3$, we can show the following lemma in a similar way as Lemma \ref{lem-isotopy-cobordism}.
\begin{lemma}\label{lem-isotopy-of-isotopy}
Let $(\Lambda^s_t)_{s,t\in [0,1]}$ be a family of Legendrian submanifolds of $U^*\R^3$ which varies smoothly on $(s,t)\in [0,1]^2$ such that $\Lambda^s_{0}=\Lambda_0$ and $\Lambda^s_{1}=\Lambda_1$ for every $s\in [0,1]$.
Then, there exist smooth families $(L^s_t)_{s,t\in [0,1]}$ and $(L'^s_t)_{s,t\in [0,1]}$ of exact Lagrangian submanifolds of $\bR\times U^*\R^3$ such that for any fixed $s\in [0,1]$, $(L^s_t)_{t\in [0,1]}$ and $(L'^s_t)_{t\in [0,1]}$ satisfy (\#) in Lemma \ref{lem-isotopy-cobordism} for $(\Lambda^s_t)_{t\in[0,1]}$.
\end{lemma}
From this lemma, we have a family $(L^s_1)_{s\in [0,1]}$ whose positive end is fixed at $\Lambda_0$ and  negative end is fixed at $\Lambda_1$.
Therefore, $L^0_1$ is isotopic to $L^1_1$ by a compactly supported Hamiltonian isotopy on $\bR\times U^*\R^3$. 

Now, we apply Lemma \ref{lem-isotopy-cobordism} and \ref{lem-isotopy-of-isotopy} to deduce the following theorem.
Let $\Lambda_0, \Lambda_1$ and $\Lambda_2$ be Legendrian submanifolds of $U^*\R^3$ as in Section \ref{subsubsec-CE-DGA}.
Recall the set $ \Pi_{\Leg}(\Lambda_i,\Lambda_j)$ of Definition \ref{def-Pi-Leg} for $i,j\in\{0,1,2\}$ and the concatenation map $\Pi_{\Leg}(\Lambda_1,\Lambda_2)\times \Pi_{\Leg}(\Lambda_0,\Lambda_1)\to \Pi_{\Leg}(\Lambda_0,\Lambda_2)\colon (x', x) \mapsto x'\cdot x$.

\begin{theorem}\label{thm-DGA-isom}
For any $y \in \Pi_{\Leg}(\Lambda_0,\Lambda_1)$ 
represented by a Legendrian isotopy $(\Lambda_t)_{t\in [0,1]}$ in $U^*\bR^3$, we take families of exact Lagrangian submanifolds $(L_t)_{t\in [0,1]}$ and $(L'_t)_{t\in [0,1]}$ as in Lemma \ref{lem-isotopy-cobordism}.
Then, $\Phi_{L_1} \colon \LCH_*(\Lambda_0) \to \LCH_\ast(\Lambda_1)$ is an isomorphism whose inverse map is given by $\Phi_{L'_1}$.
If we set
\[\Phi(y) \coloneqq \Phi_{L_1},\]
this gives a well-defined map
\[ \Phi \colon \Pi_{\Leg}(\Lambda_0,\Lambda_1) \to \Hom ( \LCH_*(\Lambda_0), \LCH_*(\Lambda_1)) . \]
Moreover, for any $y\in \Pi_{\Leg}(\Lambda_0,\Lambda_1)$ and $y'\in \Pi_{\Leg}(\Lambda_1,\Lambda_2)$, the equality $\Phi(y'\cdot y) = \Phi(y')\circ \Phi(y)$ holds in $\Hom(\LCH_*(\Lambda_0),\LCH_*(\Lambda_2))$.
\end{theorem}

\begin{proof}
Consider a Legendrian isotopy $(\Lambda_t)_{t\in [0,1]}$ in $U^*\bR^3$.
For any pair of $(L_t)_{t\in [0,1]}$ and $(L'_t)_{t\in [0,1]}$ satisfying (\#) as in Lemma \ref{lem-isotopy-cobordism},
we have homomorphisms of graded $\Z_2$-algebras
\[ \begin{array}{cc}
\Phi_{L_1}\colon \LCH_*(\Lambda_0) \to \LCH_*(\Lambda_1) , & \Phi_{L'_1} \colon \LCH_*(\Lambda_1) \to \LCH_*(\Lambda_0) . \end{array}\]
By \cite[Lemma 3.13]{EHK} about the composition of cobordisms and \cite[Lemma 3.14]{EHK} about the invariance under compactly supported Hamiltonian isotopy,
\begin{align*}
\Phi_{L'_1}\circ \Phi_{L_1} = \Phi_{L'_1\circ_{\Lambda_1}L_1} = \Phi_{\bR\times \Lambda_0} = \id_{\LCH_*(\Lambda_0)} , \\
\Phi_{L_1}\circ \Phi_{L'_1} = \Phi_{L_1\circ_{\Lambda_0}L'_1} = \Phi_{\bR\times \Lambda_1} = \id_{\LCH_*(\Lambda_1)}.
\end{align*}
Therefore, $\Phi_{L_1}$ is an isomorphism whose inverse map is $\Phi_{L'_1}$.
Moreover, this map is independent of the choices of $(L_t)_{t\in [0,1]}$ and $(L'_t)_{t\in [0,1]}$.
Indeed, if we choose another pair of $(\hat{L}_t)_{t\in [0,1]}$ and $(\hat{L}'_t)_{t\in [0,1]}$ satisfying (\#), then
\[ \Phi_{L'_1}\circ \Phi_{\hat{L}_1} = \Phi_{L'_1\circ_{\Lambda_1}\hat{L}_1} = \Phi_{\bR\times \Lambda_0} = \id_{\LCH_*(\Lambda_0)} \]
since $L'_1\circ_{\Lambda_1}\hat{L}_1$ is isotopic to $\bR\times \Lambda_0$ by a compactly supported Hamiltonian isotopy.
Therefore, $\Phi_{L_1}=(\Phi_{L'_1})^{-1} = \Phi_{\hat{L}_1}$.

Let us see that $\Phi_{L_1}$ depends only on the homotopy class of $(\Lambda_t)_{t\in [0,1]}$.
For any smooth family of Legendrian submanifolds $(\Lambda^s_t)_{s,t\in [0,1]}$ such that $\Lambda^s_{0}= \Lambda_0$ and $\Lambda^s_{1}=\Lambda_1$ for every $s\in [0,1]$, we take $(L^s_t)_{s,t\in[0,1]}$ and $(L'^s_t)_{s,t\in [0,1]}$ as in Lemma \ref{lem-isotopy-of-isotopy}.
Using $(L^s_1)_{s\in[0,1]}$, it follows from \cite[Lemma 3.14]{EHK} that $ \Phi_{L^0_1} = \Phi_{L^1_1} $.
Therefore, $\Phi(y)$ for $y\in \Pi_{\Leg}(\Lambda_0,\Lambda_1)$ is well-defined.

Lastly, for any $y\in \Pi_{\Leg}(\Lambda_0,\Lambda_1)$ and $y'\in \Pi_{\Leg}(\Lambda_1,\Lambda_2)$, the equation
$\Phi(y')\circ \Phi(y) = \Phi(y'\cdot y)$
follows from \cite[Lemma 3.13]{EHK}.
\end{proof}

\subsection{Legendrian contact homology of the unit conormal bundle of knots}\label{subsec-LCH-of-Lambda-K}

Let $K$ be a smoothly embedded knot in $\bR^3$.
Its conormal bundle 
\[L_K \coloneqq \{ (q, p)\in T^\ast \R^3 \mid q\in K,\  p(v)=0 \text{ for all }v\in T_qK \} \]
is an exact Lagrangian submanifold of $T^\ast \mathbb{R}^3$ whose Maslov class vanishes.
Then, the unit conormal bundle $\Lambda_K=L_K\cap U^\ast\mathbb{R}^3$ is a Legendrian submanifold of $U^\ast\mathbb{R}^3$. Note that $\Lambda_K$ is diffeomorphic to the $2$-torus $S^1\times S^1$.

Let us recollect several facts about Reeb chords of $\Lambda_K$.
Let $\mathcal{C}(K)$ be the set of binormal chords of $K$ with unit speed, that is,
\[ \mathcal{C}(K) \coloneqq \left\{ c\colon [0,T]\to \R^3 \ \middle| \begin{array}{l}  |\dot{c}(t)|=1 \text{ for any }t\in [0,T] \text{ and } \\
c(t)\in K,\ \dot{c}(t)\in (T_{c(t)}K)^{\perp} \text{ for }t\in \{0,T\} \end{array}\right\} .\]
Then we have a bijection
\[ \mathcal{R}(\Lambda_K)\to \mathcal{C}(K) \colon a \mapsto \pi_{\R^3}\circ a, \]
where $\pi_{\R^3}\colon U^*\R^3\to \R^3$ is the bundle projection.
In addition, if we set a $C^{\infty}$ function
\[ f_K\colon (K\times K)\setminus \Delta_K \to \R \colon (q,q')\mapsto |q-q'|^2, \]
where $\Delta_K$ is the diagonal subset, then we also have a bijection
\[ \mathcal{C}(K) \to \{ \text{critical points of }f_K \}\colon (c \colon [0,T]\to \R^3) \mapsto (c(0),c(T)). \]
Therefore, we have a bijection
\[ \mathcal{R}(\Lambda_K) \to \{ \text{critical points of }f_K \} \]
which maps any $a\in \mathcal{R}(\Lambda_K)$ with $a(0)=(q_a,p_a)\in \Lambda_K$ to $x_a\coloneqq (q_a, q_a+T_a\cdot p_a) \in (K\times K)\setminus \Delta_K$.
Moreover, the following relation holds.
See \cite[Proposition 2.7]{O24}, which relies on \cite[Corollary 4.2]{APS}.
\begin{proposition}
A Reeb chord $a\in \mathcal{R}(\Lambda_K)$ is non-degenerate if and only if $x_a$ is a non-degenerate critical point of $f_K$. 
Moreover,
\[\mu(a)= \mathrm{ind} (x_a)+1,\]
where $\mathrm{ind} (x_a)$ is the Morse index of $f_K$ at $x_a$, hence $|a|= \mathrm{ind}(x_a)$.
In particular, $|a| \in \{0,1,2\}$.
\end{proposition}

For generic $K$, $f_K$ is a Morse function, and this means that every Reeb chord of $\Lambda_K$ is non-degenerate.
Then, we can define the DGA $(\mathcal{A}_*(\Lambda_K),d_J)$
and its homology $\LCH_*(\Lambda_K)$ as in Section \ref{subsubsec-CE-DGA}.
We note that since $\mathcal{A}_p(\Lambda_K)=0$ for $p<0$,
\[  \LCH_0(\Lambda_K) = \mathcal{A}_0(\Lambda_K)/ d_J(\mathcal{A}_1(\Lambda_K)) . \]

\section{Parametrized moduli of pseudo-holomorphic curves in $T^*\R^3$}\label{sec-moduli}

In this section, referring to \cite{CEL, CELN}, we introduce moduli spaces of pseudo-holomorphic curves in $T^*\bR^3$ with switching Lagrangian boundary conditions.
More generally, we will consider cases where the boundary condition is parametrized by a manifold.

\subsection{Definition of the moduli space $\calM_{\mathscr{L},l}(a)$}\label{subsec-def-moduli}

As a notation, for every $l\in \bZ_{\geq 0}$ and $j\in \{1,\dots ,2l+1\}$, let $\partial_j D_{2l+1}$ denote the component of $\partial D_{2l+1}$ whose closure has $p_{j-1}, p_{j}\in \partial D$ as boundary points (when $j=2l+1$, $p_{2l+1}\coloneqq p_0$). We denote the closure by
$ \overline{\partial_j D_{2l+1}}$, which is homeomorphic to a closed interval.
In addition. let us denote $D^*_{\epsilon}\R^3\coloneqq \{(q,p)\in T^*\R^3 \mid |p|\leq \epsilon \}$ for any $\epsilon>0$.

Recall from Section \ref{subsubsec-symplectization-moduli} the symplectomorphism $E$ defined by (\ref{diffeo-symp-cot}) and the almost complex structure $J_0$ on $\bR\times U^*\bR^3$.
Choose a $C^{\infty}$ function $\rho \colon [0,\infty)\to \bR_{>0}$ such that $\rho(s)=1$ if $s\leq \frac{1}{2}$ and $\rho(s)=s$ if $s\geq 1$.
Then, we define an almost complex structure $J_{\rho}$ on $T^*\bR^3$ such that for every $(q,p)\in T^*\bR^3$,
\[ (J_{\rho})_{(q,p)} \colon T_{(q,p)} (T^*\bR^3) \to T_{(q,p)}(T^*\bR^3) \colon \partial_{q_i} \mapsto - \rho(|p|) \partial_{p_i},\  \partial_{p_i} \mapsto \rho(|p|)^{-1} \partial_{q_i} , \]
for $i=1,2,3$.
Note that $E^*J_{\rho}$ agrees with $J_0$ on $\bR_{\geq 0}\times U^*\bR^3$.

Fix a compact connected Legendrian submanifold $\Lambda$ of $U^*\R^3$.
Let $B$ be a compact $C^{\infty}$ manifold possibly with boundary.
Let $\mathscr{L}=(L_b)_{b\in B}$ be a smooth family of exact Lagrangian submanifolds of $T^*\bR^3$ parametrized by $B$ such that:
\begin{itemize}
\item[(L1)] The Maslov class of $L_b$ vanishes.
\item[(L2)] There exists a real number $r_0>0$ such that
\[  E^{-1}(L_b) \cap \left( [r_0,\infty) \times U^*\bR^3 \right) = \Lambda\times [r_0,\infty) \]
for every $b\in B$ 
and $\left( \bigcup_{b\in B} L_b\right) \cap D^*_{e^{r_0}}\R^3$ is a compact subset of $T^*\R^3$.
\item[(L3)] There exist $\epsilon\in (0,\frac{1}{2}]$ and a smooth family of embeddings $(\gamma_b\colon S^1\to \R^3)_{b\in B}$ such that for each $b\in B$, $\gamma_b$ is a real analytic map and
\[ L_b \cap D^*_{\epsilon}\bR^3 = L_{K_b} \cap D^*_{\epsilon}\bR^3 \]
for the knot $K_b\coloneqq \gamma_b(S^1)\subset \R^3$.
\end{itemize}

For such a family $\mathscr{L}$,
referring to \cite[Section 8.1]{CELN},
let us take an almost complex structure $J'_{\rho}$ on $T^*\bR^3$ satisfying:
\begin{itemize}
\item $d\lambda_{\R^3}(\cdot , J'_{\rho}\cdot )$ is a Riemannian metric on $T^*\bR^3$.
\item There exists $d>0$ such that $\bigcup_{b\in B} L_b \subset \{(q,p) \in T^*\bR^3 \mid |q| < d \}$ and  $J'_{\rho}$ coincides with $J_{\rho}$ on
$  \{(q,p) \in U^*\bR^3 \mid |q| > d \text{ or }|p| \leq \frac{1}{2} \} $.
\item 
$E^*J'_{\rho}$ preserves $\xi$ and $(E^*J'_{\rho})(\partial_r) = \frac{e^r}{\rho(e^r)}R_{\alpha}$.
\item There exists $J\in \mathcal{J}^{\mathrm{cyl}}_{\bR\times \Lambda}$ such that $E^*J'_{\rho}= J$ on $\bR_{\geq 0}\times U^*\bR^3$,
\end{itemize}
Note that $J_{\rho}$ satisfies these conditions with $J=J_0$.

For $a\in \mathcal{R}(\Lambda)$ and $l\in \Z_{\geq 0}$, we define $\calM_{\mathscr{L},l}(a)$ as the moduli space (modulo conformal automorphisms) consisting of triples $(b,u,\kappa)$ of $b\in B$, $\kappa \in \mathcal{C}_{2l+1}$
and a $C^{\infty}$ map $u\colon D_{2l+1} \to T^*\bR^3$ such that:
\begin{itemize}
\item $du+ J'_{\rho} \circ du \circ j_{\kappa}=0$. 
\item $u(\partial_{2i}D_{2l+1})\subset \bR^3$ for $i\in \{1,\dots ,l\}$ and  $u(\partial_{2i+1} D_{2l+1}) \subset L_b$ for $i\in \{0,\dots ,l\}$.
\item  $(u\circ \psi_0)(s,t) $ is contained in $ T^*\bR^3 \setminus \bR^3$ when $s\gg 0 $ and there exists $s_0\in\bR$ such that
\[ \lim_{s\to \infty} (\tau_{-T_a s} \circ  E^{-1}\circ u\circ \psi_0)(s,t) = (s_0,a(T_a t)) \]
in $\R\times U^*\R^3$ $C^{\infty}$ uniformly on $t\in [0,1]$
\item There exist $x_1,\dots ,x_{2l}\in K_b = L_b\cap \R^3$ such that
\[\lim_{s\to -\infty} (u\circ \psi_j)(s,t)= x_j\]
in $T^*\R^3$ $C^{\infty}$ uniformly on $[0,1]$ for every $j\in \{1,\dots ,2l\}$.
\end{itemize}
By setting  $u(p_i)=x_i$ for $j\in \{1,\dots ,2l\}$, $u$ is extended continuously on $D_1$.
See Figure \ref{fig-switch}.
\begin{remark}
When $u$ is extended as a map from $(D_1,\partial D_1) $ to $ (T^*\R^3, L_b\cup \R^3)$, the points $p_1,\dots ,p_{2l+1}\in \partial D_1$ are regarded as points where the boundary condition switches (i.e. jumps from $L_b$ to $\R^3$ or vice versa) rather than as punctures. See also \cite[Section 6.3]{CELN}.
\end{remark}

\begin{figure} \centering
 \begin{overpic}[height=4.5cm, width=8.5cm]{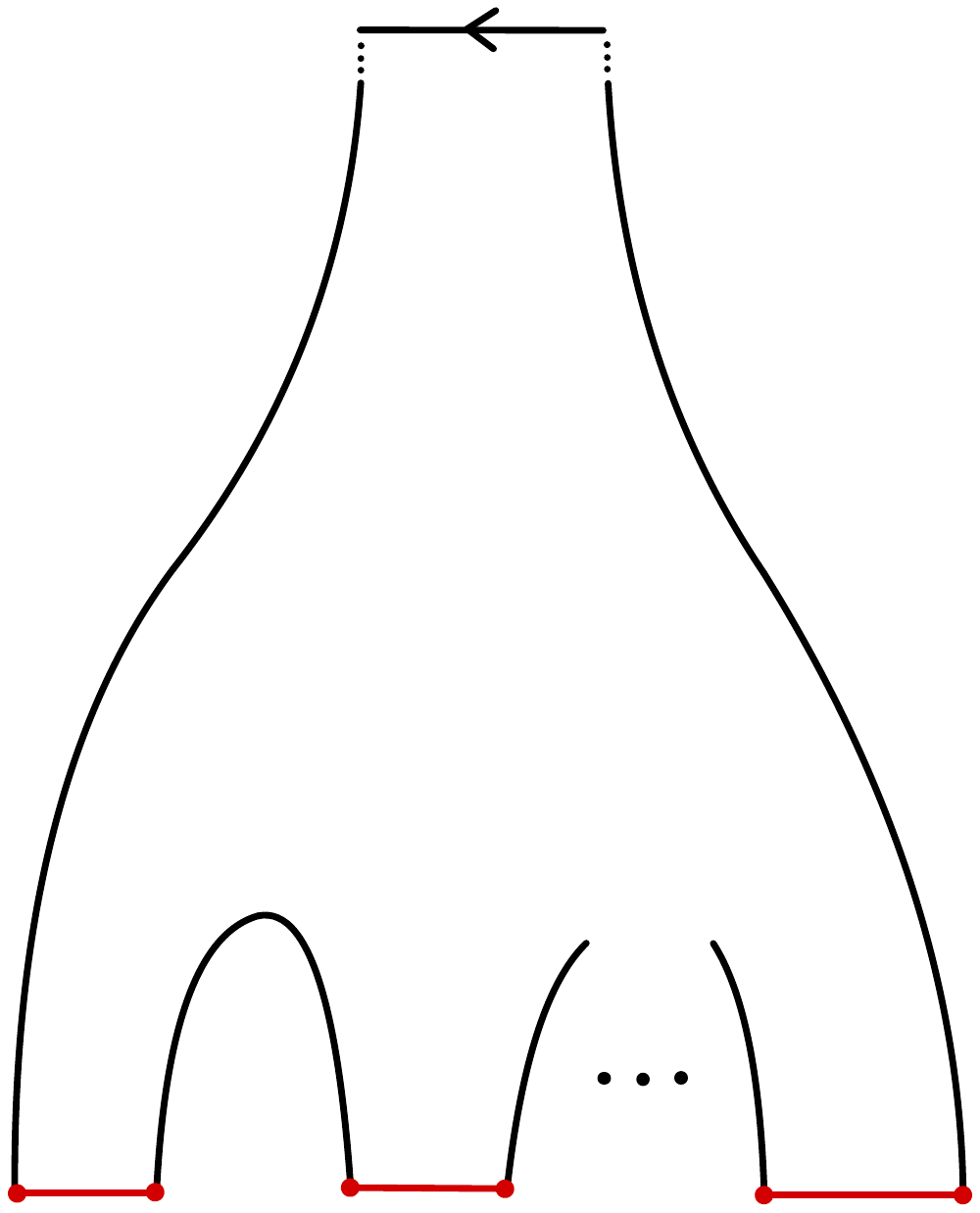}
 \put(45,54){$a$}
 \put(12,-3){$u(p_1)$}
 \put(26,-3){$u(p_2)$}
 \put(54,-3){$u(p_{2l-1})$} \put(75,-3){$u(p_{2l})$}
 \put(42,-3){\color{red}$\bR^3$}
 \put(30,36){$L_b$}
\end{overpic}
\caption{The image of $u$ in $T^*\R^3$. $a$ is a Reeb chord of $\Lambda$. The black boundaries are contained $L_b$ and the red boundaries are contained in $\bR^3$. The switching points $p_1,\dots ,p_{2l}$ are mapped to $ L_b\cap \bR^3 =K_b$.}\label{fig-switch}
\end{figure}

In addition, for $j\in\{1,\dots ,2l+1\}$, we define $\widetilde{\calM}^j_{\mathscr{L},l}(a)$ as the moduli space (modulo conformal automorphisms on $D_{2l+1}$) consisting of $((b,u,\kappa), z)$ such that $(b,u,\kappa)$ satisfies the same conditions as above and $z$ lies in $\partial_j D_{2l+1}$.
It is a moduli space of $J'_{\rho}$-holomorphic curves with a marked point on the $j$-th boundary component.

We have two submanifolds of $B\times T^*\bR^3$: $B\times \bR^3$ and
\[  L_{\mathscr{L}} \coloneqq \{ (b,x)\in B\times T^*\bR^3 \mid x\in L_b \}. \]
Then, we can define evaluation maps
\begin{align}\label{ev-j-map}
\ev_{j} \colon \widetilde{\calM}^j_{\mathscr{L},l} (a) \to \begin{cases} B\times \bR^3  & \text{ if }j\text{ is even,} \\
L_{\mathscr{L}}  & \text{ if }j \text{ is odd,} \end{cases}
\end{align}
which maps $((b,u,\kappa),z)$ to $(b,u(z)) $.
We also have a submanifold of $B\times \bR^3$
\[ K_{\mathscr{L}} \coloneqq \{ (b,q)\in B\times \bR^3 \mid q\in K_b \}. \]
Note that $B\times \bR^3$ and $L_{\mathscr{L}}$ intersect cleanly along $K_{\mathscr{L}}$.

\subsection{Winding numbers at switching points}\label{subsec-winding}

For $\delta>0$, we set
\[D_{\delta} \coloneqq \{ (z_1,z_2)\in \bC\times \bC \mid |z_1|^2+|z_2|^2<\delta \}.\]
For $S^1 = \{z\in \bC \mid |z|=1\}$,
we fix a complex structure on $S^1\times (-\delta,\delta)$ such that the map $\{z\in \bC \mid -\delta < \Image z  <\delta\} \to S^1\times (-\delta,\delta) \colon z\mapsto (e^{2\pi \sqrt{-1} \mathrm{Re}z}, \Image z)$ is holomorphic.

We continues to consider a family $\mathscr{L}$ of exact Lagrangian submanifolds of $T^*\bR^3$ in Section \ref{subsec-def-moduli}.
From the condition (L3),
there is a family of real analytic embeddings $(\gamma_b\colon S^1\to \R^3)_{b\in B}$ such that $K_b=\gamma_b(S^1)$.
Take real analytic vector fields along $K_b$
\[v_{1,b},v_{2,b}\colon K_b\to \bR^3\]
which vary smoothly on $b\in B$ such that $\{ \dot{\gamma}_b(t), (v_{1,b})_{\gamma_b(t)},(v_{2,b})_{\gamma_b(t)} \}$ is an orthonormal basis of $\bR^3$ for every $t\in S^1$.
We extend the proof of \cite[Lemma 8.6]{CELN} about a single real analytic embedding $\gamma\colon S^1\to \bR^3$.
Then, for sufficiently small $\delta>0$, we obtain a family of holomorphic embedding
\[ \textstyle{ \iota_b\colon S^1 \times (-\delta,\delta) \times D_{\delta} \to D^*_{\frac{1}{2}}}\R^3\]
which varies smoothly on $b\in B$ such that:
\begin{itemize}
\item $\iota_b(S^1\times \{0\}\times \{0\})=K_b$.
\item $\iota_b(S^1\times \{0\}\times (\bR^2\cap D_{\delta})) \subset \bR^3$.
\item $\iota_b(S^1\times \{0\}\times (\sqrt{-1} \bR^2\cap D_{\delta})) \subset L_{K_b}$.
\end{itemize}

Let $\pi_{D_{\delta}} \colon S^1\times (-\delta,\delta)\times D_{\delta} \to D_{\delta}$ denote the projection.
For any $(b,u,\kappa)\in \calM_{\mathscr{L},l}(a)$ and $j\in \{1,\dots ,2l\}$, we set \[u^{\mathrm{loc}}_j\coloneqq \pi_{D_{\delta}}\circ \iota_b^{-1}\circ u \circ \psi_j.\]
This is a holomorphic function from $(-\infty,s'_j]\times [0,1]$ to $D_{\delta}\subset \bC\times \bC$, where $s'_j\ll 0$.
Take a biholomorphic map
\[ \mu\colon (-\infty,0]\times [0,1]\to \{z\in \bC\setminus \{0\} \mid |z|\leq 1,\ \mathrm{Re}(z) \geq 0,\  \mathrm{Im} (z) \geq 0 \} \colon (s,t) \mapsto e^{\frac{\pi}{2}(s+\sqrt{-1}t)} , \]
and consider a power series expansion of $u^{\mathrm{loc}}_j\circ  \mu^{-1}$ centered at $0$.
From the arguments in \cite[Section 4.1, Case 1 and 2]{CELN}, it has the form
\[u^{\mathrm{loc}}_j\circ  \mu^{-1}(z) = \sum_{k = 0}^{\infty} a_k z^{2k+1}, \text{ where }a_k\in \begin{cases} \bR^2 & \text{ if }j \text{ is odd,} \\ \sqrt{-1}\bR^2 & \text{ if }j \text{ is even}. \end{cases} \]
The \textit{winding number} of $u$ at $p_j$ is defined by
\[ w_{b,j}(u,\kappa)\coloneqq \min \{ \textstyle{\frac{2k+1}{2}} \mid k\in \Z_{\geq 0},\  a_k\neq 0 \}. \]

Then, 
for any tuple $\bold{n}=(n_1,\dots ,n_{2l})$ such that $n_j\in \{\frac{2k+1}{2} \mid k\in \bZ_{\geq 0}\}$, a subspace of $\calM_{\mathscr{L},l}(a) $ is defined by
\[\calM_{\mathscr{L},l}(a;\bold{n})\coloneqq \{ (b,u,\kappa) \in \calM_{\mathscr{L},l}(a) \mid w_{b,j}(u,\kappa) = n_j \text{ for each }j=1,\dots ,2l \}.\]
Note that $\calM_{\mathscr{L},l}(a;\bold{n})$ for $\bold{n}=(\frac{1}{2},\dots ,\frac{1}{2})$ is an open subset of $\calM_{\mathscr{L},l}(a)$.

\subsection{Perturbation of Lagrangian submanifolds}\label{subsec-perturb-Lag}

In this section, we discuss a perturbation of $(L_b)_{b\in B}$ which is used to achieve an additional condition (L4) in Section \ref{subsec-property-of-moduli}. See also Remark \ref{rem-perturb-L}.

For any $a\in\mathcal{R}(\Lambda)$ with $a(0) = (p_a,q_a) \in \Lambda$, let us consider a plane in $T^*\R^3$
\[ H_a\coloneqq \{ (q_a+ t \cdot p_a, r\cdot p_a ) \in T^*\R^3 \mid t\in \R,\ r\in \R \}. \]
Note that
$H_a\cap U^*\R^3 =  \{ (q_a+ t \cdot p_a,  p_a ) \in U^*\R^3 \mid t\in \R \} $ 
and it is the image of the Reeb orbit containing $a$.

Assume that every Reeb chord $a\in \mathcal{R}(\Lambda)$ is non-degenerate and moreover,
\begin{align}\label{intersect-with-orbit} 
\Lambda\cap (H_a\cap U^*\R^3) = \{a(0),a(T_a)\}.
\end{align}
Using the notation in Remark \ref{rem-1-jet}, this is equivalent to that the self-intersection of the Lagrangian immersion $\rest{\pi_{T^*S^2}}{\Lambda} \colon \Lambda \to T^*S^2$ consists only of transverse double points.

We choose a small isotopy of Legendrian submanifolds $(\Lambda_s)_{s\in [0,1]}$ such that $\Lambda_0=\Lambda$ and
\begin{align}\label{non-intersect-with-orbit} 
s\neq 0\Rightarrow \Lambda_s\cap (H_a\cap U^*\R^3) = \emptyset \text{ for any }a\in\mathcal{R}(\Lambda).
\end{align}
This is equivalent to that $\pi_{T^*S^2}(a(0))\notin \pi_{T^*S^2}(\Lambda_s)$ for any $a\in \mathcal{R}(\Lambda)$ if $s\neq 0$. 
Then, $L'_s\circ_{\Lambda_s} L_s\subset \R\times U^*\R^3$ as in Lemma \ref{lem-isotopy-cobordism} and its shift in the $\R$-direction give a family of exact Lagrangian cobordisms
$(M_s)_{s\in [0,1]}$ such that:
\begin{itemize}
\item $M_0 = \R\times \Lambda$ and there exists $r_1>0$ such that
\[ M_s \cap   \left( (\R\setminus [0,2r_1]) \times U^*\R^3\right) = (\R\setminus [0,2r_1]) \times  \Lambda \]
for any $s\in [0,1]$.
\item $M_s \cap \left( \{r_1\} \times U^*\R^3\right) = \{r_1\}\times \Lambda_s$.
\end{itemize}

Take a family of exact Lagrangian submanifolds $\mathscr{L}=(L_b)_{b\in B}$ satisfying (L1), (L2) and (L3). Let $r_0> 0$ be the real number in (L2).
\begin{figure} \centering
\begin{overpic}[height=5cm]{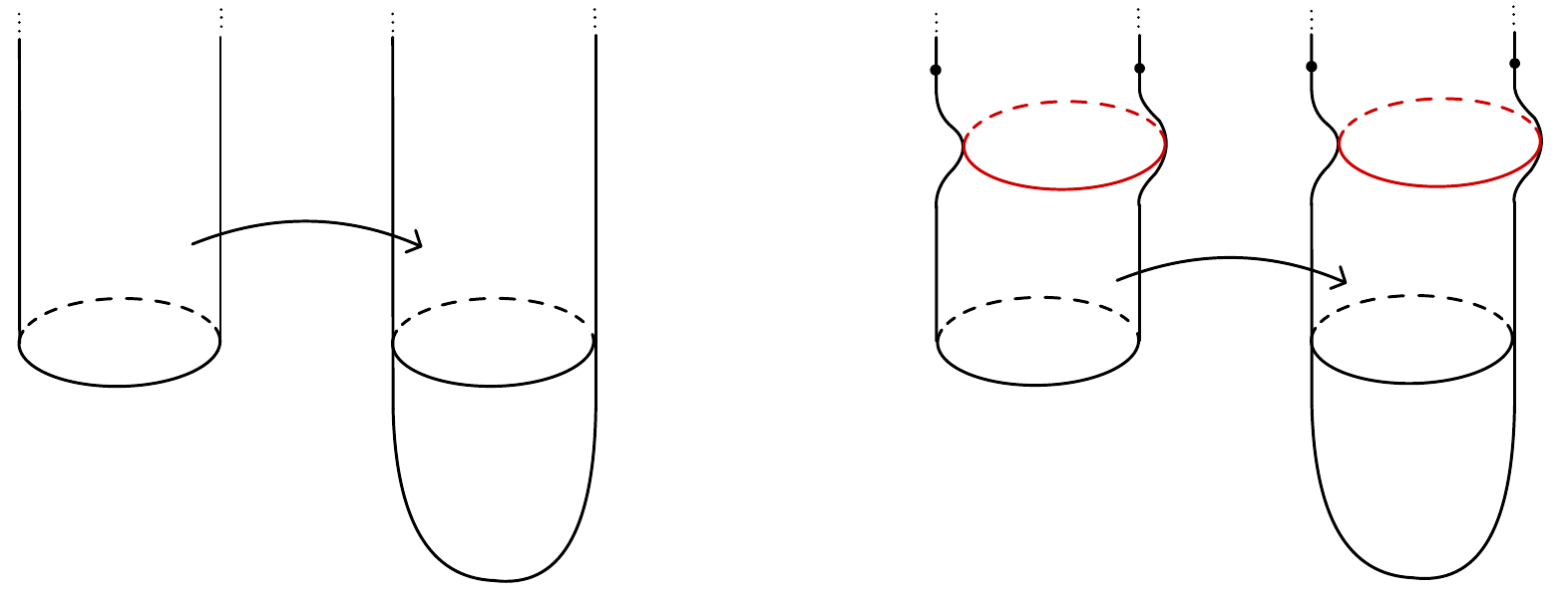}
\put(15.5,25){$E\circ \tau_{r'}$}
\put(0,11){$[0,\infty)\times \Lambda$}
\put(30,-2){$L_b$}
\put(74,22.5){$E\circ \tau_{r'}$}
\put(52,11){$M_s\cap ( [0,\infty)\times U^*\R^3)$}
\put(88.5,-2){$L^s_b$}
\put(48.5,28){\color{red}$\{r_1\}\times \Lambda_s$}
\put(55.5,33.5){$P_a^0$}
\put(73,33.5){$P_a^1$}
\end{overpic}
\caption{$L_b$ contains a cylindrical region identified with $[0,\infty)\times \Lambda$ via $E\circ \tau_{r'}$. $L^s_b$ is an exact Lagrangian filling of $\Lambda$ obtained by replacing the cylindrical region  with $M_s\cap ([0,\infty)\times U^*\R^3)$. In the right picture, $P^0_a\coloneqq (2r_1,a(0))$ and $P^1_a\coloneqq (2r_1,a(T_a))$.}\label{fig-perturb}
\end{figure}
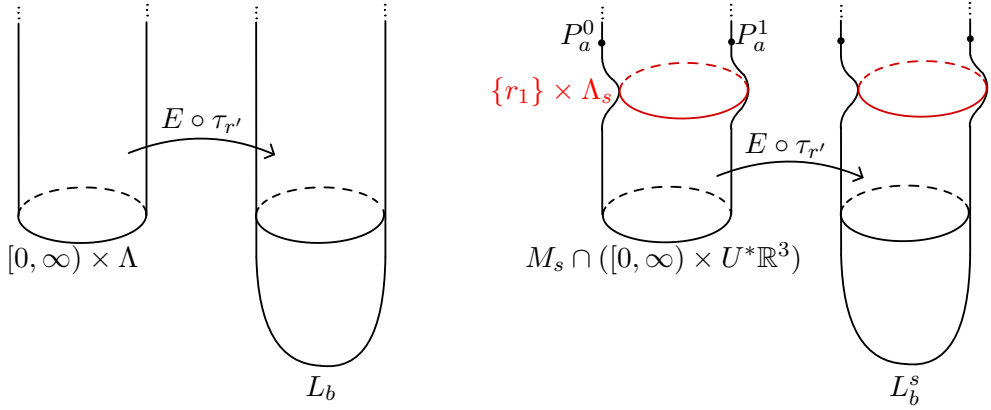
\begin{proposition}\label{prop-perturb-Lag}
Fix $b\in B$ and $r'>r_0$ arbitrarily.
Define an exact Lagrangian submanifold $L^s_b$ of $T^*\R^3$ by
\[ L^s_b \coloneqq  \left( L_b  \setminus E ( [r',\infty )\times \Lambda ) \right) \cup (E \circ \tau_{r'}) \left( M_s\cap (  [0,\infty) \times U^*\R^3 \ ) \right) . \]
See Figure \ref{fig-perturb}.
Then, there exists $s_0>0$ independent of $b$ and $r'$ such that for any $0<s\leq s_0$, the two points
\[ E(2r_1+r',a(0)), E(2r_1+r',a(T_a)) \in L^s_b\cap H_a\]
cannot be connected by a path in $(\R^3 \cup L^s_b)\cap H_a$. (In Figure \ref{fig-perturb}, these two points are $E\circ \tau_{r'}(P_a^0)$ and $E\circ \tau_{r'}(P_a^1)$.)
\end{proposition}
\begin{proof}
From the properties of $M_s$, there exists $s_0>0$ such that for any $s\in [0,s_0]$ and $r\in \R$, $M_s$ and $\{r\}\times U^*\R^3$ intersect transversely and $M^r_s \coloneqq M_s\cap (\{r\}\times U^*\R^3)$ is a $2$-dimensional submanifold of $U^*\R^3\cong \{r\}\times U^*\R^3$.
Then, $M^r_0=\Lambda$ for any $r\in\R$, $M^r_s = \Lambda$ if $r\notin [0,2r_1]$ and $M^{r_1}_s = \Lambda_s$.
By taking $s_0$ smaller if necessary, we may assume that $M^r_s$ is contained in a tubular neighborhood of $\Lambda$.
In particular, by (\ref{intersect-with-orbit}),
there are small neighborhoods $U_0$ of $a(0)$ and $U_1$ of $a(T_a)$ in $U^*\R^3$ such that $U_0\cap U_1=\emptyset$ and
\[ M^r_s \cap (H_a\cap U^*\R^3) \subset U_0 \sqcup U_1 \]
for any $s\in [0, s_0]$ and $r\in \R$.

We argue by contradiction.
For $s\in (0,s_0]$, suppose that there exists a path $\gamma \colon [0,1]\to (\R^3\cup L^s_b)\cap H_a$ such that $\gamma(0) =E(2r_1+r', a(0))$ and $\gamma(1) =E(2r_1+r', a(T_a))$.
Then, we must have
\[ \gamma([0,1]) \not\subset E ( [r_1+r',\infty)\times U^*\R^3 ). \]
This is because $\gamma([0,1]) \cap E ( [r_1+r',\infty)\times U^*\R^3 ) \subset (L^s_b\cap H_a) \cap E ( [r_1+r',\infty)\times U^*\R^3 )$, and via the diffeomorphism $E$,
\[  E^{-1}(L^s_b \cap H_a) \cap \left( [r_1+r',\infty)\times U^*\R^3 \right)   \subset [r_1+r',\infty) \times (U_0 \sqcup U_1). \]
Here, the right-hand side is disconnected,
and $E^{-1} (\gamma(i)) $ lies in $[r_1+ r',\infty)\times U_i$ for $i\in \{0,1\}$.
Therefore, the path $\gamma$ intersects the boundary of $E ( [r_1+r',\infty)\times U^*\R^3 ) $, and thus there exists $0<u<1$ such
\[ \gamma(u) \in E ( \{ r_1+r'\} \times U^*\R^3 ) \cap \left( (\R^3\cup L^s_b)\cap H_a \right) . \]
However, $E ( \{ r_1+r'\} \times U^*\R^3 ) \cap (\R^3\cup L^s_b) = E( \{r_1+ r'\}\times \Lambda_s)$ and it is disjoint from $H_a$ from the condition (\ref{non-intersect-with-orbit}) on $(\Lambda_s)_{s\in (0,1]}$, so we get a contradiction. 
\end{proof}

\subsection{Properties of low dimensional moduli spaces}\label{subsec-property-of-moduli}

Let $K$ be a knot in $\bR^3$ such that:
\begin{itemize} 
\item[(\$)] Every Reeb chord of $\Lambda_K$ is non-degenerate as in Section \ref{subsec-LCH-of-Lambda-K}.
Moreover,
for any $a\in \mathcal{R}(\Lambda_K)$ such that $a(0)=(q_a,p_a)\in \Lambda_K$,
\[\{ q_a + t\cdot p_a \in \R^3 \mid t\in \R \}\cap K = \{\pi_{\R^3}(a(0)), \pi_{\R^3}(a(T_a))\},\]
that is, the straight line in $\R^3$ containing the binormal cord $\pi_{\R^3}\circ a\colon [0,T_a]\to \R^3$ of $K$ can intersect $K$ only at the endpoints of $\pi_{\R^3}\circ a$.
In particular, (\ref{intersect-with-orbit}) holds for $\Lambda=\Lambda_K$.
\end{itemize}
Note that in the space $\mathrm{Emb}(S^1,\R^3)= \{f\colon S^1\to \R^3\mid \text{embedding}\}$ with $C^{\infty}$ topology, the subset $\{f\in \mathrm{Emb}(S^1,\R^3) \mid f(S^1)\text{ is a knot satisfying (\$)}\}$ is open and dense. For the proof of the density, see \cite[Lemma 3.6]{O24}.
Using the Fourier expansion of $f\in \mathrm{Emb}(S^1,\R^3)$ on $S^1=\{z\in \bC\mid |z|=1\}$, we can show that the subset
\[\{f \in \mathrm{Emb}(S^1,\R^3)\mid f(S^1) \text{ is a real analytic knot satisfying (\$)}\}\]
is dense in $\mathrm{Emb}(S^1,\R^3)$ with respect to the $C^{\infty}$ topology.

Let us fix $\Lambda=\Lambda_K$ for a real analytic knot $K$ satisfying (\$).
Let us also take $B=\{0\}$ or $B=[0,b_*]$ for some $b_*>0$.
Consider a family $\mathscr{L}=(L_b)_{b\in B}$ satisfying (L1), (L2), (L3) and 
\begin{itemize}
\item[(L4)]
$L_0 = L_K$.
When $B=[0,b_*]$ and $b>0$, for any $a\in\mathcal{R}(\Lambda_K)$, the two points
\[ E (r_0,a(0)) , E(r_0,a(T_a)) \in L_b\cap H_a\]
cannot be connected by a path in $(\R^3 \cup L_b)\cap H_a$. Here, $r_0>0$ is the real number of (L2).
\end{itemize}
\begin{remark}\label{rem-perturb-L}
Suppose we have $\mathscr{L} = (L_b)_{b\in [0,b_*]}$ satisfying (L1),(L2), (L3) and $L_0=L_K$.
Since (\ref{intersect-with-orbit}) holds for every $a\in \mathcal{R}(\Lambda_K)$, we can take $(L^s_b)_{b\in B,\ s\in [0,s_0]}$ as in Proposition \ref{prop-perturb-Lag}.
If we choose a function $[0,b_*]\to [0,s_0]\colon b\mapsto s(b)$ such that $s^{-1}(0)=\{0\}$, then the new family $\left( L^{s(b)}_b\right)_{b\in [0,b_*]}$ satisfies (L4).
\end{remark}

For any tuple $\bold{n}=(n_1,\dots ,n_{2l})$ such that $n_j\in \{\frac{2k+1}{2} \mid k\in \bZ_{\geq 0}\}$, we define
\[ |\bold{n}| \coloneqq \textstyle{ \sum_{j=1}^{2l} (2n_j-1). }  \]
Then, the following holds for the moduli space $\calM_{\mathscr{L},l}(a;\bold{n})$.
When $B=\{0\}$, the statement agrees with \cite[Theorem 6.7]{CELN}.

\begin{lemma}\label{lem-transv}
For generic $J'_{\rho}$, $\calM_{\mathscr{L},l}(a;\bold{n})$ is transversely cut out for every $a\in \mathcal{R}(\Lambda_K)$ and $\bold{n} \in \{ \frac{2k+1}{2} \mid k\in \bZ_{\geq 0}\}^{\times l}$, and it is a $C^{\infty}$ manifold of dimension $|a|-|\bold{n}| +\dim B$.
\end{lemma}

\begin{proof}
Let us sketch how to modify the proof of \cite[Theorem 6.7]{CELN}.
Referring to \cite[Section 9.2]{CELN}, we construct a Banach manifold $\mathscr{W}=\mathscr{W}_{\mathscr{L},l}(a;\bold{n})$ which is fibered over $B$ such that each fiber $\mathscr{W}_b$ of $b\in B$ is a Banach manifold as in \cite[Lemma 9.1]{CELN} with the conormal bundle $L_{K}$ replaced by the Lagrangian submanifold $L_b$.
(More precisely, to fix a weighted Sobolev norm, let us choose the exponential weight $\delta_0>0$ at the positive puncture to be small so that the parameter $n_0\in \Z_{\geq 0}$ in \cite[Lemma 9.2]{CELN} is equal to $0$.)
In addition, the Cauchy-Riemann operator with respect to $J'_{\rho}$ defines a section $\overline{\partial}_{J'_{\rho}}\colon \mathscr{W}\to \mathscr{E}$ of a Banach bundle $\mathscr{E}\to \mathscr{W}$ such that its zero locus is $\calM_{\mathscr{L},l}(a;\bold{n})$.
When $l=0$, we stabilize $J'_{\rho}$-holomorphic curves as in \cite[Section 9.4]{CELN} by adding marked points to the domain $D_1$.
The Fredholm index of the linearlization of $\overline{\partial}_{J'_{\rho}}$ is $|a|-|\bold{n}| + \dim B$ by \cite[Theorem A.1]{CEL}.

The transversality of the section $\overline{\partial}_{J'_{\rho}}$ with respect to the zero section of $\mathscr{E}$ is achieved by a generic perturbation of $\rest{J}{\xi}$ near the Reeb chord $a$ \cite[Lemma 9.5]{CELN}, except at $(b,u,\kappa)\in \calM_{\mathscr{L},l}(a;\bold{n})$ such that the map 
\[ E^{-1}\circ u\colon  \{z\in D_{2l+1} \mid  u(z)\notin \R^3\} \to \R\times U^*\R^3 \]
is everywhere tangent to the plane $\mathrm{Span}(\partial_r, R_{\alpha})$.
In such case, $\Image u $ is contained in the plane $H_a$, and the map $\rest{u}{\partial D_{2l+1}} \colon \partial D_{2l+1}\to (\R^3\cup L_b)\cap H_a$ is continuously extended to a curve $\partial D_1\to (\R^3\cup L_b)\cap H_a$ which passes through both $E(r_0,a(0))$ and $E(r_0,a(T_a))$, where $r_0>0$ is the real number in (L2).
From the condition (L4), this cannot happen when $b>0$.
When $b=0$, we have $L_b=L_K$ and $(u,\kappa)$ is a trivial strip over $a$ (see Remark \ref{rem-trivial-strip} below). The transversality of $\overline{\partial}_{J'_{\rho}}$ at the trivial strip $(0,u,\kappa)\in \mathscr{W}$ holds without perturbation.
For a proof, see \cite[Appendix A.3]{O24}.
\end{proof}

\begin{remark}\label{rem-trivial-strip}
The trivial strip over $a\in \mathcal{R}(\Lambda_K)$  belongs to $\calM_{L_K,1}(a, (\frac{1}{2},\frac{1}{2}))$. When $a(0)=(q_a,p_a)\in \Lambda_K$, it can be described explicitly by
\[ u_a \colon [0,\infty)\times [0,1] \to T^*\R^3 \colon (s,t) \mapsto (q_a + T_at\cdot p_a, h(s)\cdot p_a), \]
where $h\colon[0,\infty)\to [0,\infty)$ is a $C^{\infty}$ function characterized by $h(0)=0$ and $\frac{dh}{ds} (s) = T_a\cdot \rho (h(s))$.
Here, we use an identification via a biholomorphic map $D_3\to ([0,\infty)\times [0,1])\setminus \{(0,0),(0,1)\}$ which is continuously extended to $D_1 \to [0,\infty)\times [0,1]\colon p_1\mapsto (0,1),p_2\mapsto (0,0)$.
In particular, the path $\rest{u_a}{\partial_j D_3}$ satisfies:
\begin{itemize}
\item $\rest{u_a}{\partial_j D_3} $ for $j=1,3$ is an immersion into $L_K$ and $\rest{u_a}{\partial_2D_3}$ is an immersion into $\R^3$. 
\item $(\rest{u_a}{\partial_jD_3})^{-1}(K) =\emptyset$ for $K = L_K \cap \R^3$. To see this when $j=2$, note the condition (\$) on $K$.
\end{itemize}
\end{remark}

When the dimension of the moduli space is small, $\calM_{\mathscr{L},l}(a;\bold{n})$ has the following properties.

\begin{lemma}\label{lem-jet-transv}
For generic $J'_{\rho}$, the following hold when $|a|+\dim B\leq 1$:
\begin{itemize}
\item[(i)]  $\calM_{\mathscr{L},l}(a) = \calM_{\mathscr{L},l}(a;\bold{n})$ for $\bold{n}=(\frac{1}{2},\dots ,\frac{1}{2})$.
\item[(ii)] For every $(b,u,\kappa) \in \calM_{\mathscr{L},l}(a)$, the curve
\[\rest{u}{\partial_j D_{2l+1}} \colon \partial_j D_{2l+1} \to  \begin{cases} L_b & \text{ if }j\text{ is odd,} \\ \bR^3 & \text{ if }j \text{ is even,} \end{cases} \]
is an immersion.
\item[(iii)] The map $\ev_{j}$ in (\ref{ev-j-map}) is transverse to $K_{\mathscr{L}}$.
Moreover, when $B=[0,b_*]$,
\[  \left\{ ((b,u,\kappa),z)\in \ev_j^{-1}(K_{\mathscr{L}}) \mid b\in \{0,b_*\} \right\} = \emptyset. \]
\item[(iv)] When $B=[0,b_*]$,
\[\left\{ ((b,u,\kappa),z)\in \ev_j^{-1}(K_{\mathscr{L}}) \mid \text{ there exists }z'\in \partial D_{2l+1}\setminus \{z\} \text{ such that } (b,u(z'))\in K_{\mathscr{L}} \right\}=\emptyset.\]
\end{itemize}
\end{lemma}
\begin{proof}
(i) If $|a|+\dim B\leq 1$,
Lemma \ref{lem-transv} shows that $\dim \calM_{\mathscr{L},l}(a,\bold{n})\leq 1 -|\bold{n}|$.
Since $|\bold{n}|=\sum_{j=1}^{2l}(2n_j-1) \leq 1$ if and only if $n_1=\dots =n_{2l}=\frac{1}{2}$, this proves the first assertion.

(ii)
The proof is similar to that of \cite[Lemma A.4]{ES}.
Suppose that $j$ is an even number.
Consider the Banach bundle $\mathscr{E}\to \mathscr{W}$ in the proof of Lemma \ref{lem-transv}, but here let us replace $L^2$-norm in \cite[Section 9.2]{CELN} with $L^p$-norm for $p>2$ to define weighted Sobolev spaces.
Then, the map $w\colon D_{2l+1}\to T^*\R^3$ for any $(w,\kappa)\in \mathscr{W}_b$ belongs to the Sobolev space $W^{2,p}$ on any compact set, and thus $w$ is of class $C^1$.
We put
\begin{align}\label{section-sigma} 
\sigma_j \colon \mathscr{W}\times \partial_j D_{2l+1}\to \mathscr{E}\colon (\bold{w},z)\mapsto \overline{\partial}_{J'_{\rho}}(\bold{w}),
\end{align}
whose zero locus is $\widetilde{\calM}_{\mathscr{L},l}^j(a)$.
Using a non-vanishing vector field $V$ on $\partial_j D_{2l+1}$, we define a map
\[ q_j \colon \mathscr{W}\times \partial_j D_{2l+1} \to \bR^3 \colon ((b,w,\kappa),z) \mapsto (dw)_z(V_z), \]
which is transverse to $0\in \bR^3$.
Then, $(q_j)^{-1}(0)$ is a Banach manifold.
Let $\sigma'_j  \colon (q_j)^{-1}(0) \to \mathscr{E}$ be 
the restriction of $\sigma_j$.
The preimage of the zero section of $\mathscr{E}$ agrees with
\[ (\sigma'_j)^{-1}(0) = \left\{ ((b,u,\kappa),z)\in \widetilde{\calM}^j_{\mathscr{L},l}(a)\ \middle| \rest{du}{T_z (\partial_j D_{2l+1})} \colon T_z (\partial_j D_{2l+1}) \to \bR^3 \text{ is the zero map} \right\}.\]
Note that $u$ cannot be the trivial strip $u_a$ when $((b,u,\kappa),z)\in (\sigma'_j)^{-1}(0)$ since $\rest{u_a}{\partial_2 D_3}$ is an immersion by Remark \ref{rem-trivial-strip}.
Then, in the way as in \cite[Lemma 9.5]{CELN}, we can show that $\sigma'_j$ is transverse to the zero section after a generic perturbation of $\rest{J}{\xi}$.

When $(\sigma'_j)^{-1}(0)$ is transversely cut out, it is a manifold whose dimension is
\[ \dim (\sigma'_j)^{-1}(0) = \dim \widetilde{\calM}_{\mathscr{L},l}^j(a) - 3 = |a| + \dim B -2  . \]
If $|a|+\dim B\leq 1$, then $\dim (\sigma'_j)^{-1}(0) <0$ and 
$(\sigma'_j)^{-1}(0)$ must be the empty set.
This means that for any $(b,u,\kappa)\in \calM_{\mathscr{L},l}(a)$ and $z\in \partial_jD_{2l+1}$, $\rest{du}{T_z(\partial_j D_{2l+1})}$ does not vanish.
This proves the second assertion when $j$ is even. The proof when $j$ is odd is similar.

(iii)
Suppose that $j$ is an even number.
The map $\ev_{j}$ on $\widetilde{\calM}_{\mathscr{L},l}^j(a)$ is extended to
\[ \widetilde{\ev}_{j} \colon \mathscr{W} \times \partial_j D_{2l+1} \to B\times \bR^3 \colon ((b,w,\kappa),z) \mapsto (b,w(z)). \]
Using the notations in Section \ref{subsec-winding}, let us define an open subset of $\mathscr{W} \times \partial_j D_{2l+1}$
\[ \mathscr{U} \coloneqq \widetilde{\ev}_{j}^{-1} ( \{ (b,q)\in B\times \R^3 \mid q\in  \iota_b( S^1\times \{0\}\times (\bR^2\cap D_{\delta}) \} ) \]
and a $C^{\infty}$ map
\[ \tilde{p} \colon \mathscr{U} \to \R^2\cap D_{\delta}\colon ((b,w,\kappa),z) \mapsto   (\pi_{D_{\delta}}\circ \iota_b^{-1}\circ w)(z). \]
Let $p\colon \mathscr{U} \cap \widetilde{\calM}_{\mathscr{L},l}^j(a)\to \R^2 \cap D_{\delta}$ be the restriction of $\tilde{p}$.
The assertion of (iii) is equivalent to that $p$ is transverse to $0\in \R^2 \cap D_{\delta}$. 

The map $\tilde{p}$ is transverse to $0$, and $\tilde{p}^{-1}(0)$ is a Banach manifold whose tangent space at $\bold{u}\in \tilde{p}^{-1}(0)$ is $T_{\bold{u}} (\tilde{p}^{-1}(0)) = \Ker (d\tilde{p})_{\bold{u}}$.
Note that for any $((b,u,\kappa),z)\in p^{-1}(0) = \tilde{p}^{-1}(0) \cap \widetilde{\calM}_{\mathscr{L},l}^j(a)$, $u$ cannot be the trivial strip $u_a$ since $u_a(z)\notin K $ for any $z\in \partial_2 D_3$ by Remark \ref{rem-trivial-strip}.
Then, the restriction of $\sigma_j$ in (\ref{section-sigma})
\[\rest{\sigma_j}{ \tilde{p}^{-1}(0)} \colon \tilde{p}^{-1}(0) \to \mathscr{E}\]
is transverse to the zero section of $\mathscr{E}$ after a generic perturbation of $\rest{J}{\xi}$.
In such case, for any $\bold{u}\in (\rest{\sigma_j}{ \tilde{p}^{-1}(0)} )^{-1}(0)$, the linearlization $(d \sigma_j)_{\bold{u}}$ is surjective when restricted on the linear subspace $\Ker (d\tilde{p})_{\bold{u}}\subset T_{\bold{u}}(\mathscr{W}\times \partial_j D_{2l+1})$.
We claim that the differential
\[ (d p)_{\bold{u}} = \rest{(d\tilde{p})_{\bold{u}}}{\ker (d\sigma_j)_{\bold{u}}} \colon \ker (d\sigma_j)_{\bold{u}} = T_{\bold{u}} \widetilde{\calM}_{\mathscr{L},l}^j(a) \to \R^2 = T_0 (\R^2\cap D_{\delta})  \]
is surjective for every $\bold{u}\in p^{-1}(0)$, that is, $p$ is transverse to $0\in \R^2\times D_{\delta}$.
Indeed, for any $v\in \R^2$, there exists $\zeta \in T_{\bold{u}}(\mathscr{W}\times \partial_jD_{2l+1})$ such that $(d\tilde{p})_{\bold{u}}(\zeta)=v$. We can take $\eta\in \ker (d\tilde{p})_{\bold{u}}$ such that $(d\sigma_j)_{\bold{u}} (\eta) = (d\sigma_j)_{\bold{u}}(\zeta)$. Then, we have $\zeta -\eta \in \ker (d\sigma_j)_{\bold{u}}$ and $(d\tilde{p})_{\bold{u}}(\xi-\eta)= (d\tilde{p})_{\bold{u}}(\xi) =v$.
This proves the claim.
The proof when $j$ is an odd number is similar.

When $B=[0,b_*]$, we may also assume that $\ev_j$ is transverse to $K_{\mathscr{L}}$ when restricted on the boundary of $\calM_{\mathscr{L},l}(a)$.
In such case, $\{ ((b,u,\kappa),z)\in \ev_j^{-1}(K_{\mathscr{L}}) \mid b\in \{0,b_*\} \}$ is cut out transversely and has a negative dimension $|a| + \dim B -2 = -1$, so it must be the empty set.

(iv) The proof of (iv) is similar to (iii), so we omit the details. We prove that
\[ \{ ((b,u,\kappa),(z,z'))\in  \calM_{\mathscr{L},l}(a) \times (\partial D_{2l+1})^2\mid z\neq z'\text{ and } (b,u(z)),(b,u(z')) \in K_{\mathscr{L}}\} \]
is cut out transversely for generic $J'_{\rho}$, and its dimension $|a|+\dim B -2$ is negative when $|a|+\dim B\leq 1$.
\end{proof}

\subsection{Compactification of the moduli spaces}

We continue to consider $\mathscr{L}=(L_b)_{b\in B}$ satisfying (L1), $\dots$ ,(L4) for $B=\{0\}$ or $[0,b_*]$.

Let $\kappa_0$ denote the unique element consisting of 
$\mathcal{C}_3$.
We define a moduli space $\mathcal{N}_{\bR^3}$ (resp. $\mathcal{N}_{\mathscr{L}}$) consisting of pairs $(b,v)$ of $b\in B$ and a $C^{\infty}$ map $v\colon D_3\to T^*\bR^3$ such that:
\begin{itemize} 
\item $dv + J'_{\rho}\circ dv\circ j_{\kappa_0} =0$.
\item $v(\partial_1 D_3 \sqcup \partial_3 D_3) \subset \bR^3$ (resp. $L_{b}$) and $v(\partial_2 D_3) \subset L_{b}$ (resp. $\bR^3$).
\item There exists $x_0\in \bR^3$ (resp. $L_{b}$) and $x_1,x_2 \in\bR^3 \cap  L_{b} = K_b$ such that
\[\begin{array}{cc}
\lim_{s\to \infty} (v\circ \psi_0)(s,t)=x_0, & \lim_{s\to -\infty} (v\circ \psi_i)(s,t) =x_i \text{ for } i=1,2
\end{array}\]
$C^{\infty}$ uniformly on $t\in [0,1]$.
\end{itemize}
Let us denote $x_i = v(p_i)$ for $i\in \{0,1,2\}$ and define the evaluation map at $p_0$
\begin{align}\label{ev-const}
\ev_0\colon \begin{cases} \mathcal{N}_{\bR^3} \to B\times \bR^3 , \\ 
 \mathcal{N}_{\mathscr{L}} \to L_{\mathscr{L}},  \end{cases} (b,v) \mapsto (b,v(p_0)) .
 \end{align}

Since $L_b$ is an exact Lagrangian submanifold of $T^*\bR^3$, there exists a $C^{\infty}$ function $f\colon L_b\to \R$ such that $df = \rest{\lambda_{\R^3}}{L_b}$. Note that $(\lambda_{\R^3})_x=0$ for any $x\in \R^3$ and thus $f$ is constant on $K_b=L_b\cap \bR^3$.
Therefore, for any element $(b,v)\in \mathcal{N}_{\bR^3}$,
\[ \int_{D_3} v^*(d\lambda_{\R^3}) = \int_{\partial D_3} v^*\lambda_{\R^3} = \int_{\partial_2 D_3} v^*(df)  = f(v(p_2))-f(v(p_1)) =0, \]
which implies that $v$ is a constant disk at $v(p_0)$.
In the same way, for any elements $(b,v)\in \mathcal{N}_{\mathscr{L}}$, $v$ is proved to be a constant disk at $v(p_0)$.
These moduli spaces of constant disks are cut out transversely, and
the evaluation maps (\ref{ev-const}) are diffeomorphisms onto $K_{\mathscr{L}}= (B\times \bR^3) \cap L_{\mathscr{L}}$.

Using the evaluation map $\ev_{j}$ in (\ref{ev-j-map}) for $j=1,\dots ,2l$, we define
\[\delta \widetilde{\calM}_{\mathscr{L},l}^j(a) \coloneqq (\ev_{j})^{-1} ( K_{\mathscr{L}} ). \]
If we take $J'_{\rho}$ as in Lemma \ref{lem-jet-transv}, it is a $C^{\infty}$ manifold of dimension $|a|+ \dim B - 1$.
We can also consider the fiber product with respect to $\ev_{j}$ and $\ev_0$
\begin{align}\label{fiber-prod}\begin{cases}  \widetilde{\calM}^j_{\mathscr{L},l} (a) \ftimes{\ev_{j}}{\ev_0} \mathcal{N}_{\bR^3} &\text{ if }j \text{ is even,} \\
\widetilde{\calM}^j_{\mathscr{L},l} (a) \ftimes{\ev_{j}}{\ev_0} \mathcal{N}_{\mathscr{L}} & \text{ if }j \text{ is odd.} \end{cases}
\end{align}
Since $\mathcal{N}_{\bR^3}$ and $\mathcal{N}_{\mathscr{L}}$ consist of constant disks in $K_{\mathscr{L}}$, we have a natural bijection
\begin{align}\label{delta-fiber-prod} 
\delta \widetilde{\calM}_{\mathscr{L},l}^j(a) \to 
\begin{cases} 
\widetilde{\calM}^j_{\mathscr{L},l} (a) \ftimes{\ev_{j}}{\ev_0} \mathcal{N}_{\bR^3} &\text{ if }j \text{ is even,} \\
\widetilde{\calM}^j_{\mathscr{L},l} (a) \ftimes{\ev_{j}}{\ev_0} \mathcal{N}_{\mathscr{L}} & \text{ if }j \text{ is odd,}
\end{cases} \end{align}
which maps $((b,u,\kappa),z)$ to $\left( ((b,u,\kappa),z), (b,v_{u(z)})\right) $,
where $v_{u(z)}$ is a constant disk at $u(z) \in K_b$.

\begin{remark}\label{rem-winding}
For $((b,u,\kappa),z)\in \delta\widetilde{\calM}^j_{\mathscr{L},l}(a)$, $u(z)$ lies in $K_b$ and the winding number of $u$ at $z$ is defined in a similar way as in Section \ref{subsec-winding}. In this case, it is given by a positive integer. See \cite[Section 6.3]{CELN}.
When $|a|+\dim B \leq 1$, for generic $J'_{\rho}$ this integer must be equal to $1$ (i.e. the smallest one) since the differential of $u$ at $z$ does not vanish by Lemma \ref{lem-jet-transv} (ii).
\end{remark}

Let us discuss a compactification of the moduli space $\calM_{\mathscr{L},l}(a)$.

\begin{proposition}\label{prop-compactify}
Let $a\in \mathcal{R}(\Lambda_{K_0})$ with $|a|=0$.
\begin{itemize}
\item If $B=\{0\}$, $\calM_{L_K,l}(a)$ is a compact $0$-dimensional manifold.
\item If $B=[0,b_*]$, $\calM_{\mathscr{L},l}(a)$ is compactified to a compact $1$-dimensional manifold whose boundary is
\[
 \calM_{L_{K_0},l }(a) \sqcup \calM_{L_{b_*},l}(a) \sqcup \coprod_{j=1}^{2l-1} \delta \widetilde{\calM}_{\mathscr{L},l-1}^j(a).
\]
\end{itemize}
\end{proposition}

\begin{proof}
In the case of $B=\{0\}$ and $\mathscr{L}=L_K$, the compactness of $\calM_{L_K,l}(a)$ is shown in \cite[Section 6.4]{CELN}.
In the case of $B=[0,b_*]$, we argue as in \cite[Proposition 6.8 and 6.9]{CELN} by observing all possible limits of sequences in the $1$-dimensional moduli space $\calM_{\mathscr{L},l}(a)$ (see also Figure \ref{fig-building}). 
Then, a compactification of $\calM_{\mathscr{L},l}(a)$ is given by
\begin{align}\label{compactify-MLn}
 \calM_{\mathscr{L},l}(a) \sqcup \coprod_{j=1}^{2l-1} \delta \widetilde{\calM}_{\mathscr{L},l-1}^j(a).
\end{align}
This is similar to \cite[Proposition 6.8 and 6.9]{CELN}, except that any $J'_{\rho}$-holomorphic building in
\[ \bar{\calM}_{J,\Lambda_{K_0}}(a;a_1,\dots,a_m) \times \prod_{k=1}^m\calM_{\mathscr{L},l_k}(a_k) , \]
illustrated by the rightmost picture in  Figure \ref{fig-building}, does not contribute to the compactification.
\begin{figure}
\centering
\begin{overpic}[height=4.45cm]{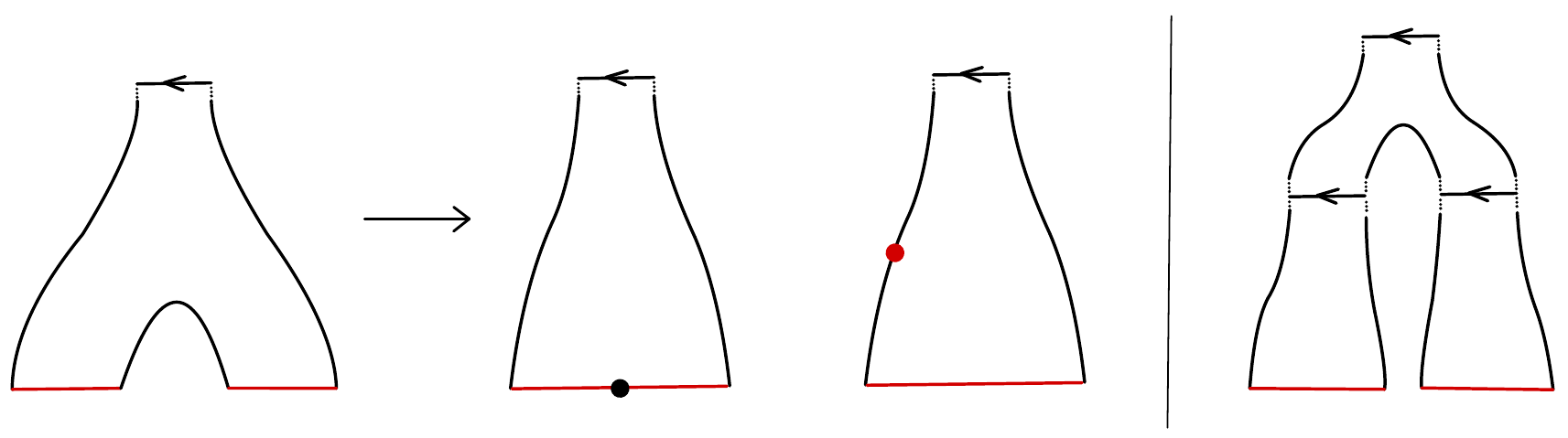}
\put(2,15){$L_{b_n}$}
\put(3,0.5){\color{red}$\R^3$}
\put(38,0){$u(z)$}
\put(52,12){$u(z)$}
\put(10.5,23.5){$a$}
\put(38.5,24){$a$}
\put(61.5,24){$a$}
\put(88.5,26.5){$a$}
\put(83.5,17){$a_1$}
\put(93,17){$a_2$}
\put(76.5,22){$\R\times \Lambda_{K_0}$}
\put(34,-3.5){($j$: even)}
\put(58,-3.5){($j$: odd)}
\put(10,-3){$u_n$}
\put(23,15){$n\to \infty$}
\put(33,16){$L_b$}
\put(55.5,16){$L_b$}
\put(34,0.5){\color{red}$\R^3$}
\put(57,0.5){\color{red}$\R^3$}
\end{overpic}
\caption{The leftmost picture represents a sequence $(b_n,u_n,\kappa_n)_{n=1,2,\dots}$ in $\calM_{\mathscr{L},l}(a)$ when $l=2$. Taking a subsequence if necessary, such a sequence can converge to a $J'_{\rho}$-holomorphic curve $u$ with a marked point $z$ on the boundary such that $u(z)\in K_b$ for $b=\lim_{n\to \infty}b_n$. (We may think of it as $J'_{\rho}$-holomorphic curve with a boundary node at $z$ such that a constant disk at $u(z)$ is attached to $u$ at the node.)
A priori, the sequence may converge to a $J'_{\rho}$-holomorphic building as in the rightmost picture, but this cannot happen since $\bar{\calM}_{\Lambda_{K_0},J}(a;a_1,\dots,a_m)=\emptyset$ when $|a|=0$.}\label{fig-building}
\end{figure}
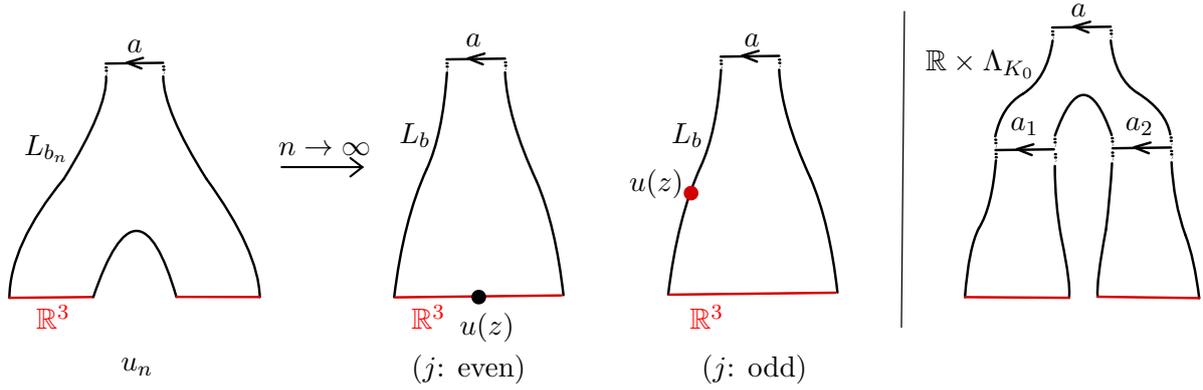
Indeed, $|a'|\geq 0$ for any $a'\in \mathcal{R}(\Lambda_K)$ and the moduli space $\bar{\calM}_{J,\Lambda_{K_0}}(a;a_1,\dots,a_m)$ has dimension $|a|-\sum_{k=1}^m|a_k|-1$, which is negative  when $|a|=0$.
(When comparing with \cite[Proposition 6.8]{CELN}, note that for any $( (b,u,\kappa),z)\in \delta \widetilde{\calM}_{\mathscr{L},l-1}^j(a)$, the winding number of $u$ at $z$ is $1$ by Remark \ref{rem-winding}.)

We can show
that (\ref{compactify-MLn}) is a compact $1$-dimensional manifold  by modifying the proof of \cite[Theorem 10.3]{CELN}.
For each element of $\delta \widetilde{\calM}_{\mathscr{L},l-1}^j(a)$,
an element $(((b,u,\kappa),z), (b, v))$ of the set (\ref{fiber-prod}) is determined by the bijection (\ref{delta-fiber-prod}).
Here, $v\colon D_3\to T^*\R^3$ is a constant disk at $u(z)\in L_b\cap \R^3 = K_b$.
As in \cite[Section 10.3]{CELN}, we construct a pre-gluing $w_{\rho}\colon D_{2l+1}\to T^*\R^3$ from $u$ and $v$, where $\rho\gg0$ is a gluing parameter, such that $w_{\rho}(\partial_{2i-1} D_{2l+1})\subset L_b$ and $w_{\rho}(\partial_{2i} D_{2l+1})\subset \R^3$.
We have the same estimate on the linearlization of $\overline{\partial}_{J'_{\rho}}$ at $(b,w_{\rho})$ as in \cite[Lemma 10.12]{CELN} and apply the Floer's Picard lemma \cite[Lemma 10.10]{CELN}.
The rest of argument is parallel to that in \cite[page 772]{CELN}.
\end{proof}

We also have the following result, which is a slight modification of \cite[Theorem 6.5]{CELN}.

\begin{proposition}\label{lem-finiteness}
For each $a\in \mathcal{R}(\Lambda_K)$, there are only finitely many $l\in \Z_{\geq 0}$ such that $\calM_{\mathscr{L},l}(a)\neq \emptyset$.
\end{proposition}
\begin{proof}
Assume that there exists a sequence $(l_n)_{n=1,2,\dots}$ in $\Z_{\geq 0}$ and $\bold{u}_n=(b_n,u_n,\kappa_n)\in \calM_{\mathscr{L},l_n}(a)$ for $n=1,2,\dots$ such that $\lim_{n\to \infty}l_n=\infty$.
Here, the domain of each $u_n$ is extended to $D_1$.

Taking a subsequence if necessary, $u_n$ converges  in the sense of \cite[Theorem 4.1]{CEL} to a $J'_{\rho}$-holomorphic curve whose boundary condition is given by $L_b\cup \R^3$ for $b=\lim_{n\to \infty}b_n$.
For the proof, the difference from \cite[Theorem 4.1]{CEL} is that the boundary condition of $u_n$ varies in $n$.
Nevertheless, after taking a subsequence, on a domain in $D_1$ disjoint from $(\rest{u_n}{\partial D_1})^{-1}(K_{b_n})$, we can apply usual compactness argument, and on a domain $S\subset D_1$ such that $u_n(S)\subset \iota_{b_n}(S^1\times (-\delta,\delta) \times D_{\delta})$, we can apply the same argument as in \cite[Theorem 4.1]{CEL} to the holomorphic map $\iota_{b_n}^{-1}\circ u_n$, whose boundary condition $S^1\times \{0\} \times (\R^2\cup \sqrt{-1}\R^2)$ is independent of $n$.
Note that a necessary bound on the energy appearing in \cite[Theorem 4.1]{CEL} is automatically satisfied in the present case, as mentioned in \cite[Remark 6.6]{CELN}.

The rest of argument to deduce a contradiction, using the winding numbers of $u_n$ at the points in $(\rest{u_n}{\partial D_1})^{-1}(K_{b_n})$, is parallel to \cite[Theorem 5.1]{CEL} since it is reduced to studying holomorphic maps $\pi_{D_{\delta}}\circ \iota_{b_n}\circ u_n$, whose boundary condition $\R^2\cup \sqrt{-1}\R^2$ is independent of $n$.
\end{proof}

%

\section{String homology, cord algebra and $\LCH_0(\Lambda_K)$}\label{sec-string}
\label{sec:string}

Let $K$ be a knot in $\bR^3$.
We will introduce the the string homology in degree $0$ and the cord algebra of $K$.
They are $\Z_2$-algebras defined in topological ways and isomorphic to $\LCH_0(\Lambda_K)$.

\subsection{String homology in degree $0$ without $N$-strings}\label{subsec-string}

Fix a constant $c>1$.
We choose $\epsilon>0$ such that
\begin{align}\label{tub-nbd}
N= \{ q+v \in \R^3 \mid q\in K,\ v\in (T_qK)^{\perp},\ |v|\leq 2\epsilon \}
\end{align}
is a tubular neighborhood of $K$ in $\R^3$.
For each $l\in \bZ_{\geq 1}$, we define a space of $l$ tuples of paths
\begin{align}\label{Sigma-l}
\Sigma^l_K \coloneqq \left\{ (s_1,\dots ,s_l)\ \middle| 
\begin{array}{l}
\text{for each }i =1,\dots ,l,\  s_i\colon [0,T_i]\to \R^3 \text{ is a }C^1\text{ path such that } \\
T_i>0,\ s_i(\{0,T_i\})\subset K \text{ and } |\dot{s}_i(t)| < c \text{ for every }t\in[0,T_i] 
\end{array}
\right\} . \end{align}
\begin{remark}\label{rem-Banach}
For the path $s_i\colon [0,T_i]\to \bR^3$ as above, we can associate a real number $T_i\in \R_{>0}$ and a $C^1$ path $[0,1]\to \R^3 \colon t\mapsto s_i(T_it)$. In this way, $\Sigma^l_K$ is embedded onto an open subset of a Banach manifold
\[ \left( \R_{>0} \times \{s\colon [0,1]\to \R^3 \mid C^1\text{ path such that }s(\{0,1\})\subset K\} \right)^{\times l}.\]
Here, the space of $C^1$ paths is endowed with the $C^1$ topology.
Through this embedding, $\Sigma^l_K$ is equipped with a structure of Banach manifold.
\end{remark}

We define an open subset of $\Sigma^l_K$ by
\begin{align}\label{Sigma-l-epsilon} \Sigma^l_{K,\epsilon}\coloneqq \left\{ (s_1,\dots ,s_l) \in \Sigma^l_K  \middle| \  \min_{1\leq i\leq l} \len (s_i) <\epsilon \right\}. \end{align}
For $l=0$, we exceptionally define $\Sigma^0_K$ to be a one point set $\{ \mathrm{pt} \}$ and $\Sigma^0_{K,\epsilon}\coloneqq \emptyset$.
We will later mod out $0$-chains in $\Sigma^l_{K,\epsilon}$. That is, we ignore $0$-chains $(s_1,\dots ,s_l)$ if at least one of $s_1,\dots ,s_l$ has length less than $\epsilon$.

Using $\Sigma_K\coloneqq \coprod_{l=0}^{\infty} \Sigma^l_K$, we shall define a $\Z_2$-algebra $\bar{H}^{\str}_0(K)$, which is a reduction of the string homology $H^{\str}_0(K)$ of $K$ in degree $0$ defined by
\cite[Section 2.1]{CELN} (see also Section \ref{subsec-full-string}).
The reduction is done by omitting `N-strings' which are paths in the tubular neighborhood $N$ of $K$.
The following construction of $\bar{H}^{\str}_0(K)$ is very close to the $\R$-algebra in \cite[Section 6.1]{O22}.
\begin{definition}\label{def-gen-chain}
We define $\bZ_2$-vector spaces $C_0(\Sigma^l_K)$, $C_1(\Sigma^l_K)$ and $\bar{C}_0(\Sigma^l_K)$ for $l\in \bZ_{\geq 0}$ as follows:
For $l=0$, we set $C_0(\Sigma^0_K)= \bar{C}_0(\Sigma^0_K)= \Z_2$ and $C_1(\Sigma^0_K)=0$.
For $l\geq 1$:
\begin{itemize} 
\item We define $C_0(\Sigma^l_K)$ as the $\bZ_2$-vector space spanned by $\bold{s}=(s_1,\dots ,s_l)\in \Sigma^l_K$ such that for every $i\in \{1,\dots ,l\}$:
\begin{itemize}
\item[(0a)] $s_i^{-1}(K)=\{0,T_i\}$. Here, $[0,T_i]$ is the domain of $s_i$.
\item[(0b)]  $s_i$ is not tangent to $K$ at the endpoints $\{0,T_i\}$.
\end{itemize}
\item  We define $C_1(\Sigma^l_K)$ as the $\bZ_2$-vector space spanned by continuous $1$-parameter families
\[(\bold{s}^{\lambda})_{\lambda \in [0,1]} = (s^{\lambda}_1,\dots ,s^{\lambda}_{l})_{\lambda \in [0,1]}\]
in $\Sigma^l_K$ such that $\bold{s}^0$ and $\bold{s}^1$ satisfy the conditions (0a) and (0b), and for every $i\in \{1,\dots ,l\}$:
\begin{itemize}
\item[(1a)]
Let us set $U_i\coloneqq \{(\lambda,t) \mid 0 <  \lambda < 1,\ 0 < t < T^{\lambda}_{i}\}$, where $[0,T^{\lambda}_i]$ is the domain of $s^{\lambda}_i$, and consider a map
\[ 
\Gamma_{i} \colon U_{i} \to \R^3 \colon (\lambda,t) \mapsto s^{\lambda}_{i}(t).
\]
Then, $\Gamma_i$ is a $C^1$ map and transverse to $K$.
\item[(1b)] For every  $\lambda\in [0,1]$, $s^{\lambda}_i$ is not tangent to $K$ at the endpoints $\{0, T^{\lambda}_i\}$.
\item[(1c)] For any two distinct points $(\lambda,t), (\lambda',t') \in \coprod_{i=1,\dots ,l} \Gamma_i^{-1}(K)$, we have $\lambda\neq \lambda'$.
\end{itemize}
\item 
Let $C_0(\Sigma^l_{K,\epsilon})$ be the $\Z_2$-linear subspace of $C_0(\Sigma^l_K)$ generated by $\bold{s}\in \Sigma^l_{K,\epsilon}$ satisfying (0a) and (0b).
We then define the quotient
\[ \bar{C}_0(\Sigma^l_K) \coloneqq C_0(\Sigma^l_K)/C_0(\Sigma^l_{K,\epsilon}). \]
\end{itemize}
\end{definition}

Next, we define $\bZ_2$-linear maps
\begin{align*} 
\partial & \colon C_1(\Sigma^l_K) \to  \bar{C}_0(\Sigma^l_K) , \\
\delta \coloneqq \sum_{i=1}^l\delta_i & \colon C_1(\Sigma^l_K) \to  \bar{C}_0(\Sigma^{l+1}_K) .
\end{align*}
as follows:
On $C_1(\Sigma^0_K)=0$, we set $\partial \coloneqq 0 $ and $\delta \coloneqq 0$.
Let $l\in \Z_{\geq 1}$. For every $x=(\bold{s}^{\lambda})_{\lambda\in[0,1]}\in C_1(\Sigma^l_K)$,
we define
\[ \partial (x) \coloneqq \bold{s}^1 - \bold{s}^0 \in \bar{C}_0(\Sigma^l_K) .\]
Moreover, for each $i\in \{ 1,\dots ,l\}$ and $(\lambda,t)\in \Gamma_i^{-1}(K)$, where $\Gamma_i$ is the map in (1a), we set
\begin{align}\label{split-path}
\bold{s}^{(\lambda,t),i} \coloneqq \left( s^{\lambda}_1,\dots  , s^{\lambda}_{i-1} , \rest{s^{\lambda}_{i}}{[0,t]} ,  \rest{s^{\lambda}_{i}}{[t,T^{\lambda}_{i}]}(\cdot -t) , s^{\lambda}_{i+1} , \dots , s^{\lambda}_{l} \right) \in \Sigma^{l+1}_K 
\end{align}
and define
\[ \delta_{i} (x) = \sum_{(\lambda,t) \in \Gamma_i^{-1}(K)} \bold{s}^{(\lambda,t),i}\in \bar{C}_0(\Sigma^{l+1}_K) .\]
We can check that $\bold{s}^{(\lambda,t),i}$ satisfies (0a) from (1c), and (0b) from (1a) and (1b).

$\Gamma_i^{-1}(K)$ is a $0$-dimensional manifold by (1a).
The definition of $\delta_i(x)$ makes sense for the following reason.

\begin{lemma}
There are only finitely many $(\lambda,t)\in \Gamma_i^{-1}(K)$ such that $\bold{s}^{(\lambda,t),i}\neq 0$ in $\bar{C}_0(\Sigma^{l+1}_K)$.
\end{lemma}
\begin{proof}
Take any $(\lambda,t)\in \Gamma_i^{-1}(K)$.
If $0\leq t\leq \frac{\epsilon}{2c}$ (resp. $T^{\lambda}_i- \frac{\epsilon}{2c} \leq t \leq T^{\lambda}_i$), then the path $\rest{s^{\lambda}_{i}}{[0,t]}$ (resp. $\rest{s^{\lambda}_{i}}{[t,T^{\lambda}_i]}$) has a length less than $\epsilon$ since $|\dot{s}^{\lambda}_i(t)|<c$ for every $t$.
Let us put
\[
I_i\coloneqq \Gamma_i^{-1}(K) \cap \textstyle{ \{(\lambda,t) \mid \frac{\epsilon}{2c} \leq t\leq T^{\lambda}_i-\frac{\epsilon}{2c}\} } ,\]
which is a discrete subset of $U_i$.
Since $s^{0}_i([\frac{\epsilon}{2c},T^0_i-\frac{\epsilon}{2c}])$ and $s^{1}_i([\frac{\epsilon}{2c},T^1_i-\frac{\epsilon}{2c}])$ are disjoint from $K$ by the condition (0a) on $\bold{s}^0$ and $\bold{s}^1$, there exists $\lambda'>0$ such that $s^{\lambda}_i([\frac{\epsilon}{2c},T^{\lambda}_i-\frac{\epsilon}{2c}])$ is disjoint from $K$ for any $\lambda \in [0,\lambda') \cup (1-\lambda',1]$.
This means that $I_i$ is contained in a compact set
\[ \textstyle{ \{(\lambda,t) \mid \lambda' \leq \lambda \leq 1-\lambda', \  \frac{\epsilon}{2c} \leq t\leq T^{\lambda}_i-\frac{\epsilon}{2c}\}, }\]
so $I_i$ is a finite set.

For any $(\lambda,t)\in \Gamma_i^{-1}(K) \setminus I_i$, 
$\bold{s}^{(\lambda,t),i}$ belongs to $\Sigma^{l+1}_{K,\epsilon}$ since either $\rest{s^{\lambda}_{i}}{[0,t]} $ or $\rest{s^{\lambda}_{i}}{[t,T^{\lambda}_{i}]}$ has a length less than $\epsilon$.
Therefore, $\bold{s}^{(\lambda,t),i} =0$ in $\bar{C}_0(\Sigma^{l+1}_K)$ if $(\lambda,t)\in  \Gamma_i^{-1}(K) \setminus I_i$.
\end{proof}

We set $\bar{C}_0(\Sigma_K)\coloneqq \bigoplus_{l=0}^{\infty} \bar{C}_0(\Sigma^l_K)$ and
$C_1(\Sigma_K) \coloneqq \bigoplus_{l=0}^{\infty} C_1(\Sigma^l_K)$, and define
\[D_{K} \coloneqq \partial + \delta \colon C_1(\Sigma_K)\to \bar{C}_0(\Sigma_K).\]
Then, the \textit{string homology of $K$ in degree $0$} with coefficients in $\bZ_2$ is defined by the quotient
\[ \bar{H}^{\str}_0(K) \coloneqq \bar{C}_0(\Sigma_K)/  \Image D_{K}. \]

In addition, $\bar{C}_0(\Sigma_K)$ has a product structure
\[
\bar{C}_0(\Sigma^l_K)\otimes \bar{C}_0(\Sigma^{l'}_K) \to \bar{C}_0(\Sigma^{l+l'}_K) \colon x\otimes x' \mapsto x*x'
\]
induced by a natural map
\begin{align}\label{product-map} \Sigma^l_K \times \Sigma^{l'}_K \to \Sigma^{l+l'}_K \colon \left( (s_1,\dots ,s_{l}),(s'_1,\dots ,s'_{l'})\right) \mapsto ( s_1,\dots ,s_{l} , s'_1,\dots ,s'_{l'} ). 
\end{align}
When $l=0$ or $l'=0$, we regard it as the identity map since $\Sigma^0_K=\{\pt\}$.
Likewise, we define
\[  C_0(\Sigma^l_K)\otimes C_1(\Sigma^{l'}_K)\otimes C_0(\Sigma^{l''}_K)\to C_1(\Sigma^{l+l'+l''}_K) \colon x_1\otimes y \otimes x_2 \mapsto x *_1 y *_1 x' , \]
for which $D_K( x *_1 y *_1 x' ) = x* (D_K y)* x'$ holds in $\bar{C}_0(\Sigma_K)$.
Therefore, $\Image D_K$ is a two-sided ideal of $\bar{C}_0(\Sigma_K)$.
Hence, the product $*$ on $\bar{C}_0(\Sigma_K)$ descends to a product on
$\bar{H}^{\str}_0(K)$, and this determines a structure of $\bZ_2$-algebra.
It has a unit which comes from $1\in \bar{C}_0(\Sigma^0_K)=\bZ_2$.

The next lemma will be used in Section \ref{subsubsec-proof-main}.

\begin{lemma}\label{lem-homologous}
Suppose that $\bold{s}=(s_1,\dots ,s_l)\in \Sigma^l_K$ satisfies (0a) and (0b).
Then, there exists a neighborhood $U_{\bold{s}}$ of $\bold{s}$ in $\Sigma^l_K$ such that the following hold  for any $\bold{s}'\in U_{\bold{s}}$:
\begin{itemize} 
\item $\bold{s}'$ satisfies (0a) and (0b).
\item There exists $h\in C_1(\Sigma^l_K)$ such that $ \partial (h) = \bold{s}'-\bold{s}$ in $\bar{C}_0(\Sigma^l_K)$ and $\delta (h) =0 $ in $\bar{C}_0(\Sigma^{l+1}_K)$.
\end{itemize}
\end{lemma}
\begin{proof}
The conditions (0a) and (0b) are open conditions.
Since $\Sigma^l_K$ is a Banach manifold, we can take $U_{\bold{s}}$ to be an open ball containing $\bold{s}$ such that every $\bold{s}' \in U_{\bold{s}}$ satisfies (0a) and (0b).
Furthermore, for any $\bold{s}' \in U_{\bold{s}}$,
we choose a $C^1$ path $[0,1]\to U_{\bold{s}}\colon \lambda \mapsto \bold{s}^{\lambda}$ such that $\bold{s}^0=\bold{s}$ and $\bold{s}^1=\bold{s}'$.
Consider a $1$-chain $h=(\bold{s}^{\lambda})_{\lambda\in [0,1]}$.
$\bold{s}^{\lambda}$ satisfies (0a) for each $\lambda \in [0,1]$.
Therefore, if we consider a map $\Gamma_i$ as in (1a), $\Gamma_i^{-1}(K)$ is the empty set for $i=1,\dots ,l$.
Then, (1a), (1b) and (1c) apparently hold for $h$, and as an element of $C_1(\Sigma^l_K)$, $h$ satisfies
$\partial (h) = \bold{s}'-\bold{s}$ and $\delta (h)=0$.
\end{proof}

\subsection{Cord algebra over $\bZ_2$}\label{subsubsec-cord-Z2}
Fix an orientation and a framing of $K$.
Let $K'\subset N$ be a parallel copy of $K$ with respect to the framing.

\begin{definition}\label{def-cord-Z/2}
Let $\mathcal{P}_{K}$ be the set of homotopy classes of paths $\gamma\colon[0,1] \to \bR^3 \setminus K$ such that $\gamma(0),\gamma(1)\in K'$.
Then, we define $\mathcal{A}$ as the non-commutative $\bZ_2$-algebra freely generated by $\mathcal{P}_{K}$.
The \textit{cord algebra} of $K$ is defined by
\[ \Cord(K)\coloneqq \mathcal{A}/\mathcal{I}, \]
where $\mathcal{I}$ is the two-sided ideal of $\mathcal{A}$
generated by the following relations (see Figure \ref{fig-cord-Z2}):
\begin{itemize}
\item[(i)] If $x\in \mathcal{P}_{K}$ is represented by a constant path in $K'$, then $x=0$.
\item[(ii)] If $x_1,x_2\in \mathcal{P}_{K}$ are represented by paths $\gamma_1,\gamma_2$ such that $\gamma_1(1)=\gamma_2(0)$, then
\[ [\gamma_1\bullet \gamma_2] - [\gamma_1\bullet m \bullet \gamma_2] = x_1 \cdot x_2.\]
Here, $m\colon [0,1]\to \R^3\setminus K$ is a loop based at $\gamma_1(1)$ and represents the meridian of $K$.
\end{itemize}
\end{definition}
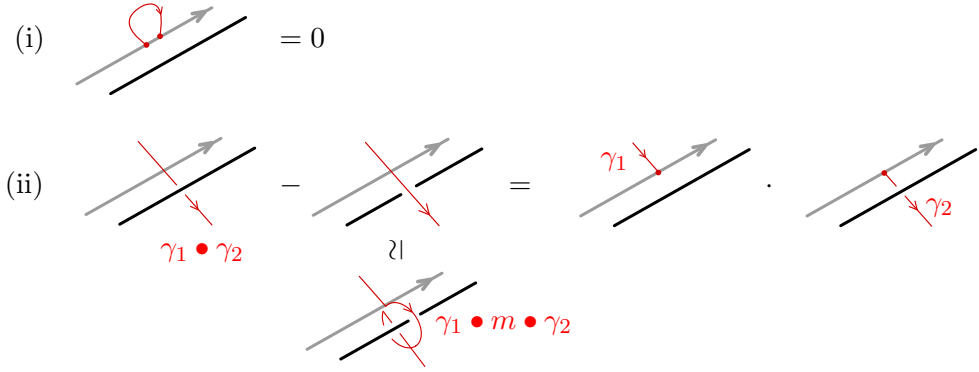
\begin{figure}
\centering
\begin{overpic}[height=5cm]{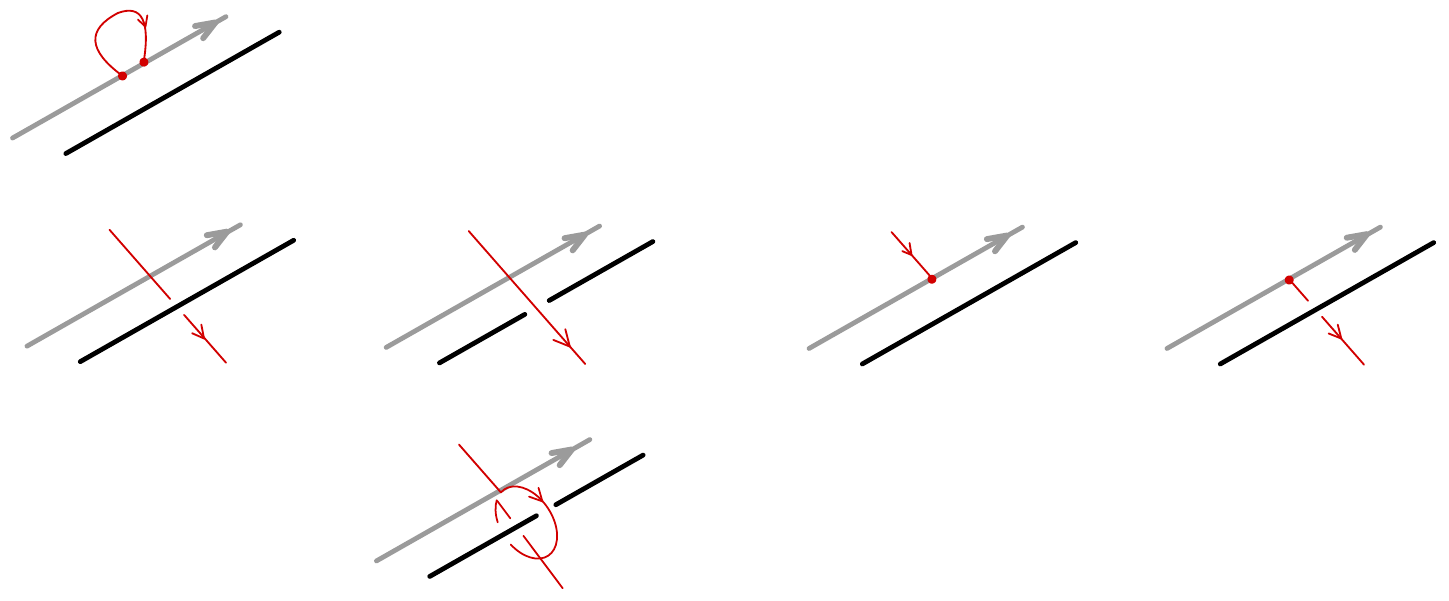}
\put(-6,36){(i)}
\put(23,36){$=0$}
\put(-7,20){(ii)}
\put(23,20){$-$}
\put(48,20){$=$}
\put(76,20){$\cdot$}
\put(35,13){\rotatebox[origin=c]{90}{$\simeq$}}
\put(10,13){\color{red}$\gamma_1\bullet \gamma_2$}
\put(40,5){\color{red}$\gamma_1\bullet m\bullet \gamma_2$}
\put(58,23){\color{red}$\gamma_1$}
\put(93.5,18){\color{red}$\gamma_2$}
\end{overpic}
\caption{The relations generating the ideal $\mathcal{I}$. `$\simeq$' means a homotopy equivalence.}\label{fig-cord-Z2}
\end{figure}


Fix a point $q_1\in K$ and let $q'_1$ be the corresponding point on the copy $K'$.
Referring to \cite[Proposition 2.9]{CELN}, choose paths $\gamma_Q, \tilde{\gamma}_Q\colon [0,1] \to \{q\in N \mid |q-q_1|<\epsilon\}$ such that: 
\begin{itemize}
\item $\gamma_Q(0)=q_1$ and $\gamma_Q(1)=q'_1$. $\tilde{\gamma}_Q(0)=q'_1$ and $\tilde{\gamma}_Q(1)=q_1$.
\item $\gamma_Q^{-1}(K) =\{0\}$ and $\tilde{\gamma}_Q^{-1}(K)=\{1\}$. $\dot{\gamma}_Q(0)\notin T_{q_1}K$ and $\dot{\tilde{\gamma}}_Q(1)\notin T_{q_1}K$.
\end{itemize}

For any $x\in \mathcal{P}_{K}$,
take a representative $\gamma_x\colon [0,1] \to \bR^3 \setminus K$ such that $\gamma_x(0) = \gamma_x(1) = q'_1$.
By concatenation, we obtain $\hat{\gamma}_x\coloneqq \gamma_Q\bullet \gamma_x\bullet \tilde{\gamma}_Q \colon [0,3]\to \bR^3$.
We remark that $\hat{\gamma}_x $ belongs to $ \Sigma^1_K$ and satisfies the conditions (0a) and (0b).
(More precisely, we need to modify $\hat{\gamma}_x$ near the concatenating points so that it is a $C^1$ path.)

\begin{proposition}\label{prop-isom-IK}
There exists an isomorphism of $\Z_2$-algebra 
\[ I_K\colon \Cord(K) \to \bar{H}_0^{\str}(K) \]
which maps $x_1\cdots x_l\in \Cord(K)$ for $x_1,\dots ,x_l\in \mathcal{P}_{K}$ to an element of $\bar{H}^{\str}_0(K)$ represented by
$(\hat{\gamma}_{x_1},\dots ,\hat{\gamma}_{x_l})\in \bar{C}_0(\Sigma^l_K)$.
\end{proposition}
The proof is completely parallel to \cite[Proposition 6.3]{O22}, where the cord algebra and the string homology are defined over $\R$.
We also remark that the proof of \cite[Proposition 6.3]{O22} is a modification of \cite[Proposition 2.9]{CELN}.

The isomorphism $I_K$ is independent of the choice of $\gamma_Q$ and $\tilde{\gamma}_Q$.
Since $\bar{H}^{\str}_0(K)$ is free from the orientation and the framing of $K$, this proposition implies that up to natural isomorphisms, $\Cord(K)$ is independent of  the orientation and the framing of $K$.

\subsection{Isomorphism from $\LCH_0(\Lambda_K)$ to $\bar{H}^{\str}_0(K)$}\label{subsubsec-isom-LCH-str}

Let $K$ be a real analytic knot in $\R^3$ satisfying the condition (\$) in Section \ref{subsec-property-of-moduli}.

Let $a\in \mathcal{R}(\Lambda_K)$ with $|a|=0$.
Taking $B=\{0\}$ and $\mathscr{L}=L_K$, consider the moduli space $\calM_{\mathscr{L},l}(a)$ for $l\in \Z_{\geq 1}$.
Hereafter, to denote its elements, we often omit writing the conformal structures.
Let us take $J'_{\rho}$ as in Lemma \ref{lem-jet-transv}.
For every $u \in \calM_{L_K,l}(a)$ and $i\in \{1,\dots ,l\}$,
we have a curve
\[\rest{u}{\overline{\partial_{2i} D_{2l+1}}} \colon \overline{\partial_{2i} D_{2l+1}} \to \bR^3\]
which is in $K$ at the endpoints $\{p_{2i-1}, p_{2i}\}$.
Since $\dim \calM_{L_K,l}(a)=|a| = 0$, Lemma \ref{lem-jet-transv} (iii) implies that $u(z)\notin K$ if $z\in \partial_{2i}D_{2l+1}$.
Moreover, this curve is not tangent to $K$ at the endpoints since $u\in \calM_{L_K,l}(a;(\frac{1}{2},\dots ,\frac{1}{2}))$ by Lemma \ref{lem-jet-transv} (i).
It is also an immersion by Lemma \ref{lem-jet-transv} (ii), so there is a unique diffeomorphism $\varphi^u_i\colon [0,T^u_i]\to \overline{\partial_{2i} D_{2l+1}}$ which maps $0$ to $p_{2i}$ such that
\begin{align}\label{path-si}
s^u_i \coloneqq \rest{u}{\overline{\partial_{2i} D_{2l+1}}} \circ \varphi^u_i \colon [0,T^u_i] \to \R^3 
\end{align}
is a $C^{\infty}$ path and the speed $|\dot{s}^u_i(t)|$ is constant to $1$ on $t\in [0,T^u_i]$.
(In particular, $T_i^u$ is the length of the curve $\rest{u}{\overline{\partial_{2i} D_{2l+1}}}$.)
We take
\[\bold{s}^u\coloneqq (s_1^u,\dots ,s_l^u ) \in \Sigma^l_K,\]
then, we can check that $\bold{s}^u$ satisfies (0a) and (0b) from the properties of $\rest{u}{\overline{\partial_{2i} D_{2l+1}}}$.

We define a unital $\Z_2$-algebra map
\begin{align}\label{Psi'K} \Psi'_K \colon \mathcal{A}_0(\Lambda_K) \to \bar{C}_0(\Sigma_K) 
\end{align}
which maps $a\in \mathcal{R}(\Lambda_K)$ with $|a|=0$ to  $(\Psi'_{K,l}(a))_{l=0,1,\dots} \in \bigoplus_{l=0}^{\infty} \bar{C}_0(\Sigma^l_K)$ determined by
\begin{align*} \Psi'_{K,0}(a) & \coloneqq \#_{\Z_2} \calM_{L_K,0}(a) \in \bar{C}_0(\Sigma^0_K) =\bZ_2 ,\\
 \Psi'_{K,l}(a) & \coloneqq \sum_{u \in \calM_{L_K,l}(a)} \bold{s}^u \in \bar{C}_0(\Sigma^l_K) \text{ if }l\geq 1,
\end{align*}
(here, note that $\calM_{L_K,l}(a)$ is a finite set by Lemma \ref{prop-compactify} and that $\calM_{L_K,l}(a) =\emptyset$ when $l\gg 0$ by Proposition \ref{lem-finiteness}).
Note that $\Psi'_K(1)=1\in \bar{C}_0(\Sigma^0_K)=\Z_2$ and
\[\Psi'_K(a_1\cdots a_m) = \Psi'_K(a_1)*\dots * \Psi'_K(a_m)\]
for every $a_1,\dots ,a_m \in \mathcal{R}(\Lambda_K)$ with $|a_1|=\dots =|a_m|=0$.

\begin{theorem}\label{thm-str-cord}
For the subspace $d_J(\mathcal{A}_1(\Lambda_K))$ of $\mathcal{A}_0(\Lambda_K)$, $\Psi'_K\left( d_J(\mathcal{A}_1(\Lambda_K))\right) $ is contained in $ \Image D_K$.
Moreover, the induced $\Z_2$-algebra map
\[ \Psi_K \colon \LCH_0(\Lambda_K) \to \bar{H}^{\str}_0(K) \]
is an isomorphism.
\end{theorem}
The proof is left to Appendix \ref{subsec-proof-isom}.
\begin{remark}
In Appendix \ref{subsec-proof-isom}, we will give a proof which relies on the result \cite[Theorem 1.2]{CELN} via the fully non-commutative Legendrian contact homology of $\Lambda_K$.
By introducing the string homology over $\Z_2$ in degree $1$ as in \cite[Section 5]{CELN} and using the length filtrations as in \cite[Section 7.6]{CELN},
it would be possible to directly show that $\Psi'_K$ induces an isomorphism in degree $0$.
\end{remark}

\section{Coincidence of isomorphisms induced by family of knots}\label{sec-coincidence}
\label{ch:autom}

\subsection{Isomorphisms on three $\Z_2$-algebras}
\label{sec:isom_cord_alg}


Let $(K_t)_{t\in [0,1]}$ be an isotopy of smooth knots in $\bR^3$.
We then take an ambient isotopy $(F_t\colon \bR^3\to\bR^3)_{t\in [0,1]}$ with compact support such that:
\begin{itemize}
\item $F_0=\id_{\bR^3}$.
\item $F_{1-t}(K_{t})= K_1$ for every $t\in[0,1]$.
\end{itemize}
In particular, $F_1(K_0)=K_1$.
This is obtained as follows: We take a smooth family of embedding $(f_t\colon S^1\to \R^3)_{t\in [0,1]}$ such that $f_t(S^1)=K_{1-t}$.
It is extended to an ambient isotopy $(G_t\colon \R^3\to \R^3)_{t\in [0,1]}$ such that $G_0=\id_{\R^3}$ and $G_t\circ f_0 = f_t$, and then we set $F_t\coloneqq G_t^{-1}$.

Let us define isomorphisms of $\Z_2$-algebras induced by $(K_t)_{t\in [0,1]}$ and $(F_t)_{t\in [0,1]}$.

\subsubsection{On the string homology}\label{subsubsec-isom-F-string}
For $j\in \{0,1\}$, take $c_j>1$ and $\epsilon_j>0$ such that we can define the spaces $\Sigma^l_{K_j}$ from (\ref{Sigma-l}) and $\Sigma^l_{K_j,\epsilon_j}$ from (\ref{Sigma-l-epsilon})
by setting $c=c_j$, $K=K_j$ and $\epsilon=\epsilon_j$.
Then, as in Section \ref{subsec-string}, a linear map
\[D_{K_j} \colon C_1(\Sigma_{K_j}) \to \bar{C}_0(\Sigma_{K_j})\]
and a $\Z_2$-algebra $\bar{H}^{\str}_0(K_j)= \bar{C}_0(\Sigma_{K_j})/\Image D_{K_j}$ are defined.

We take $c_1> 1$ to be large and $\epsilon_0<0$ to be small so that
\[ \textstyle{ |(dF_1)_q(v)|< \min \{ \frac{c_1}{c_0}, \frac{\epsilon_1}{\epsilon_0}\}\cdot |v| }\]
for any $q\in \bR^3$ and $v\in T_q\bR^3$.
Then, we can define a map
\[ \Sigma^l_{K_0}\to \Sigma^l_{K_1} \colon (s_1,\dots ,s_l) \mapsto (F_1\circ s_1,\dots ,F_1\circ s_l)   \]
for $l\in \Z_{\geq 1}$.
The conditions (0a), (0b), (1a), (1b) and (1c) are preserved by this map since $F_1$ is a diffeomorphism which maps $K_0$ onto $K_1$.
Let us define $\Z_2$-linear maps $F_1^{(0)}$ and $F_1^{(1)}$ by
\begin{align}\label{F1-chain} F_1^{(i)} \colon C_i(\Sigma^l_{K_0})\to C_i(\Sigma^l_{K_1}) \colon \begin{cases} (s_1,\dots ,s_l) \mapsto (F_1\circ s_1,\dots ,F_1\circ s_l) &\text{if }i=0  ,\\
(s^{\lambda}_1,\dots ,s^{\lambda}_l)_{\lambda\in[0,1]} \mapsto (F_1\circ s^{\lambda}_1,\dots ,F_1\circ s^{\lambda}_l)_{\lambda\in[0,1]} & \text{if }i=1,
\end{cases}
\end{align}
for $l\in \Z_{\geq 1}$. Let us also set $F_1^{(0)} = \id_{\Z_2} \colon C_0(\Sigma^0_{K_0})\to C_0(\Sigma^0_{K_1})$.
We note that
\[F_1^{(0)}(C_0(\Sigma^l_{K_0,\epsilon_0})) \subset C_0(\Sigma^l_{K_1,\epsilon_1})\]
since for any $C^1$ path $s \colon [0,T]\to \bR^3$ with $\len (\gamma)<\epsilon_0$,
\[\len (F_1\circ s) \leq \sup_{q\in \bR^3, 0\neq v\in T_q\bR^3} \textstyle{ \frac{| dF_1 (v) |}{|v|} } \cdot \len (s) \leq \frac{\epsilon_1}{\epsilon_0} \cdot 
\len (s) <\epsilon_1. \]
Therefore, $F^{(0)}_1$ induces a linear map
$ \bar{F}_1^{(0)} \colon \bar{C}_0(\Sigma_{K_0}) \to \bar{C}_0(\Sigma_{K_1})$.
It is claer from the construction of $D_K=\partial +\delta$ for $K\in \{K_0,K_1\}$ that
\[ D_{K_1} \circ F_1^{(1)} = \bar{F}^{(0)}_1 \circ D_{K_0} \colon C_1(\Sigma_{K_0}) \to \bar{C}_0(\Sigma_{K_1}). \]
Therefore, a $\Z_2$-linear map
\[  F_1^{\str} \colon \bar{H}^{\str}_0(K_0)\to \bar{H}^{\str}_0(K_1) \]
is induced from $\bar{F}^{(0)}_1$.
Moreover, this preserves the product structures and the units.

\subsubsection{On the cord algebra}\label{subsubsec-isom-on-cord}
For $j\in \{0,1\}$, fix an orientation and a framing of $K_j$, and let $K'_j$ be a parallel copy of $K_j$ with respect to the framing.
We may fix them so that $F_1(K'_0)=K'_1$.
Then, $F_1$ induces a bijection
\[ \mathcal{P}_{K_0}  \to \mathcal{P}_{K_1} \colon [\gamma] \mapsto [F_1\circ \gamma]  \]
between the sets of generators of cord algebras.
Moreover, this bijection preserves the relations in Definition \ref{def-cord-Z/2}, so
it induces an isomorphisms
\[ (F_1)_* \colon \Cord(K_0) \to \Cord (K_1) . \]

From the construction of the isomorphism $I_{K_j} \colon \bar{H}^{\str}_0(K_j) \to \Cord (K_j)$ of Proposition \ref{prop-isom-IK} for $j\in \{0,1\}$,
the following diagram commutes:
\begin{align}\label{diagram-F-IK} 
\begin{split}
\xymatrix@C=30pt{
\bar{H}^{\str}_0(K_0) \ar[d]_-{I_{K_0}} \ar[r]^-{F_1^{\str}} & \bar{H}^{\str}_0(K_1) \ar[d]_-{I_{K_1}} \\
\Cord (K_0) \ar[r]^-{(F_1)_*} & \Cord (K_1).
}\end{split}
\end{align}

The map $(F_1)_*$ depends only on the isotopy $(K_t)_{t\in [0,1]}$.
Indeed, if we take another ambient isotopy $(F'_t)_{t\in[0,1]}$ satisfying the same condition as $(F_t)_{t\in[0,1]}$, $F_1$ is homotopic to $F'_1$ through a family $(F'_s\circ F_s^{-1}\circ F_1)_{s\in [0,1]}$ of diffeomorphisms which map $K_0$ to $K_1$.
This means that $[F_1\circ \gamma]=[F'_1\circ \gamma]$ for any $[\gamma]\in \mathcal{P}_{K_0}$.
Moreover, given a smooth family $(K^s_t)_{s,t\in [0,1]}$ of knots such that $K^s_0=K_0$ and $K^s_1=K_1$ for every $s\in [0,1]$, we take an ambient isotopy $(F^s_t)_{t\in [0,1]}$ as above for $(K^s_t)_{t\in [0,1]}$ which varies smoothly on $s$. Then, $[F^0_1\circ \gamma]=[F^1_1\circ \gamma]$ for every $[\gamma]\in \mathcal{P}_{K_0}$.
Therefore, by setting $\rho \left( [(K_t)_{t\in [0,1]}]\right) \coloneqq (F_1)_*$, we obtain a well-defined map
\[ \rho\colon \Pi_{\sm}(K_0,K_1)\to \Hom (\Cord (K_0),\Cord(K_1)) . \] 
In addition, given another knot $K_2$, we can check that $\rho(x')\circ\rho(x) = \rho(x'\cdot x)$ for any $x\in \Pi_{\sm}(K_0,K_1)$ and $x'\in \Pi_{\sm}(K_1,K_2)$.

\subsubsection{On the Legendrian contact homology in degree $0$}\label{subsubsec-isom-on-LCH}

For $j\in \{0,1\}$, assume that every Reeb chord of $\Lambda_{K_j}$ is non-degenerate. 

From an isotopy $(K_t)_{t\in [0,1]}$, we obtain a family of Legendrian tori $(\Lambda_{K_t})_{t\in [0,1]}$ in $U^*\R^3$.
By Lemma \ref{lem-isotopy-cobordism}, we get a smooth family of exact Lagrangian cobordisms $(L_t)_{t\in [0,1]}$ such that
\[\begin{array}{cc} L_t \cap \left( (-\infty,r_-]\times U^*\R^3 \right) = (-\infty,r_-]\times \Lambda_{K_t}, &
L_t \cap \left(  [r_+,\infty) \times U^*\R^3 \right) = [r_+,\infty) \times \Lambda_{K_0},
\end{array}\]
and $L_0 =\R\times \Lambda_K$.
For the element $\bLambda \left( [(K_t)_{t\in [0,1]}]\right) = [(\Lambda_{K_t})_{t\in [0,1]}] \in \Pi_{\Leg}(\Lambda_{K_0},\Lambda_{K_1})$, we have an isomorphism
\[\Phi\left( [(\Lambda_{K_t})_{t\in [0,1]}] \right) = \Phi_{L_1} \colon \LCH_0(\Lambda_{K_0}) \to \LCH_0(\Lambda_{K_1}).\]
by Theorem \ref{thm-DGA-isom}.

As explained in Section \ref{subsec-perturb-Lag}, take an isotopy $(\Lambda_s)_{s\in [0,1]}$ satisfying $\Lambda_0=\Lambda_{K_0}$ and (\ref{non-intersect-with-orbit}), and then construct a smooth family of exact Lagrangian cobordisms $(M_s)_{s\in [0,1]}$.
Recall that $M_0=\R\times \Lambda_{K_0}$ and both the positive and negative ends of $M_s$ are $\Lambda_{K_0}$.
We take $0<s_0\leq 1$ as in Proposition \ref{prop-perturb-Lag} for the family $(L_t)_{t\in [0,1]}$.
Changing the parameter, we may assume that $s_0=1$.
By using $(M_s)_{s\in [0,1]}$, we define a smooth family of exact Lagrangian cobordisms $(\widetilde{L}_t)_{t\in [0,1]}$  to be
\begin{align}\label{tilde-L} \widetilde{L}_t\coloneqq L_t\circ_{\Lambda_{K_0}} M_t.
\end{align}
Note that $\Phi\left( [(\Lambda_{K_t})_{t\in [0,1]}] \right) = \Phi_{L_1} = \Phi_{\widetilde{L}_1}$ since $(L_1\circ_{\Lambda_K} M_s)_{s\in [0,1]}$ is an isotopy from $L_1$ to $\widetilde{L}_1$.
In addition, note that $\widetilde{L}_0 = \R\times \Lambda_{K_0}$.

\subsection{Proof of main results}\label{subsec-proof-main}

\subsubsection{The road to Theorem \ref{thm-1}}

We will prove the following proposition in Section \ref{subsubsec-proof-main}.
Recall the condition (\$) in Section \ref{subsec-property-of-moduli}.
\begin{proposition}\label{prop-chain-homotopy}
Let $(K_t)_{t\in [0,1]}$, and $(\widetilde{L}_t)_{t\in [0,1]}$ be as in Section \ref{subsubsec-isom-on-LCH}.
We additionally assume that $K_0$ and $K_1$ satisfy (\$) and that $(K_t)_{t\in [0,1]}$ is a smooth family of real analytic knots.
Take $(F_t)_{t\in[0,1]}$ and define $F^{\str}_1$ as in Section \ref{subsubsec-isom-F-string}.
Then, the following diagram commutes:
\begin{align}\label{diagram-F-PsiK}
\begin{split}
\xymatrix@C=30pt{
\LCH_0(\Lambda_{K_0}) \ar[d]_-{\Psi_{K_0}} \ar[r]^-{\Phi_{\widetilde{L}_1}} & \LCH_0(\Lambda_{K_1}) \ar[d]_-{\Psi_{K_1}} \\
\bar{H}^{\str}_0(K_0) \ar[r]^-{F_1^{\str}} & \bar{H}^{\str}_0(K_1).
}
\end{split}\end{align}
\end{proposition}

Assuming this proposition, we get the following.

\begin{theorem}\label{thm-monodromy}
Let $K_0$ and $K_1$ be real analytic knots satisfying (\$).
For the map $\rho$ defined in Section \ref{subsubsec-isom-on-cord}, the following diagram commutes:
\[ \xymatrix{
\Pi_{\Leg}(\Lambda_{K_0},\Lambda_{K_1})\ar[r]^-{\Phi} & \Hom (\LCH_0(\Lambda_{K_0}), \LCH_0(\Lambda_{K_1}))  \ar[d]_-{\cong}^-{I_{K_1}^{-1}\circ \Psi_{K_1}\circ(\cdot)\circ \Psi_{K_0}^{-1}\circ I_{K_0}} \\
\Pi_{\sm}(K_0,K_1) \ar[u]^-{\bLambda} \ar[r]^-{\rho} & \Hom (\Cord(K_0),\Cord(K_1)) .
} \]

\end{theorem}

\begin{proof}[Proof of Theorem \ref{thm-monodromy} assuming Proposition \ref{prop-chain-homotopy}]
Take any $x \in \Pi_{\sm}(K_0,K_1)$ and its representative $(K_t)_{t\in [0,1]}$.
$(K_t)_{t\in[0,1]}$ is given by a smooth family of embeddings $(f_t\colon S^1\to \R^3)_{t\in[0,1]}$, 
and by using the Fourier expansion of $f_t$ on $S^1$, one can show that
$(f_t\colon S^1\to \R^3)_{t\in[0,1]}$ is homotopic to a smooth family of real analytic embeddings.
Therefore, we may make the same assumption on $(K_t)_{t\in[0,1]}$ as in Proposition \ref{prop-chain-homotopy}.

We have seen in Section \ref{subsubsec-isom-on-LCH} that 
$\Phi( \bLambda(x)) = \Phi\left( [(\Lambda_{K_t})_{t\in [0,1]}] \right) = \Phi_{\widetilde{L}_1}$.
From the diagrams (\ref{diagram-F-IK}) and (\ref{diagram-F-PsiK}), we obtain a commutative diagram
\begin{align*}
\xymatrix@C=40pt@R=15pt{
\LCH_0(\Lambda_{K_0}) \ar[d]_-{\Psi_{K_0}} \ar[r]^-{\Phi(\bLambda(x))} & \LCH_0(\Lambda_{K_1}) \ar[d]_-{\Psi_{K_1}} \\
\bar{H}^{\str}_0(K_0) \ar[r]^-{F_1^{\str}} & \bar{H}^{\str}_0(K_1) \\
\Cord(K_0)\ar[u]^-{I_{K_0}} \ar[r]^-{(F_1)_* = \rho (x)} & \Cord(K_1) \ar[u]^-{I_{K_1}} .
}
\end{align*}
This proves that $\rho(x) = I_{K_1}^{-1}\circ \Psi_{K_1}\circ \Phi(\bLambda(x)) \circ \Psi_{K_0}^{-1}\circ I_{K_0}$.
\end{proof}

Theorem \ref{thm-1} is deduced from Theorem \ref{thm-monodromy}.

\begin{proof}[Proof of Theorem \ref{thm-1} using Theorem \ref{thm-monodromy}]
For general smooth knots $K_0,K_1$ in $\R^3$, we take real analytic knots $K'_0$ and $K'_1$ satisfying (\$) such that $K_j$ is isotopic to $K'_j$ for $j\in \{0,1\}$.
Fix elements $x_0\in \Pi_{\sm}(K'_0,K_0)$ and $x_1\in \Pi_{\sm}(K_1,K'_1)$ and define for every $y\in \Pi_{\Leg}(\Lambda_{K_0},\Lambda_{K_1})$
\begin{align*} \Theta (y) &\coloneqq \rho(x_1)^{-1}\circ I_{K'_1}^{-1}\circ \Psi_{K'_1} \circ \Phi( \bLambda(x_1)\cdot y \cdot \bLambda(x_0))\circ \Psi_{K'_0}^{-1}\circ I_{K'_0} \circ \rho(x_0)^{-1} \\
& \in \Hom(\Cord(K_0),\Cord(K_1)).
\end{align*}
Then, for any $x\in \Pi_{\sm}(K_0,K_1)$, $\bLambda(x_1)\cdot \bLambda(x)\cdot \bLambda(x_0)=\bLambda (x_1\cdot x\cdot x_0)$ and Theorem \ref{thm-monodromy} for $K'_0$ and $K'_1$ shows that
\[ \Theta (\bLambda(x)) = \rho(x_1)^{-1}\circ \rho(x_1\cdot x\cdot x_0) \circ \rho(x_0)^{-1} , \]
and this is equal to $\rho(x)$.
Therefore, $\rho = \Theta \circ \bLambda$ holds.
\end{proof}

The rest of this section is devoted to proving Proposition \ref{prop-chain-homotopy}.

\subsubsection{A family of exact Lagrangian fillings of $\Lambda_{K_0}$}\label{subsubsec-family-of-Lag}

We construct a $1$-parameter family $(L_b)_{b\in [0,\infty)}$ of exact Lagrangian fillings of $\Lambda_{K_0}$ in $T^*\R^3$ as follows:
Take non-decreasing $C^{\infty}$ functions $t\colon [0,\infty)\to [0,1]$ and $s\colon [0,\infty)\to \R_{\geq 0}$ such that:
\begin{itemize}
\item $t(b)=b$ near $b=0$ and $t(b)=1$ when $b\geq 1$.
\item $s(b)=0$ when $b\in [0,1]$ and $\lim_{b\to \infty}s(b)=\infty$.
\end{itemize}
Using $(\widetilde{L}_t)_{t\in [0,1]}$ in Section \ref{subsubsec-isom-on-LCH}, we define for every $b\in [0,\infty)$
\[  L_b \coloneqq  \left( L_{K_{t(b)}} \setminus E ( [s(b) +r_-,\infty)\times \Lambda_{K_{t(b)}}) \right) \cup  (E\circ \tau_{s(b)}) \left(\widetilde{L}_{t(b)} \cap ( [r_-,\infty)\times U^*\bR^3) \right) . \]
See Figure \ref{fig-Lagrangian}.
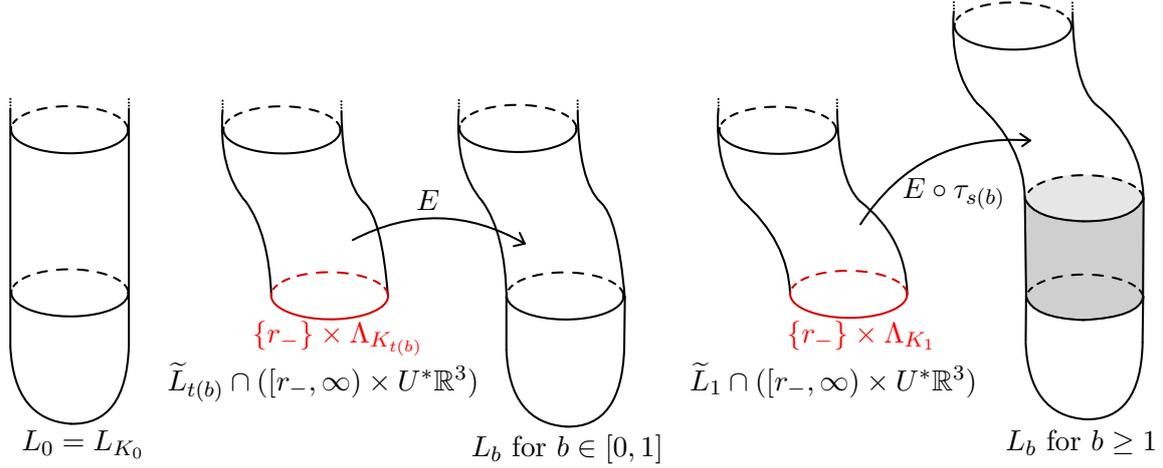
\begin{figure}
\centering
\begin{overpic}[height=6cm]{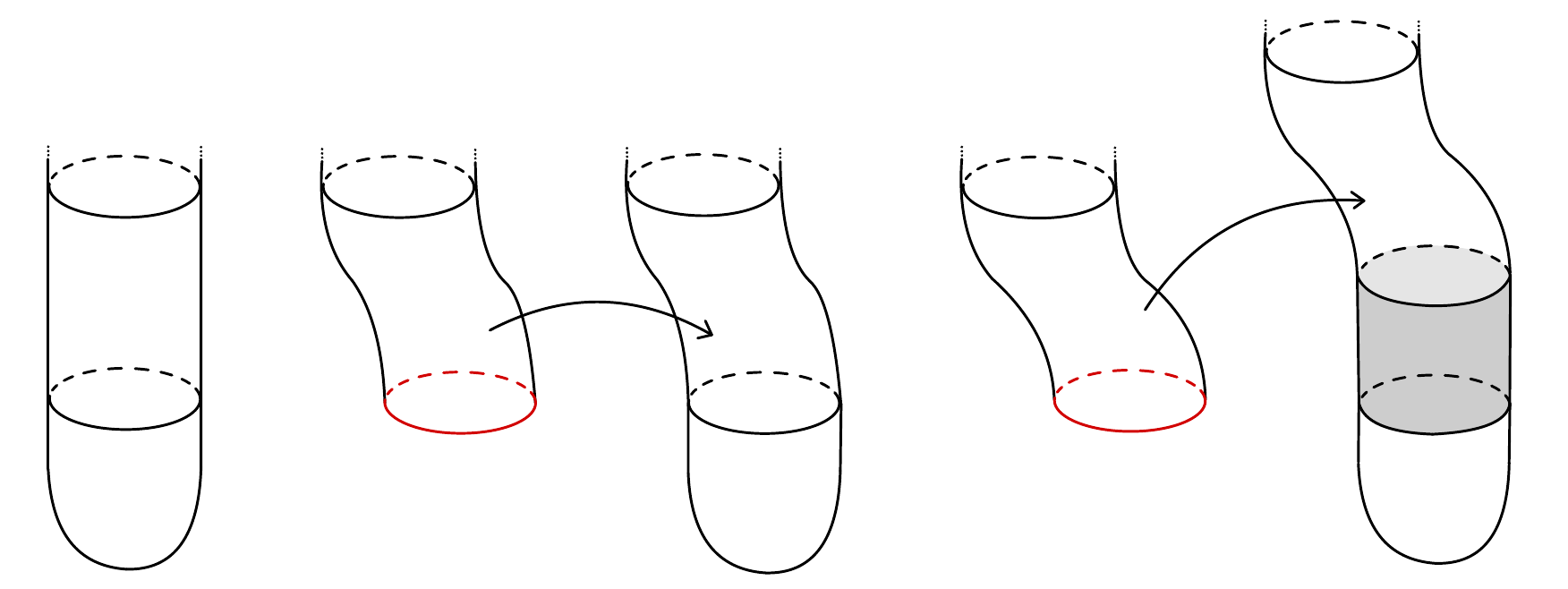}
\put(4,-1){$L_0=L_{K_0}$}
\put(23,8){\color{red}$\{r_-\}\times\Lambda_{K_{t(b)}}$}
\put(16,4){$\widetilde{L}_{t(b)}\cap ([r_-,\infty)\times U^*\R^3)$}
\put(41,-1.5){$L_b$ for $b\in [0,1]$}
\put(36.5,19){$E$}
\put(67,8){\color{red}$\{r_-\}\times \Lambda_{K_1}$}
\put(59,4){$\widetilde{L}_1\cap ([r_-,\infty)\times U^*\R^3)$}
\put(85,-1){$L_b$ for $b\geq 1$}
\put(76.5,20){$E\circ \tau_{s(b)}$}
\end{overpic}
\caption{
When $b=0$, $L_0 = L_{K_0}$ (the left picture).
When $b\in [0,1]$, $L_b$ is an exact Lagrangian filling of $\Lambda_{K_0}$ obtained by concatenating the conormal bundle $L_{K_{t(b)}}$ with $\widetilde{L}_{t(b)}$ via $E$ (the middle picture).
When $b\geq 1$, $L_b$ is given by a concatenation of $L_{K_1}$ and $\widetilde{L}_1$, with the shaded cylindrical region stretched as $b\to \infty$ (the right picture).}\label{fig-Lagrangian}
\end{figure}

Take a sequence $(b_n)_{n=1,2,\dots}$ in $[1,\infty)$ such that $\lim_{n\to \infty} b_n =\infty$.
We set
\[ \mathscr{L}_n\coloneqq (L_b)_{b\in [0,b_n]}. \]
Then, this satisfies the conditions (L1), (L2) and (L3) for the family of knots $(K_{t(b)})_{b\in [0,b_n]}$.
Moreover, from the construction of $(\widetilde{L}_t)_{t\in [0,1]}$ by (\ref{tilde-L}), the condition (L4) is also ensured by Proposition \ref{prop-perturb-Lag}.

For each triple $(a,l,n)$ of  $a\in \mathcal{R}(\Lambda_{K_0})$, $l\in \Z_{\geq 0}$ and $n=1,2,\dots$,
we have a moduli space $\calM_{\mathscr{L}_n,l}(a)$ defined in Section \ref{subsec-def-moduli}.
For generic $J'_{\rho}$,  every moduli space $\calM_{\mathscr{L}_n,l}(a)$ has the properties in Lemma \ref{lem-transv} and \ref{lem-jet-transv}.
Let $C\calM_{\mathscr{L}_n,l}(a)$ denote the compactification of $\calM_{\mathscr{L}_n,l}(a)$ as in Proposition \ref{prop-compactify}.
It is a $1$-dimensional manifold with boundary
\[ \calM_{L_{K_0},l}(a)\sqcup \calM_{L_{b_n},l}(a) \sqcup \coprod_{j=1}^{2l-1}\delta\widetilde{\calM}_{\mathscr{L}_n,l-1}^j(a) . \]

Let us observe the boundary component $\calM_{L_{b_n},l}(a)$ when $n\to \infty$.

\begin{lemma}\label{lem-bijection}
When $n\in\Z_{\geq 1}$ is sufficiently large, there are bijections
\begin{align*} 
\mathcal{G}_{n,0} &\colon \calM_{\widetilde{L}_1,J}(a,\emptyset) \sqcup  \coprod_{ |a_1|=\dots = |a_m|=0 } \left( \calM_{\widetilde{L}_1,J}(a;a_1,\dots,a_m) \times \prod_{k=1}^m\calM_{L_{K_1},0}(a_k) \right) \to \calM_{L_{b_n},0}(a) , \\
\mathcal{G}_{n,l} & \colon \coprod_{ |a_1|=\dots = |a_m|=0, \  l_1+\dots+l_m =l } \left( \calM_{\widetilde{L}_1,J}(a;a_1,\dots,a_m) \times \prod_{k=1}^m\calM_{L_{K_1},l_k}(a_k) \right) \to \calM_{L_{b_n},l}(a),
\end{align*}
for $l\geq 1$
such that for any pseudo-holomorphic building $\bold{u}$ in the domain of $\mathcal{G}_{n,l}$ ($l\in\Z_{\geq 0}$), $\mathcal{G}_{n,l}(\bold{u})$ converges to $\bold{u}$ when $n\to \infty$ in the sense of \cite[Section 8.2]{BEHWZ}. 
Here, $m$, $a_k$ and $l_k$ run over $\bZ_{\geq 1}$, $\mathcal{R}(\Lambda_{K_1})$ and $\bZ_{\geq 0}$ respectively.
\end{lemma}

\begin{proof}
The proof is the same as \cite[Corollary 3.11]{EHK} using a stretching argument. 
Note that when $n\gg 0$, $\mathcal{G}_{n,0}(\bold{u})$ for $\bold{u}\in  \calM_{\widetilde{L}_1,J}(a,\emptyset)$ is a $J$-holomorphic curve contained in $T^*\R^3\setminus D^*_1\R^3 \cong (0,\infty)\times U^*\R^3$ and escapes to $\{\infty\}\times U^*\R^3$ as $n\to\infty$.
\end{proof}

\subsubsection{Proof of Proposition \ref{prop-chain-homotopy}}\label{subsubsec-proof-main}

Let us observe the compactification $C\calM_{\mathscr{L}_n,l}(a)$.
First, in the case of $l=0$,
\[ \partial \left(  C\calM_{\mathscr{L}_n,0}(a)\right) = \calM_{L_{K_0},0}(a)\sqcup \calM_{L_{b_n},0}(a). \]
We use the bijection $\mathcal{G}_{n,0}$ in Lemma \ref{lem-bijection}.
Then, we have an equation
\begin{align}\label{conut-M0}
\begin{split}
0 =\ & \#_{\Z_2} \calM_{L_{K_0},0}(a) + \#_{\Z_2} \calM_{L_{b_n},0}(a) \\
=\  &  \#_{\Z_2} \calM_{L_{K_0},0}(a) + \#_{\Z_2} \calM_{\widetilde{L}_1,J}(a,\emptyset) \\
&+ \sum_{m\geq 1,\ |a_1|=\dots = |a_m|=0} \left( \#_{\Z_2} \calM_{\widetilde{L}_1,J}(a;a_1,\dots ,a_m) \right) \cdot \left(  \prod_{k=1}^m \#_{Z_2} \calM_{L_{K_1},0}(a_k)\right).
\end{split}
\end{align}

Next, we observe the case $l\geq 1$.
A possible picture of the moduli space is illustrated in Figure \ref{fig-moduli}.
Hereafter, we abbreviate the conformal structures.
First, consider a boundary point
\[ \lambda= ((b,u),z) \in \delta\widetilde{\calM}_{\mathscr{L}_n,l-1}^{2i-1}(a)\]
for $i\in \{1,\dots ,l\}$.
For any sequence $(b_{\nu},u_{\nu})_{\nu=1,2,\dots}$ in $ \calM_{\mathscr{L}_n,l}(a)$ converging to $\lambda$, the path
\[\rest{u_{\nu}}{\partial_{2i}D_{2l+1}}\colon \partial_{2i}D_{2l+1}\to \R^3\]
converges to a constant path at $u(z)\in K$ in $C^{\infty}$ topology when $\nu \to \infty$.
(When $l=2$, see the picture in Figure \ref{fig-building} when $j$ is odd.)
Therefore, there exists an open neighborhood $U'_{\lambda} \cong [0,1)$ of $\lambda$ such that
\begin{align}\label{length-short}
\len \left( F_{1-t(b)}\circ \rest{u}{\overline{\partial_{2i} D_{2l+1} }} \right) <\epsilon_1 
\end{align}
for every $(b,u)\in U'_{\lambda}\setminus \{\lambda\}\subset \calM_{\mathscr{L}_n,l}(a)$.
Take a subset $U_{\lambda}\subset U'_{\lambda}$ such that $U_{\lambda}\cong [0,\frac{1}{2})$ via the homeomorphism $U'_{\lambda}\cong [0,1)$, and let $\wh{\lambda}\in U'_{\lambda}$ be the point corresponding to $\frac{1}{2}\in [0,1)$.
(In Figure \ref{fig-moduli}, they appear as $U_{\lambda_4}$ and $\wh{\lambda}_4$.)
Removing $U_{\lambda}$ from $C\calM_{\mathscr{L}_n,l}(a)$ for all $\lambda \in \coprod_{i=1}^{l} \delta\widetilde{\calM}_{\mathscr{L}_n,l-1}^{2i-1}(a)$, we define $C'\calM_{\mathscr{L}_n,l}(a)$.
It is a compact $1$-dimensional manifold with boundary
\begin{align}\label{boundary-C'M}
 \calM_{L_{K_0},l }(a) \sqcup \calM_{L_{b_n},l}(a)  \sqcup \coprod_{i=1}^{l-1}  \delta\widetilde{\calM}_{\mathscr{L}_n,l-1}^{2i}(a) 
\sqcup 
\coprod_{i=1}^{l} \left\{ \wh{\lambda} \in \calM_{\mathscr{L}_n,l}(a)\  \middle| \lambda \in \delta\widetilde{\calM}_{\mathscr{L}_n,l-1}^{2i-1}(a)   \right\} .
\end{align}

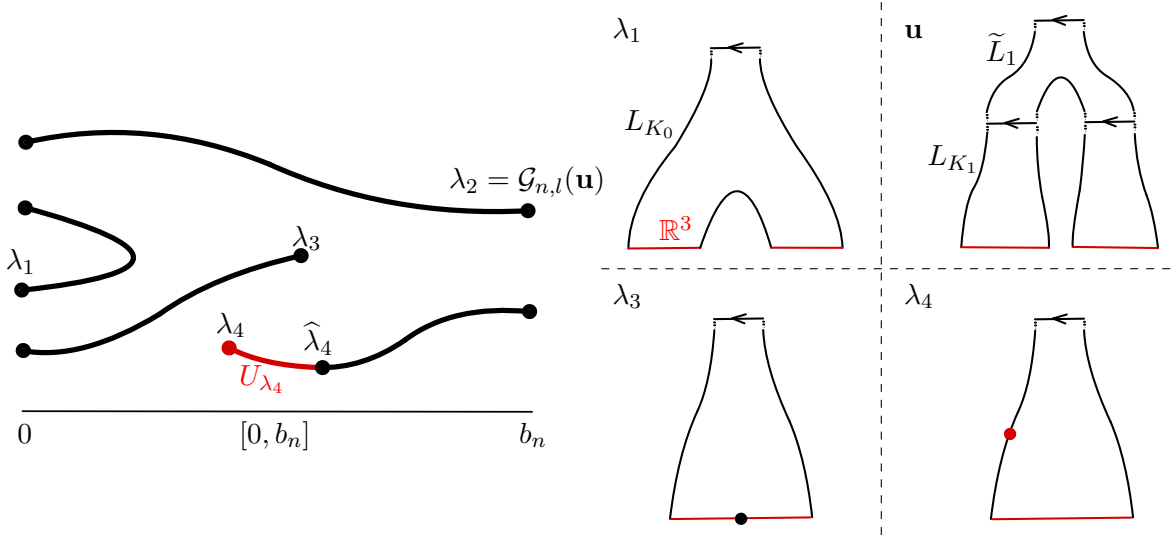
\begin{figure}
\centering
\begin{overpic}[height=7cm]{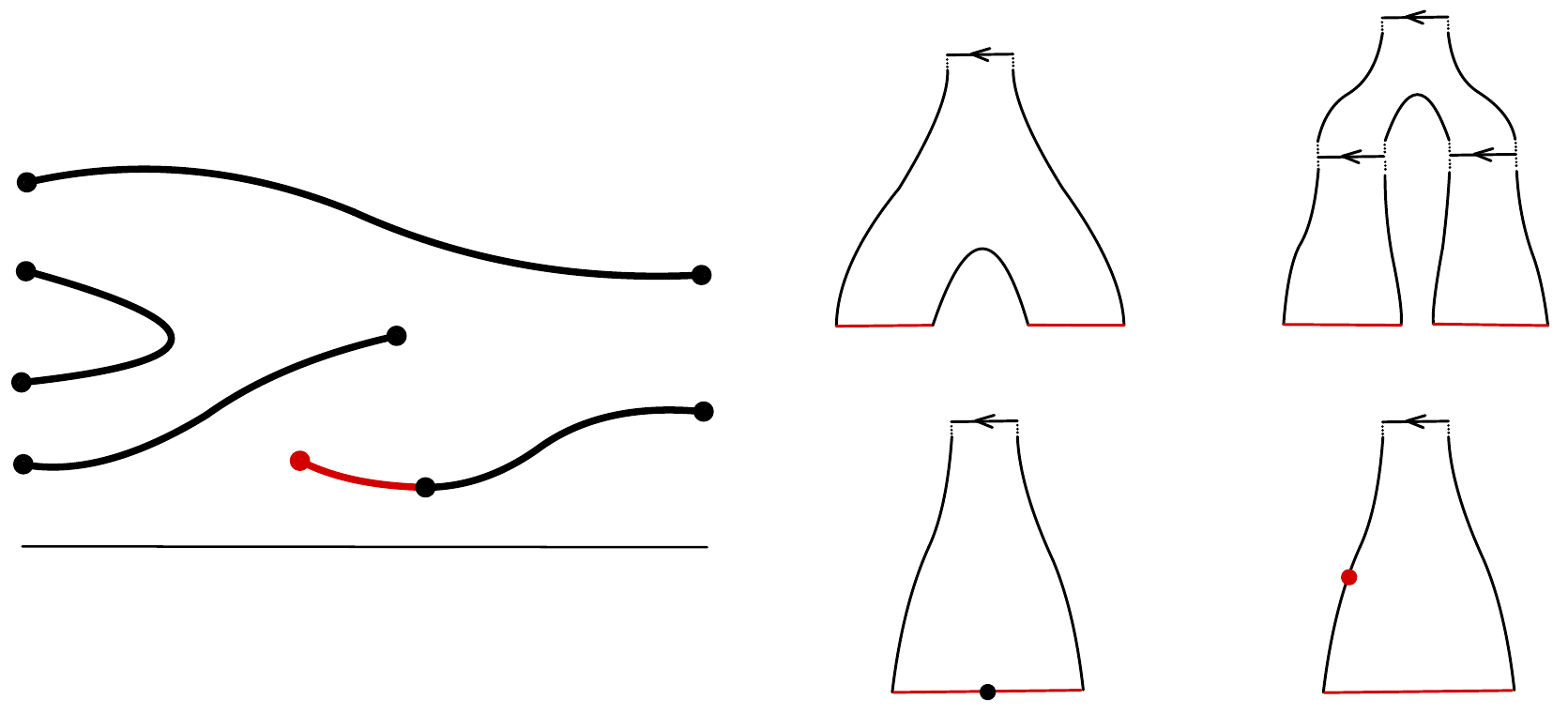}
\put(20,8){$[0,b_n]$}
\put(1,8){$0$}
\put(44,8){$b_n$}
\put(0,23){$\lambda_1$}
\put(24.5,25){$\lambda_3$}
\put(18,17.5){$\lambda_4$}
\put(20,13){\color{red}$U_{\lambda_4}$}
\put(25.5,16){$\wh{\lambda}_4$}
\put(38,30){$\lambda_2=\mathcal{G}_{n,l}(\bold{u})$}
\put(53,35){$L_{K_0}$}
\put(56,25.5){\color{red}$\R^3$}
\put(84,41){$\widetilde{L}_1$}
\put(79,32){$L_{K_1}$}
\put(52,43){$\lambda_1$}
\put(77,43){$\bold{u}$}
\put(52,20){$\lambda_3$}
\put(77,20){$\lambda_4$}
\put(51,23){\tikz \draw[dashed] (0,0)--(7.5,0);}
\put(75,0){\tikz \draw[dashed] (0,-7.1)--(0,0);}
\end{overpic}
\caption{The left-hand side is a (possible) picture of $C\calM_{\mathscr{L}_n,l}(a)$ over $[0,b_n]$. $\lambda_1,\dots ,\lambda_4$ represent the four types of boundary points.
Among the four pictures in the right-hand side, the lower two ($\lambda_3$ and $\lambda_4$) correspond to the middle picture in Figure \ref{fig-building}.
The upper left one illustrates $\lambda_1$ which represents $J'_{\rho}$-holomorphic curves used to define $\Psi_{K_0}$ as in Section \ref{subsubsec-isom-LCH-str}.
When $n$ is sufficiently large, $\lambda_2$ is equal to $\mathcal{G}_{n,l}(\bold{u})$ for a $J'_{\rho}$-holomorphic building $\bold{u}$ as in the upper right picture. It consists of a $J$-holomorphic curve used to define $\Phi_{\widetilde{L}_1}$ as in Section \ref{subsubsec-DGA-map} and $J'_{\rho}$-holomorphic curves used to define $\Psi_{K_1}$.}\label{fig-moduli}
\end{figure}

For every $\lambda\in C'\calM_{\mathscr{L}_n,l}(a)$, we define $\bold{s}^{\lambda}=(s^{\lambda}_1,\dots ,s_l^{\lambda})\in \Sigma^l_{K_1}$ as follows:
\begin{itemize}
\item[(I-$l$)] If $\lambda= (b,u) \in \calM_{\mathscr{L}_n,l}(a)$, $s^{\lambda}_{i'}\colon [0,T^{\lambda}_{i'}]\to \bR^3$ for $i'\in \{1,\dots ,l\}$ is the curve
\[F_{1-t(b)} \circ \rest{u}{\overline{\partial_{2i'}D_{2l+1}}}\]
pre-composed with a parametrization $\varphi^{\lambda}_{i'}\colon [0,T^{\lambda}_{i'}]\to \overline{\partial_{2i'}D_{2l+1}}$ such that $\varphi^{\lambda}_{i'}(0)=p_{2i'-1}$ and $|\dot{s}^{\lambda}_{i'}(t)|=1$ for every $t$.
In particular, $T^{\lambda}_{i'}$ is the length of the curve $\rest{u}{\overline{\partial_{2i'}D_{2l+1}}}$.
Such $\varphi^{\lambda}_{i'}$ uniquely exists since $\rest{u}{\overline{\partial_{2i'}D_{2l+1}}}$ is an immersion by Lemma \ref{lem-jet-transv}.
Note that
\[s^{\lambda}_{i'}(\{0,T^{\lambda}_{i'}\})= F_{1-t(b)}(\{u(p_{2i'-1}),u(p_{2i'})\}) \subset F_{1-t(b)}(K_{t(b)})= K_1.\]
\item[(II-$l$)] If $\lambda =((b,u),z) \in \delta \widetilde{\calM}_{\mathscr{L}_n,l-1}^{2i}(a)$, where $z$ lies in $\partial_{2i}D_{2l+1}$,
$s^{\lambda}_{i'}$ for $i'\in \{1,\dots ,l\}$ is the curve
\[ \begin{cases}
F_{1-t(b)}\circ \rest{u}{\overline{\partial_{2i'}D_{2l-1}}} & \text{ if }i'\leq i-1 , \\
F_{1-t(b)}\circ \rest{u}{\overline{\partial_{2i}D_{2l-1}}^0} & \text{ if }i'=i ,\\
F_{1-t(b)}\circ  \rest{u}{\overline{\partial_{2i}D_{2l-1}}^1} & \text{ if }i'=i+1, \\
F_{1-t(b)}\circ  \rest{u}{\overline{\partial_{2i'-2}D_{2l-1}}} & \text{ if }i'\geq i+2 ,
\end{cases}\]
pre-composed with a parametrization such that $s^{\lambda}_{i'}$ has the unit speed.
Here, $\overline{\partial_{2i}D_{2l-1}}^0 \subset \overline{\partial_{2i}D_{2l-1}}$ is the arc whose endpoints are $\{p_{2i-1}, z\}$, and 
$\overline{\partial_{2i}D_{2l-1}}^1 \subset \overline{\partial_{2i}D_{2l-1}}$ is the arc whose endpoints are $\{z, p_{2i}\}$.
Note that $u(z) \in K_{t(b)}$ and thus $F_{1-t(b)}(u(z))\in K_1$.
If a sequence $(\lambda_{\nu})_{\nu=1,2,\dots}$ in $\calM_{\mathscr{L}_n,l}(a)$ converges to $\lambda$ in $C'\calM_{\mathscr{L}_n,l}(a)$, $\bold{s}^{\lambda_{\nu}}$ defined as in (I-$l$) converges to $\bold{s}^{\lambda}$ in $\Sigma^l_{K_1}$ as $\nu\to \infty$. See Figure \ref{fig-building} in the case where $l=2$ and $i=1$.
\end{itemize}

We take a subdivision of the compact $1$-dimensional manifold $C'\calM_{\mathscr{L}_n,l}(a)$ into intervals. Let $I_{n,l}(a)$ denote the set of the intervals and fix an identification $\Delta \cong [0,1]$ for each $\Delta \in I_{n,l}(a)$.
Define $H(a)=(H_l(a))_{l=0,1,\dots} \in \bigoplus_{l=0}^{\infty} C_1(\Sigma^l_{K_1})$ by
\[H_l(a) = \sum_{\Delta\in I_{n,l}(a)}(\bold{s}^{\lambda})_{\lambda \in \Delta}  . \]
for $l\geq 1$ and $H_0(a)=0$ (recall that $C_1(\Sigma^0_{K_1}) =0$ by definition).
This makes sense since $(\bold{s}^{\lambda})_{\lambda\in \Delta}$ for any $\Delta \in I_{n,l}(a)$ satisfies the conditions (1a), (1b) and (1c) by (iii), (i) and (iv) of Lemma \ref{lem-jet-transv} respectively.
\begin{remark}
More precisely, we need to take a subdivision so that $\bold{s}^{\lambda}$ for $\lambda\in \partial \Delta\cong \{0,1\}$ satisfies (0a) and (0b).
This is possible due to the transversality of the evaluation map by Lemma \ref{lem-jet-transv} (iii). 
\end{remark}

The next lemma is a key to prove Proposition \ref{prop-chain-homotopy}.
\begin{lemma}\label{lemma-delta-Hl}
For each $i\in \{1,\dots ,l\}$
\[ \delta_i H_l(a) = \sum_{\lambda' \in \delta\widetilde{\calM}_{\mathscr{L}_n,l}^{2i}(a)} \bold{s}^{\lambda'}\]
in $\bar{C}_0(\Sigma^{l+1}_{K_1})$, where $\bold{s}^{\lambda'}$ in the right-hand side is defined by (II-$(l+1)$) for $\lambda' \in \delta\widetilde{\calM}_{\mathscr{L}_n,l}^{2i}(a)\subset C' \calM_{\mathscr{L}_n,l+1}(a)$.
\end{lemma}
\begin{proof}
$\delta_i H_l(a) $ is given by the summand $\bold{s}^{(\lambda,t),i}$ of (\ref{split-path}) for all pairs $(\lambda,t)$ in the set
\[ \Gamma_i^{-1}(K_1) = \{(\lambda,t) \in \calM_{\mathscr{L}_n,l}(a)\times \R \mid 0<t<T^{\lambda}_i,\  s_i^{\lambda}(t) \in K_1 \}. \]
(Here, note that if $\lambda \in \calM_{\mathscr{L}_n,l}(a)\setminus C'  \calM_{\mathscr{L}_n,l}(a)$, we may ignore $\bold{s}^{(\lambda,t),i}$ since it belongs to $\Sigma^{l+1}_{K_1,\epsilon_1}$ by (\ref{length-short}).)
This set is naturally identified with $\delta\widetilde{\calM}_{\mathscr{L}_n,l}^{2i}(a)$ by a bijection
\[ P_{l,i} \colon \Gamma_i^{-1}(K_1) \to  \delta\widetilde{\calM}_{\mathscr{L}_n,l}^{2i}(a) \colon (\lambda,t)  \mapsto (\lambda,\varphi^{\lambda}_i(t)), \]
where $\varphi^{\lambda}_i\colon [0,T^{\lambda}_i]\to \overline{\partial_{2i}D_{2l+1}}$ is a parametrization used in (I-$l$).
By the process of (II-($l+1$)), $\bold{s}^{P_{l,i}(\lambda,t)}\in \Sigma^{l+1}_{K_1}$ is defined for $P_{l,i}(\lambda,t) \in \delta\widetilde{\calM}_{\mathscr{L}_n,l}^{2i}(a)$.
It is obtained from $\bold{s}^{\lambda}$ by splitting the $i$-th path $s^{\lambda}_i$ at the point $t\in [0,T^{\lambda}_i]$ where $s^{\lambda}_i$ intersects $K_1$.
This is exactly the process to construct $\bold{s}^{(\lambda,t),i}$ from $\bold{s}^{\lambda}$ as in (\ref{split-path}).
Therefore,
$ \bold{s}^{(\lambda,t),i} = \bold{s}^{P_{l,i}(\lambda,t)}$. This shows that
\[ \delta_i H_l(a) = \sum_{(\lambda,t)\in \Gamma_i^{-1}(K_1)} \bold{s}^{P_{l,i}(\lambda,t)} =  \sum_{\lambda' \in \delta\widetilde{\calM}_{\mathscr{L}_n,l}^{2i}(a)} \bold{s}^{\lambda'}\]
in $\bar{C}_0(\Sigma^{l+1}_{K_1})$.
\end{proof}

In addition to $H(a)$, let us prepare another $1$-chain.
Consider $\lambda=\mathcal{G}_{n,l}(\bold{u})\in \calM_{L_{b_n},l}(a)$ for any
\[\bold{u} = (u, (u_1,\dots ,u_l)) \in \calM_{\widetilde{L}_1,J}(a;a_1,\dots,a_m) \times \prod_{k=1}^m\calM_{L_{K_1},l_k}(a_k) ,\]
where $a_1,\dots ,a_m\in \mathcal{R}(\Lambda_{K_1})$ and $l_1,\dots ,l_m\in \Z_{\geq 0}$ satisfy $|a_1|=\dots =|a_m|=0$ and $l_1+\dots +l_m=l$.
For $u_k\in\calM_{L_{K_1},l_k}(a_k)$, a $0$-chain $\bold{s}^{u_k}\in \bar{C}_0(\Sigma^{l_k}_{K_1})$ is defined by (\ref{path-si}) when $l_k\geq 1$.
Let us put $\bold{s}^{u_k}\coloneqq 1\in\Z_2 = \bar{C}_0(\Sigma^{l_k}_{K_1})$ when $l_k=0$.
Then, Lemma \ref{lem-bijection} implies that as $n\to \infty$, $\bold{s}^{\lambda}= \bold{s}^{\mathcal{G}_{n,l}(u)}$ converges in $\Sigma^l_{K_1}$ to a tuple of paths representing $\bold{s}^{u_1} * \cdots * \bold{s}^{u_m}\in \bar{C}_0(\Sigma^l_{K_1})$.
Here, $*$ is the product on $\bar{C}_0(\Sigma^l_{K_1})$ induced by (\ref{product-map}).
Therefore, when $n$ is sufficiently large, we can apply Lemma \ref{lem-homologous} to get a $1$-chain
$h_{\lambda}\in C_1(\Sigma_{K_1}^l)$ such that $\delta (h_{\lambda})=0$ and 
\[ D_{K_1}h_{\lambda} = \partial h_{\lambda} = \bold{s}^{\lambda} - \bold{s}^{u_1} * \cdots * \bold{s}^{u_m}. \]
Let us define $h(a)=(h_l(a))_{l=0,1,\dots}\in \bigoplus_{l=0}^{\infty} C_1(\Sigma^l_K)$ by $h_0(a)=0$ and
\[ h_{l}(a) = \sum_{\lambda \in \calM_{L_{b_n},l}(a) } h_{\lambda}  \]
for $l\geq 1$.
We remark that $\delta h(a)=0$ in $\bar{C}_0(\Sigma^{l+1}_{K_1})$.

\begin{proof}[Proof of Proposition \ref{prop-chain-homotopy}]
Consider a diagram
\[ \xymatrix{
\mathcal{A}_0(\Lambda_{K_0}) \ar[d]_-{\Psi'_{K_0}} \ar[r]^-{\Phi'_{\widetilde{L}_1}} & \mathcal{A}_0(\Lambda_{K_1}) \ar[d]_-{\Psi'_{K_1}} \\
\bar{C}_0(\Sigma_{K_0}) \ar[r]^-{\bar{F}_1^{(0)}} & \bar{C}_0(\Sigma_{K_1}). 
}\]
Those maps in the diagram preserve the $\Z_2$-algebra structures, and the subspace $\Image D_{K_1}\subset \bar{C}_0(\Sigma_{K_1})$ is a two-sided ideal with respect to the product $*$.
Therefore, it suffices to show that for every $a\in\mathcal{R}(\Lambda_{K_0})$ with $|a|=0$,
\begin{align}\label{eq-goal}
\Psi'_{K_1}\circ \Phi'_{\widetilde{L}_1}(a) - \bar{F}_1^{(0)} \circ \Psi'_{K_0}(a) = D_{K_1} \left( h(a) + H (a)\right)
\end{align}
in $\Z_2$-coefficient, in particular, the left-hand side is contained in $\Image D_{K_1}$.
For each $l\in \Z_{\geq 0}$, let $\left( \Psi'_{K_1}\circ \Phi'_{\widetilde{L}_1}(a) \right)_l\in \bar{C}_0(\Sigma^l_{K_1})$ denote the $l$-th component of $\Psi'_{K_1}\circ \Phi'_{\widetilde{L}_1}(a) \in \bar{C}_0(\Sigma_{K_1})$.
Then, (\ref{eq-goal}) is rewritten as the equations
\begin{align}\label{final-equation}
\begin{split}
\left( \Psi'_{K_1}\circ \Phi'_{\widetilde{L}_1}(a) \right)_0 - \Psi'_{K_0,0}(a)  & =0  \text{ in } \bar{C}_0(\Sigma^0_{K_1}) = \Z_2 , \\
 \left( \Psi'_{K_1}\circ \Phi'_{\widetilde{L}_1}(a) \right)_l - \bar{F}_1^{(0)} (\Psi'_{K_0,l}(a)) & = \partial h_l(a) + \partial H_l(a) + \delta H_{l-1}(a)  \text{ in } \bar{C}_0(\Sigma^l_{K_1}),
 \end{split}
\end{align}
for $l\in \Z_{\geq 1}$.

The first equation of (\ref{final-equation}) follows from (\ref{conut-M0}) since
\begin{align*} 
& \left( \Psi'_{K_1} \circ \Phi'_{\widetilde{L}_1} (a)\right)_0  \\
=\  & \#_{\Z_2} \calM_{\widetilde{L}_1,0}(a) \cdot \Psi'_{K_1}(1) +  \sum_{m\geq 1,\ |a_1|= \dots =|a_m|=0} \left( \#_{\bZ_2}  \calM_{\widetilde{L}_1,J}(a;a_1,\dots,a_m)\right) \cdot \left(  \prod_{k=1}^m \Psi'_{K_1,0}(a_k) \right) \\
=\  & \#_{\Z_2} \calM_{\widetilde{L}_1,0}(a) +  \sum_{m\geq 1,\ |a_1|= \dots =|a_m|=0} \left( \#_{\bZ_2}  \calM_{\widetilde{L}_1,J}(a;a_1,\dots,a_m)\right) \cdot \left(  \prod_{k=1}^m \#_{\Z_2} \calM_{L_{K_1},0}(a_k) \right) \\
= \ & \#_{\bZ_2} \calM_{L_{K_0},0}(a) 
= \Psi'_{K_0,0}(a)
\end{align*}
in $\bar{C}_0(\Sigma^0_{K_1}) = \bZ_2$.

Let us prove the second equation of (\ref{final-equation}) for $l\geq 1$.
Since $C'\calM_{\mathscr{L}_n,l}(a) = \bigcup_{\Delta \in I_{n,l}(a)} \Delta$ has the boundary (\ref{boundary-C'M}),
\begin{align}\label{boundary-chains} 
\begin{split}
\partial (H_l(a)) & = \sum_{\lambda \in \partial \left( C'\calM_{\mathscr{L}_n,l}(a)\right) } \bold{s}^{\lambda} \\
&= \sum_{\lambda \in \calM_{L_{K_0},l}(a)} 
\bold{s}^{\lambda} + \sum_{\lambda \in \calM_{L_{b_n},l}(a)}\bold{s}^{\lambda} + \sum_{i=1}^{l-1} \sum_{\lambda \in \delta \widetilde{\calM}_{\mathscr{L}_n,l-1}^{2i}(a)} \bold{s}^{\lambda} .
\end{split}
\end{align}
in $\bar{C}_0(\Sigma^l_{K_1})$.
Here, note that $\bold{s}^{\wh{\lambda}}$ vanishes in $\bar{C}_0(\Sigma^l_{K_1})$ for any $\lambda \in \delta \widetilde{\calM}_{\mathscr{L}_n,l-1}^{2i-1}(a)$
since the $i$-th path in $\bold{s}^{\wh{\lambda}}$ has length less than $\epsilon_1$ by (\ref{length-short}), which means that $\bold{s}^{\wh{\lambda}}$ is contained in $\Sigma^l_{K_1,\epsilon_1}$.

By (I-$l$) for $\lambda=(0,u)\in \calM_{L_{K_0},l}(a)$, the path $s^{\lambda}_{i'}$ agrees with $F_1\circ \rest{u}{\overline{\partial_{2i'}D_{2l+1}}}$ up to parametrization.
Therefore,
\[ \sum_{\lambda \in \calM_{L_{K_0},l}(a)} \bold{s}^{\lambda} = \sum_{u \in \calM_{L_{K_0},l}(a)} \bar{F}_1^{(0)} (\bold{s}^u)= \bar{F}_1^{(0)} \circ \Psi'_{K_0,l} (a) , \]
where $\bold{s}^u\in \bar{C}_0(\Sigma^l_{K_0})$ for $u\in \calM_{L_{K_0},l}(a)$ is given by (\ref{path-si}).
In addition, from the construction of $h_l(a)\in C_1(\Sigma^l_K)$ as above,
\begin{align*} 
&\left( \Psi'_{K_1} \circ \Phi'_{\widetilde{L}_1} (a)\right)_l \\
= & \sum_{|a_1|= \dots =|a_m|=0,\ l_1+\dots +l_m=l} \#_{\bZ_2}  \calM_{\widetilde{L}'_1,J}(a;a_1,\dots,a_m) \cdot \Psi'_{K_1,l_1}(a_1)*\dots * \Psi'_{K_1,l_m}(a_m) \\
= & \sum_{|a_1|= \dots =|a_m|=0,\ l_1+\dots +l_m=l} \#_{\bZ_2}  \calM_{\widetilde{L}'_1,J}(a;a_1,\dots,a_m) \cdot \sum_{(u_1,\dots ,u_m)\in \prod_{k=1}^m\calM_{L_{K_1},l_k}(a_k)} ( \bold{s}^{u_1} * \cdots * \bold{s}^{u_m}) \\
= &  \sum_{\lambda \in \calM_{L_{b_n},l}(a) } ( \bold{s}^{\lambda} + \partial h_{\lambda} ) =  \left( \sum_{\lambda \in \calM_{L_{b_n},l}(a) } \bold{s}^{\lambda}  \right) + \partial h_{l}(a).
\end{align*}
Finally, for each $i\in \{1,\dots ,l-1\}$,
\[ \sum_{\lambda \in \delta \widetilde{\calM}_{\mathscr{L},l-1}^{2i}(a)} \bold{s}^{\lambda}  = \delta_{i} (H_{l-1} (a))\]
by Lemma \ref{lemma-delta-Hl}.
Therefore, we obtain from (\ref{boundary-chains})
\[\partial (H_l(a)) = \bar{F}_1^{(0)} \circ \Psi'_{K_0,l} (a) +  \left( \left( \Psi'_{K_1} \circ \Phi'_{\widetilde{L}_1} (a)\right)_l - \partial h_l(a) \right) +  \delta_{i} (H_{l-1} (a)), \]
and this agrees with the second equation of (\ref{final-equation}).
This finishes the proof of (\ref{eq-goal}).
\end{proof}

\section{Examples of detecting non-contractible loops}\label{sec-examples-of-detection}

\subsection{Loops of knots and Legendrian tori}
\label{sec:mapping_spaces}

Motivated by the work of \cite{Kalman, Casals, martinez2024legendrian}, we define the blue-box loops.

\begin{definition}\label{def-blue-box}
Consider a knot $ L_1\#L_2\subset\bR^3$, where the knot $L_1$ is oriented and equipped with fixed framing, and the $L_2$ -tangle is arbitrarily small. Next, let the blue box $B[L_2]$ be a small part of a tubular neighborhood $N_{L_1}$ of $L_1$ that contains a $L_2$-tangle. Then the \textit{blue-box loop} is a smooth family of knots $( K^{\mathrm{box}}_t)_{t\in[0, 1]}$ such that $K^{\mathrm{box}}_0=K^{\mathrm{box}}_1$ and is realized by passing the blue box $B[L_1]$ rigidly inside $N_{L_1}$ once around $L_1$. By a rigid motion we mean that the track of the blue box respects the framing of $L_1$. See Figure \ref{figure_blue_box_loop}.
\end{definition}
\begin{figure}[ht]
\centering
\hspace{-0.5cm}
\includegraphics[scale=0.90]{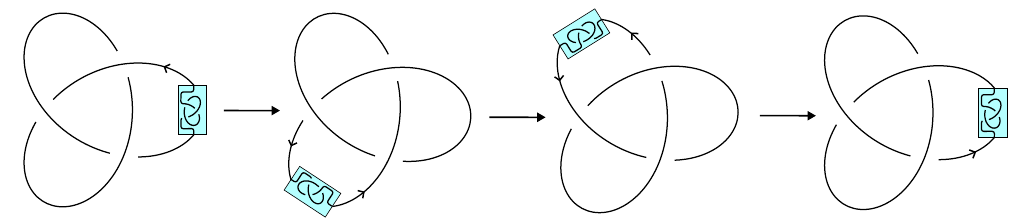}
\vspace{0.1cm}
\caption{A blue-box loop of the square knot $T_{2, 3}\#T_{2, -3}$.}
\label{figure_blue_box_loop}
\end{figure}

The unit conormal lift $K\mapsto \Lambda_K$ induces from the blue-box loop of knots $(K_t^{\mathrm{box}})_{t\in [0,1]}$ a loop of Legendrian tori $(\Lambda_{K_t^{\mathrm{box}}})_{t\in [0,1]}$.
We note that the loop $(\Lambda_{K_t^{\mathrm{box}}})_{t\in [0,1]}$ is realized by the trajectory of the domain $U^* B[L_2]$ inside $U^*\R^3$.

The next proposition shows that
if we forget the contact structure of $U^*\R^3$, $(\Lambda_{K_t^{\mathrm{box}}})_{t\in [0,1]}$ is contractible by a homotopy of loops of $C^{\infty}$ submanifolds.

\begin{proposition}\label{lemma_loop_lift}
$[(\Lambda_{K_t^{\mathrm{box}}})_{t\in [0,1]}]=[e]$ in $ \pi_{\sm}(\Lambda_{K_0})$, where $e$ is a constant loop.
\end{proposition}

\begin{proof}
By the construction, for every $t\in[0, 1]$ the blue box has the same form. Therefore, it suffices to show that there is a smooth isotopy on $U^* B[L_2]$ fixed near the boundary which deforms the conormal of the $L_2$-tangle to the conormal of the untangle. To achieve this, we consider a planar diagram of $L_2$ and focus on each crossing.
Focusing on a neighborhood $B$ of one crossing, the problem is reduced to constructing a smooth isotopy from $\Lambda_{L_2}\cap U^*B$ to $\Lambda_{L'_2}\cap U^*B$, where $L'_2$ is obtained by changing the overcrossing in $B$ to the undercrossing. 

%

We may assume that there are local coordinates $(x; y)$ on $U^\ast \bR^3\subset T^\ast\mathbb{R}^3$ such that the crossing has the following form. It consists of strands $c_1(\ell_1)=(\ell_1, 0, 0;\, 0, 0, 0)$ and $c_1(\ell_2)=(0, \ell_2, \varepsilon\phi(\ell_2);\, 0, 0, 0)$, where $\ell_1, \ell_2\in[-1, 1],$ $\varepsilon>0$ is small and $\phi$ is a bump function supported near $0$. Then, after a small isotopy of fibers, the unit conormal bundle of the tangle is the union of two cylinders $\{C_i(\ell_i,\beta_i)\mid \ell_i\in [-1,1],\ \beta_i\in S^1\}$ for $i\in \{1,2\}$, where  
\[\begin{array}{cc} C_1(\ell_1, \beta_1)=(\ell_1, 0, 0;\, 0, \cos{\beta_1}, \sin{\beta_1}) , & C_2(\ell_2, \beta_2)=(0, \ell_2, \varepsilon\phi(\ell_2);\, \cos{\beta_2}, 0, \sin{\beta_2}). \end{array}\]
Consider an isotopy in $t\in [0,1]$
\[\begin{array}{cc} C^t_1(\ell_1,\beta_1) = (\ell_1,0,0; c_t(\ell_1, \beta_1)) , &  C^t_2 (\ell_2,\beta_2) = (0,\ell_2, (1-2t) \cdot \epsilon\phi(\ell_2); \cos \beta_2,0, \sin \beta_2 ), \end{array}\]
where $c_t(\ell_1,\cdot )\colon S^1\to S^2$ is an embedding which coincides with $(0,\cos\beta_1,\sin\beta_1)$ if $\ell_1\in \{1,-1\}$ or $t\in \{0,1\}$, and is disjoint from $\{(\cos\beta_2,0,\sin \beta_2)\mid \beta_2\in S^1\}$ if $\ell_1=0$ and $t=\frac{1}{2}$. (For example, $c_{\frac{1}{2}}(0,\beta) = (\frac{1}{\sqrt{2}}\cos\beta_1, \frac{1}{\sqrt{2}},\frac{1}{\sqrt{2}}\sin\beta_1)$.)
Then, the two cylinders $\{C^t_i(\ell_i,\beta_i)\mid \ell_i\in [-1,1],\ \beta_i\in S^1\}$ for $i=1,2$ are disjoint for every $t\in [0,1]$, and this gives a desired isotopy since $C^1_1(\ell_1,\beta_1)=C_1(\ell_1,\beta_1)$ and $C^1_2(\ell_2,\beta_2) = (0,\ell_2,-\epsilon\phi(\ell_2);\cos\beta_2,0,\sin\beta_2)$.
%
%
\end{proof}

\begin{remark}\label{rem_embedding_spaces}Let us briefly discuss the description of the blue-box loop in the setting of embedding spaces, referring to the work by Hatcher \cite{Hatcher83, Hatcher2002TopologicalMS} and Budney \cite{Budney2006, Budney2010}. See also a survey in Martínez \cite{martinez2024legendrian}.
To simplify the discussion, we consider knots in the unit $3$-sphere $S^3$.



Let $\Emb(S^1, S^3)$ be a space of smooth (parametrized) knots in $S^3$ equipped with $C^\infty$ topology. Similarly, we introduce the space of long knots $\LEmb(\bR, \bR^3)$, i.e. the space of embeddings $f:\bR\rightarrow\bR^3$ such that $f(t)=(t, 0, 0)$ if $\vert t\vert>1$. 
Fix the north pole $p_0 =(0,0,0,1)\in S^3$ and a unit vector $V_0=(1,0,0,0)\in T_{p_0}S^3\subset \R^4$, and fix an identification $S^3\setminus \{p_0\}\cong \R^3$ by  the stereographic projection.
Then, up to homotopy, $\LEmb(\bR, \bR^3)$ is identified with a subspace of $\Emb(S^1, S^3)$ that consists of embeddings $f:\bR/\bZ\rightarrow S^3$ such that $f(0) = p_0$ and $\dot{f}(0)/|\dot{f}(0)| =V_0$.
Consequently, the map $e_0\colon \Emb(S^1,S^3)\to V_2(\R^4) \colon f\mapsto (f(0),\dot{f}(0)/|\dot{f}(0)|)$ induces a fibration sequence
\begin{equation}\label{eqn_fibre_knots}
\xymatrix{
\LEmb(\bR, \bR^3) \ar[r]^-{i_0} & \Emb(S^1, S^3) \ar[r]^-{e_0} & V_2(\bR^4)=\SO(4)/\SO(2).}
\end{equation}
%
We remark that the homotopy type of the space $\Emb(S^1, \bR^3)$ can also be described by a similar fibration involving $\LEmb(\R,\R^3)$. See \cite[Proposition 2.2]{Budney2006}.

The topological type of $\LEmb(\bR, \bR^3)$ has been intensively studied. 
Notably, as a consequence of Smale conjecture, Hatcher \cite{Hatcher2002TopologicalMS} showed that connected components of $\LEmb(\bR, \bR^3)$ are $K(\pi, 1)$.
For any knot $L\subset S^3$, let $\mathcal{K}_L$ denote the component of $\LEmb(\R,\R^3)$ containing $L$-tangles.
As a special case of Budney's formula \cite{Budney2010}, we have the following:
Let $L_1$ and $L_2$ be prime knots which are non-isotopic to each other.
By  \cite[Theorem 11]{Budney2007} and \cite[Theorem 2.4]{Budney2010}, we have a homotopy equivalence
\begin{align}\label{connect-sum-homotopy}
\mathcal{K}_{L_1\#L_2} \simeq  \mathcal{C}_2(\R^2)\times \mathcal{K}_{L_1}\times \mathcal{K}_{L_2},  
\end{align}
where $\mathcal{C}_2(\R^2)$ is the configuration space of distinct two points in $\R^2$.
(A generator of $\pi_1(\mathcal{C}_2(\R^2))$ gives a loop such that a small $L_1$-tangle slides along $L_2$, and then a small $L_2$-tangle slides along $L_1$.)
For any knot $L$, there are two types of loops in $\mathcal{K}_{L}$ described in \cite[page 3]{Hatcher2002TopologicalMS}:
\begin{itemize}
\item One loop is obtained by twisting a long knot around the long axis.
\item The other loop has the form $f(S^1\setminus \{t\})\subset S^3\setminus \{ f(t)\}$, where $f\colon S^1\to S^3$ is an embedding with a framing such that $f(S^1)=L$.
More precisely, a loop $(A_t)_{t\in [0,1]}$ in $\SO(4)$ which maps $(f(t),\dot{f}(t)/|\dot{f}(t)|)$ to $(p_0,V_0)\in T S^3$ is determined by the framing of $f$ at $t$ (see \cite[Section 4.2]{martinez2024legendrian}), and then we obtain a long knot $(A_t\circ f)(S^1)\setminus \{p_0\}\subset S^3\setminus \{p_0\}\cong \R^3$.
\end{itemize}
Following \cite[Section 4]{martinez2024legendrian}, let us call them a \textit{Gramian's} loop and a \textit{Hatcher-Fox} loop respectively.
It is worth noting that
in the case of a torus knot $T_{p,q}$,  $\mathcal{K}_{T_{p,q}}$ has the same homotopy type as $S^1$ \cite[page 3]{Hatcher2002TopologicalMS} and in particular, a Gramain's loop and a Hatcher-Fox loop coincide up to homotopy.
%

Let $f\in \Emb(S^1,S^3)$ be an embedding with a fixed framing such that $f(S^1)=L_1$. This induces a Hatcher-Fox loop $\ell_{\mathrm{HF}}$ for $L=L_1$ as above, together with a loop $(A_t)_{t\in [0,1]}$ in $\SO(4)$.
(The homotopy class $[(A_t)_{t\in [0,1]}]\in \pi_1(\SO(4))$ is invariant under homotopies of framed immersions $S^1\looparrowright S^3$.)
Consider a segment $[-\epsilon,0]\subset \R/\Z \cong S^1 $ for $0<\epsilon \ll 1$.
Our blue-box loop is given by a family $(g_t)_{t\in [0,1]}$ of embeddings $g_t\colon S^1 \to S^3$ such that $g_t(s)= f(t+s)$ if $s\in S^1\setminus [-\epsilon,0]$, and $\rest{g_t}{[ -\epsilon, 0]}$ is an embedding onto the $L_2$-tangle in the blue box.
Note that $g_t(0)=f(t)$ is a point just after the $L_2$-tangle.
Then, $e_0(A_t\circ g_t) =  \left( A_t ( f(t) ), A_t (\dot{f}(t)/|\dot{f}(t)|) \right)   = (p_0,V_0)$, and thus $(A_t\circ g_t)_{t\in [0,1]}$ is identified with a loop in $\LEmb(S^1,S^3) \simeq e_0^{-1}(p_0,V_0)$.
Up to homotopy, this loop has the form $(g'_t)_{t\in [0,1]}$ of
\[g'_t\colon \R\to \R^3 \colon u \mapsto \begin{cases} g^1_t(u) & \text{ if } u\leq 0, \\
g^2(u) & \text{ if }u\geq 0,
\end{cases}\]
where $(g^1_t)_{t\in[0,1]}$ is a loop in $\mathcal{K}_{L_1}$ homotopic to $\ell_{\mathrm{HF}}$, and $g^2\in \mathcal{K}_{L_2}$ does not depend on $t$.
In other words, the homotopy class $[(A_t\circ g_t)_{t\in [0,1]}] \in \pi_1(\Emb(S^1,S^3), A_0\circ g_0)$ is equal to the image of $(1, [\ell_{\mathrm{HF}}], 1)$ under the homomorphism
\[ (i_0)_*\colon \pi_1(\mathcal{C}_2(\R^2))\times \pi_1(\mathcal{K}_{L_1},g^1_0)\times \pi_1(\mathcal{K}_{L_1},g^2) \to \pi_1 ( \Emb(S^1,S^3),A_0\circ g_0). \]
Here, we use the homotopy equivalence (\ref{connect-sum-homotopy}) and the map $i_0$ in (\ref{eqn_fibre_knots}).
\end{remark}

\subsection{Describing the automorphism induced by a blue-box loop}\label{subsec-automorphism}
In order to find a non-contractible loop of Legendrian tori, it is enough by Theorem \ref{thm-1} to find a loop of knots that induces a non-trivial automorphism on the cord algebra.
The natural candidate for us is the blue-box loop from Section \ref{sec:mapping_spaces}.

First, we are going to describe explicitly how the blue-box loop acts on the cord algebra.
Instead of $\Cord(K)$ from Definition \ref{def-cord-Z/2}, it will be more convenient to consider a loop interpretation of the cord algebra $\Cord^{\mathrm{loop}}(K)$ which is given as follows.
It is obtained from \cite[Definition 2.1]{Ng} by reducing the coefficient ring from $\Z[H_1(\partial N)]$ to $\Z_2$.

Recall that $K\subset \R^3$ is an oriented knot with a framing, and $K^\prime\subset N$ is the shift of $K$ with respect to the framing.
Fix a base point $q_0\in K$, ane let $q'_0\in K'$ be the corresponding point.
Let $m$ and $\ell$ be the meridian and the longitude of $K$ respectively. 

\begin{definition}
$\Cord^{\mathrm{loop}}(K)$ is a unital non-commutative $\bZ_2$-algebra generated by the set $\pi_{1}(\mathbb{R}^3\setminus K, q_0^\prime)$ modulo the relations
\begin{itemize}
\item[(i)]$[e]=0$, where $e$ is the trivial loop in $\pi_1(\bR^3\setminus K,q'_0),$
\item[(ii)]$[\ell\bullet\gamma]=[\gamma\bullet \ell]=[\gamma]$ for any $\gamma\in\pi_1(\bR^3\setminus K,q'_0),$
\item[(iii)]$[\gamma_1\bullet\gamma_2]+[\gamma_1\bullet m\bullet\gamma_2]=[\gamma_1] \cdot [\gamma_2]$ for any $\gamma_1, \gamma_2\in\pi_1(\bR^3\setminus K,q'_0).$
\end{itemize}    
\end{definition}
We remark that
\begin{align}\label{meridian-vanish}
[m\bullet \gamma]=[\gamma\bullet m]=[\gamma]\end{align}
holds by (i) and (iii).

\begin{remark}\label{rem-cord-and-loop}
Consider an equivalence relation $\sim$ on $\pi_1(\R^3\setminus K',q'_0)$ generated by $[\ell\bullet \gamma] \sim [\gamma] $ and $ [\gamma\bullet \ell] \sim [\gamma]$.
Then, there is a natural bijection between the quotient set
$\pi_1(\R^3\setminus K,q'_0)/\sim $ and $\mathcal{P}_K$, and
it induces a canonical isomorphism $\Cord^{\mathrm{loop}}(K)\to \Cord(K)$ which maps  generators $[\gamma]\in \pi_1(\R^3\setminus K,q'_0)$ to $[\gamma]\in \mathcal{P}_K$.
See also \cite[Proposition 2.5]{Ng}.
For readers' convenience, let us summarize several versions of cord algebra which appear in Section \ref{sec-examples-of-detection}.
\[\xymatrix@R=10pt@C=45pt{
\Cord^{\mathrm{nc}}(K) \ar[r]^-{\lambda=\mu=1} & \Cord(K) \ar[r]^-{ [\gamma] = [\gamma^{-1}]}_{\text{abelianization}} \ar@{=}[d]_-{\vsim} & \Cord^{\mathrm{ab}}(K) \\ 
& \Cord^{\mathrm{loop}}(K). & 
}\]
$\Cord^{\mathrm{nc}}(K)$ will appear only in the last Remark \ref{rem-potential}.
$\Cord^{\mathrm{ab}}(K)$ is a commutative algebra which is suitable for concrete computations.
\end{remark}

Let $K=L_1\# L_2$ be an oriented framed knot such that $L_2$ is in the blue box $B[L_2]\subset N_{L_1}$.
We fix a base point $q'_0\in K' \cap \partial B[L_2]$.
For $i\in\{0,1\}$, let $m_i, \ell_i\in \pi_1(\R^3\setminus L_i,q'_0)$ denote the meridian and the longitude of $L_i$, respectively.
(Hereafter, we omit writing the base point $q'_0$.)
Note that by Seifert-van Kampen theorem
\begin{equation}\label{eqn_van_Kampen}
\pi_1(\bR^3\setminus K)\cong\frac{\pi_1(\bR^3\setminus L_1)\ast \pi_1(\bR^3\setminus L_2)}{\langle [m_1]=[m_2]\rangle},
\end{equation}
and the meridian and longitude of $K$ are $m=m_1=m_2$ and $\ell = \ell_1\bullet \ell_2 = \ell_2\bullet\ell_1$.

Now, in the spirit of Section \ref{sec:isom_cord_alg}, the blue-box loop determines an ambient isotopy $(F_t)_{t\in[0, 1]}$. Put $F^{\mathrm{box}}\coloneqq F_1$.
Then, $F^{\mathrm{box}}$ induces on $\pi_1(\R^3\setminus K)$ a group automorphism $(F^{\mathrm{box}})_\ast$.
By (\ref{eqn_van_Kampen}), this can be described on generators as
\begin{align}
\begin{split}\label{eqn_autom_gener}
    (F^{\mathrm{box}})_\ast [a]&=[\ell_1^{-1}\bullet a\bullet\ell_1], \hspace{0.4cm}\hbox{ where }[a]\in\pi_1(\bR^3\setminus L_1) ,\\
    (F^{\mathrm{box}})_\ast [b]&=[b], \hspace{1.35cm}\hbox{ where }[b]\in\pi_1(\bR^3\setminus L_2).
\end{split}
\end{align}
The group automorphism $(F^{\mathrm{box}})_*$ on $\pi_1(\R^3\setminus K)$ fixes $m$ and $\ell$,
so an automorphism on $\Cord^{\mathrm{loop}}(K)$ is determined.
Via the canonical isomorphism $\Cord^{\mathrm{loop}}(K)\cong \Cord(K)$,
this agrees with the automorphism $(F^{\mathrm{box}})_*$ defined on $\Cord(K)$ as in Section \ref{subsubsec-isom-on-cord}.



Before continuing with interesting examples, we discuss the computability of $\Cord(K)$.
From now on, we assume that $K$ has the blackboard framing $K'$ induced by a knot diagram.
We utilize a combinatorial way in \cite[Section 4.3]{Ng2003} to compute the cord algebra from a knot diagram.

\begin{remark}
To understand the combinatorial description, it will be more convenient to work with a slightly modified version of the cord algebra as in \cite[Definition 1.2]{Ng2003}.
It is generated by the set of homotopy classes of \textit{cords} of $K$, which are paths $c\colon [0,1]\to \R^3$ such that $c^{-1}(K) = \{0,1\}$. 
A canonical isomorphism from $\Cord(K)$ to this modified version is given by concatenating any path $\gamma\colon ([0,1],\{0,1\})\to (\R^3\setminus K,K')$ with the short segments along the shift from $K$ to $K'$ at the endpoints $\gamma(0)$ and $\gamma(1)$. 
\end{remark}

We reduce $\Cord(K)$ to the \textit{abelian cord algebra} $\Cord^{\mathrm{ab}}(K)$ by imposing additionally that cords are unoriented and commute each together.
More precisely, we mod out by relations
\[\begin{array}{cc} [\gamma]=[\gamma^{-1}], & [\gamma_1]\cdot [\gamma_2] = [\gamma_2]\cdot[\gamma_1], \end{array}\]
for any $[\gamma],[\gamma_1],[\gamma_2]\in \mathcal{P}_K$,
where $\gamma^{-1}$ is the inverse path of $\gamma$.
Note that $F^{\mathrm{box}}$ induces an automorphism $(F^{\mathrm{box}})_\ast$ on $\Cord^{\ab}(K)$, and it is intertwined with the the automorphism on $\Cord(K)$ by the projection $\Cord(K)\rightarrow \Cord^{\mathrm{ab}}(K)$. 

Now, referring to \cite[Section 4.3]{Ng2003}, let us see how to compute $\Cord^{\mathrm{ab}}(K)$ from the knot diagram.
The crossings divide the diagram into $n$ strands.
We label them by $1,\dots, n$.
Then, the cord $a_{ij}$ is taken as an (unoriented) path between the strands $i, j$ such that outside of its endpoints, $a_{ij}$ lies completely above the plane of the knot diagram.
Then,
\[\Cord^{\mathrm{ab}}(K)=\mathcal{A}_n/\mathcal{I}^{\mathrm{diag}},\]
where $\mathcal{A}_n$ is a polynomial ring with $\bZ_2$-coefficients generated by $a_{ij}$ for all $i, j\in \{1,\dots ,n\}$. Also $\mathcal{I}^{\mathrm{diag}}$ is an ideal generated by $a_{ii}$ and $a_{ij}+a_{ji}$ for $i,j\in\{1,\dots ,n\}$, and 
\begin{equation}\label{eqn_cross}
a_{\ell j}+ a_{\ell k}+a_{\ell i}\cdot a_{i j},
\end{equation}
where $\ell\in \{1,\dots, n\}$ and $(i; j, k)$ goes over all $n$ crossings of the knot diagram. In more detail, $i$ represents the overcrossing and $j, k$ are the undercrossings as in Figure \ref{figure_move}. We stress that $a_{ij}=a_{ji}$, that is, the cords are unoriented in $\Cord^{\mathrm{ab}}(K)$.
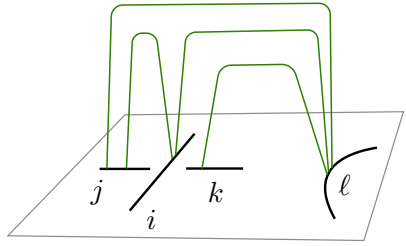
\begin{figure}
\labellist
\pinlabel $j$ at 75 707
\pinlabel $k$ at 130 707
\pinlabel $i$ at 100 694
\pinlabel $\ell$ at 190 710
\endlabellist
\centering
\includegraphics[scale=0.8]{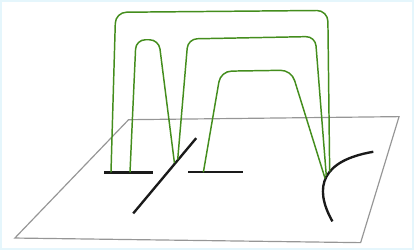}
\vspace{0.3cm}
\caption{The cords entering the relation given by (\ref{eqn_cross}).}
\label{figure_move}
\end{figure}

By definition, $\Cord^{\mathrm{ab}}(K)$ is generated by $n(n-1)/2$ elements in $\{a_{ij}\mid 1\leq i<j\leq n\}$. On the first sight, the number of the relations generating $\mathcal{I}^{\mathrm{diag}}$ might appear large, since (\ref{eqn_cross}) contribute by $2n^2$ relations.
However, this can be reduced.
If we apply relations (\ref{eqn_cross}) to the case $\ell=i$ and the case $\ell=j$, we obtain the  relations 
\begin{equation}\label{eqn_triv_rel}
a_{ik}=a_{ij}\hbox{ and }a_{kj}=a_{ij}^2.
\end{equation}
Furthermore, it follows that for any $\ell$, the relations (\ref{eqn_cross}) with $(i; j, k)$ and $(i; k, j)$ are the same.
Hence, in general, it remains to inspect only $n(n-3)$ relations contributing by (\ref{eqn_cross}): For each $\ell\in \{1,\dots ,n\}$, there are $n-3$ crossings in which the strand $\ell$ does not appear.

\subsection{Examples}\label{subsec-example}

Let us give examples of loops of knots $(K_t)_{t\in [0,1]}$ such that $\rho([(K_t)_{t\in [0,1]}])$ is a non-trivial automorphism on the cord algebra.

\begin{example}\label{exmpl_nontriv1}We consider the square knot $K\coloneqq T_{2, 3}\#T_{2, -3}$, where the $T_{2, -3}$ lies in the blue box. See Figure \ref{figure_square_knot}.
\begin{figure}
\labellist
\pinlabel $1$ at 43 680
\pinlabel $2$ at 43 750
\pinlabel $3$ at 120 680
\pinlabel $6$ at 120 746
\pinlabel $4$ at 250 730
\pinlabel $5$ at 250 700
\pinlabel $3$ at 225 659
\pinlabel $6$ at 225 770
\pinlabel $\ast$ at 197 678
\endlabellist
\centering
\includegraphics[scale=1.1]{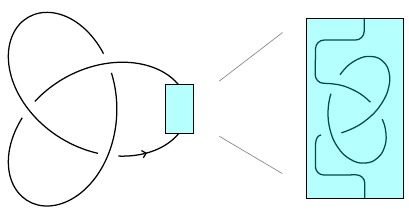}
\vspace{0.3cm}
\caption{The square knot $K$ with strands labeled by $1,\dots, 6$. Note that $K$ is oriented and the base point $\ast$ lies in $\partial B[T_{2, -3}]$ over the strand $3$.}
\label{figure_square_knot}
\end{figure}
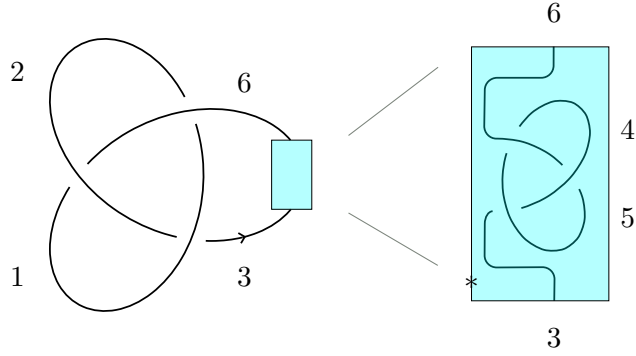

The strategy is as follows.
First, we describe the relations in $\Cord^{\mathrm{ab}}(K)$ using the knot diagram description.
We then construct a homomorphism $\varphi:\Cord^{\mathrm{ab}}(K)\rightarrow\mathbb{Z}_2[z]$.
Using (\ref{eqn_autom_gener}), we show that there is an element $x\in \Cord^{\mathrm{ab}}(K)$ such that $\varphi(x)\neq \varphi((F^{\mathrm{box}})_* x)$.
If we take an element $x'\in \Cord(K)$ which is projected to $x$, this satisfies $x' \neq (F^{\mathrm{box}})_*x'$ in $\Cord(K)$.

Now, we examine the relations in  $\Cord^{\mathrm{ab}}(K)$.
By the relations (\ref{eqn_triv_rel}), it is straightforward that, except $a_{36}$, all cords between different strands of $T_{2, 3}$-tangle (i.e. strands $1,2,3,6$) are the same and are idempotent. Let us denote them by $x$. Analogously for $T_{2, -3}$.
Here, we denote the corresponding cords by $y$. We are left with $3\cdot 6=18$ relations given by (\ref{eqn_cross}). From the crossing $(2; 1, 6)$ we obtain a relation that $a_{36}+a_{31}+a_{21}^2=0$. Since $a_{21}$ is idempotent and equal to $a_{31}$, it follows that $a_{36}=0$.
Then, all other $5$ remaining relations involving the cord $a_{36}$ are vacuous.
Hence, we are left with $12$ relations ($2$ for each crossing, where each trefoil has $3$ crossings):
\begin{align}
    y+a_{41}+a_{42}x=0\quad\hbox{ and }\quad y+a_{51}+a_{52}x=0,\label{eqn1}\\
    a_{42}+a_{41}+xy=0\quad\hbox{ and }\quad a_{52}+a_{51}+xy=0,\label{eqn2}\\
    y+a_{42}+a_{41}x=0\quad\hbox{ and }\quad y+a_{52}+a_{51}x=0,\label{eqn3}\\
    \nonumber\\
    x+a_{51}+a_{41}y=0\quad\hbox{ and }\quad x+a_{52}+a_{42}y=0,\label{eqn4}\\
    a_{51}+a_{41}+xy=0 \quad\hbox{ and }\quad a_{52}+a_{42}+xy=0,\label{eqn5}\\
    x+a_{41}+a_{51}y=0\quad\hbox{ and }\quad x+a_{42}+a_{52}y=0.\label{eqn6}
\end{align}
After subtracting (\ref{eqn2}) from (\ref{eqn5}) we obtain that $a_{51}=a_{42}$ and $a_{52}=a_{41}$.

According to the above computation, we can construct a well-defined $\Z_2$-algebra map $\varphi: \Cord^{\mathrm{ab}}(K)\rightarrow \mathbb{Z}_2[z]$ by
\begin{align}
\label{simplif_hom}
\begin{split}
x,y&\mapsto 1,\\
a_{51}\mapsto z+1\quad&\hbox{and}\quad a_{52}\mapsto z.
\end{split}
\end{align}
For convenience, let us describe a table of images of generators of $\mathcal{A}_n$ under $\varphi$.
See Table \ref{table_gen}.

\begin{table}[!h]
\renewcommand{\arraystretch}{1.2}%
\begin{center}
\aboverulesep = 0pt
\belowrulesep = 0pt
\begin{tabular}{ c |c | c | c | c | c | c | } 
 \multicolumn{1}{c}{} & \multicolumn{1}{c}{\textcolor{gray}{1}} & \multicolumn{1}{c}{\textcolor{gray}{2}} & \multicolumn{1}{c}{\textcolor{gray}{\,\,\,3\,\,\,}}  & \multicolumn{1}{c}{\textcolor{gray}{4}} & \multicolumn{1}{c}{\textcolor{gray}{5}} & \multicolumn{1}{c}{\textcolor{gray}{\,\,\,6\,\,\,}}\\
 \cmidrule{2-7}
 \textcolor{gray}{1} & 0 & 1 & \,\,\,1\,\,\,  & z+1 & z & \,\,\,1\,\,\,\\
 \cmidrule{2-7}
 \textcolor{gray}{2} & 1 & 0 & 1 & z & z+1 & 1 \\ 
 \cmidrule{2-7}
 \textcolor{gray}{3} & 1 & 1 & 0 & 1 & 1 & 0\\
 \cmidrule{2-7}
 \textcolor{gray}{4} & z+1 & z & 1 & 0 & 1 & 1\\
 \cmidrule{2-7}
 \textcolor{gray}{5} & z & z+1 & 1 & 1 & 0 & 1\\
 \cmidrule{2-7}
 \textcolor{gray}{6} & 1 & 1 & 0 & 1 & 1 & 0\\
 \cmidrule{2-7}
\end{tabular}
\caption{Table of images of generators of $\mathcal{A}_n$ under $\varphi$, where $\varphi(a_{ij})$ is on $i$-th row and $j$-th column.}
\label{table_gen}
\end{center}
\end{table}

Next, let us study $(F^{\mathrm{box}})_\ast[a_{24}]$.
First, note that $a_{24}$ is the same as the concatenation of $a_{23}$ and $ a_{34}$ at the base point.
Recall that the base point lies over the strand $3$ of $K$.
From our convention, $[a_{23}]$ extends to a loop $ [\hat{a}_{23}] \in \pi_{1}(\R^3\setminus L_1)$ and $[a_{34}]$ extends to a loop $[\hat{a}_{34}] \in \pi_{1}(\R^3\setminus L_2)$.
Hence, $[a_{24}]\in \Cord^{\mathrm{ab}}(K)$ corresponds to $[\hat{a}_{23}\bullet \hat{a}_{34}]\in \Cord^{\mathrm{loop}}(K)$ by the projection $\Cord^{\mathrm{loop}}(K)\cong \Cord(K)\to \Cord^{\mathrm{ab}}(K)$.
By (\ref{eqn_autom_gener}), we obtain
\[(F^{\mathrm{box}})_\ast[\hat{a}_{23}\bullet \hat{a}_{34}]=[\ell_1^{-1}\bullet \hat{a}_{23}\bullet \ell_1\bullet \hat{a}_{34}].
\]

As an element of $\mathcal{P}_K$ (or a cord of $K$), we can draw
$$[\ell_1^{-1}\bullet \hat{a}_{23}\bullet\ell_1\bullet \hat{a}_{34}]=\includegraphics[height=3cm,valign=c,scale=1.6]{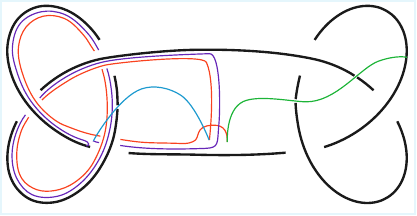},$$
where the purple and orange paths around $T_{2, 3}$ correspond to $\ell^{-1}_1$ and $\ell_1$, respectively. The blue and green paths are the cords $a_{23}$ and $a_{34}$, in this order.
(Actually, in the present case, it is not hard to directly see that the cord $a_{24}$ is mapped by $F^{\mathrm{box}}$ to the cord in the right-hand side.)


Recall the relation (ii) illustrated in Figure \ref{fig-cord-Z2}. (One can alternatively refer to the `skein relation' \cite[(1) in Section 1.1]{Ng2003} for cords.)
Then, we continue the computation in $\Cord(K)$:
\begin{align*}
    [\ell_1^{-1}\bullet \hat{a}_{23}\bullet \ell_1 \bullet \hat{a}_{34}]&\stackrel{(1)}{=}\includegraphics[height=3cm,valign=c]{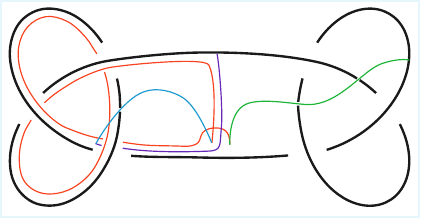}\\
    &\stackrel{(2)}{=}\includegraphics[height=3cm,valign=c]{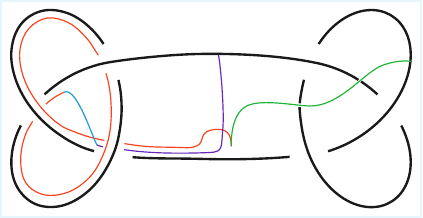}\\
    &\stackrel{(3)}{=} 0 +\includegraphics[height=3cm,valign=c]{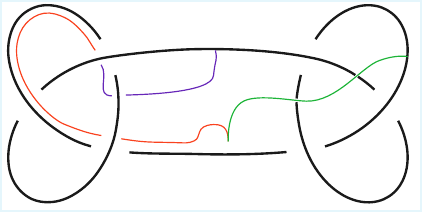}\\
    &\stackrel{(4)}{=}[a_{64}]+\includegraphics[height=2.8cm,valign=c]{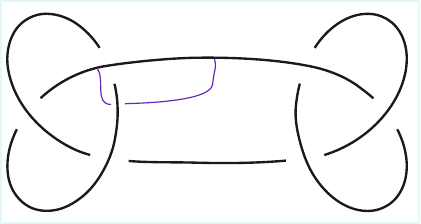}\cdot\includegraphics[height=2.8cm,valign=c]{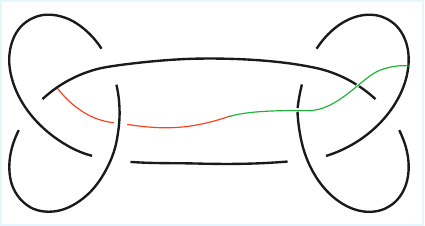}\\
    &\stackrel{(5)}{=}[a_{64}]+( [a_{61}] \cdot [a_{16}] + [a_{66}] )\cdot([a_{64}]+[a_{61}]\cdot[a_{14}]).
\end{align*}
In the equalities $(1)$ and $(2)$, we used the homotopy invariance of cords.
In $(3)$, we pushed the cord through the strand $2$ in a neighborhood of the crossing $(2:1,6)$, and apply the relation $[a_{36}]=0$.
In the case $(4)$, we pushed the cord through the strand $6$ in a neighborhood of $(6;1,2)$, and finally in $(5)$, we push twice the cords through the strand $1$.

Referring to Table \ref{table_gen}, we can compute  that
\[\varphi((F^{\mathrm{box}})_\ast[a_{24}])=1+(1+0)(1+1\cdot (z+1))=1+1+z+1=z+1.\]
On the other hand, $\varphi([a_{24}])=z$. Hence, the action of $(F^{\mathrm{box}})_\ast$ on $\Cord(T_{2, 3}\#T_{2, -3})$ is non-trivial.

To the end, we note that in the spirit of Remark \ref{rem_embedding_spaces}, instead of the action of blue-box loop, we can also study the action of a single Gramain's loop of $B[T_{2, -3}]$ on $\Cord(K)$. In fact, this leads to a slightly shorter computation. However, the authors find that the blue-box loop is more suitable for a further generalization to a loop of cable knots in Example \ref{ex_cable}.
\end{example}

From Example \ref{exmpl_nontriv1} we can derive the following families of infinite number of knots, where $(F^{\mathrm{box}})_\ast$ also has a non-trivial action on cord algebra.

\begin{example}We consider the knot $K\coloneqq T_{2, 3}\#(T_{2, -3}\# L)$, where $L$ is an arbitrary knot that is connected to $T_{2, -3}$.
We consider that $L$ is connected to the strand $4$ in Figure \ref{figure_square_knot}.

Now, let us consider a blue box $B[T_{2, -3}\# L]\subset N_{T_{2, 3}}$. Even though the structure of $\Cord^{\mathrm{ab}}(K)$ can be much more complicated than Example \ref{exmpl_nontriv1}, we want to construct a simplifying  map  $\widetilde{\varphi}:\Cord^{\mathrm{ab}}(K)\rightarrow\mathbb{Z}_2[z]$ similar to $\varphi$ in Example \ref{exmpl_nontriv1}.

We define $\widetilde{\varphi}$ on the generators of $\mathcal{A}_n$ as follows. For the cords between the strands of $T_{2, 3}$ or $T_{2, -3}$, the map $\widetilde{\varphi}$ agrees with the map $\varphi$ from (\ref{simplif_hom}).
Next, every cord with endpoints on $L$ is mapped to $0$. It remain the case of cords $c \in \{a_{i,j} \mid 1\leq i,j \leq n\}$ with one endpoint on $L$ and other endpoint on the complement of $L$ in $K$.
For the strand where the latter endpoint lives, let $r\in \{1,\dots ,6\}$ be the number of this strand in Figure \ref{figure_square_knot}.
Then, we put $\widetilde{\varphi} (c)\coloneqq \varphi (a_{r4})$.
The relations (\ref{eqn_cross}) for all crossings are mapped to $0$, and we obtain a well-defined ring homomorphism $\widetilde{\varphi}\colon \Cord^{\mathrm{ab}}(K)\to \Z_2[z]$.
We take a cord $c$ as above such that $a_{r4}=a_{24}$.
Note that the $L$-part does not affect the computation of $\widetilde{\varphi}((F^{\mathrm{box}})_\ast[c])$ at all.
It follows that $\widetilde{\varphi}((F^{\mathrm{box}})_\ast [c]) \neq \widetilde{\varphi} ([c])$, and thus the action of $(F^{\mathrm{box}})_\ast$ on $\Cord^{\mathrm{ab}}(K)$ is non-trivial.
\end{example}

\begin{example}\label{ex_cable}
For $k\in \mathbb{Z}$. let us consider $(2k+1, 1)$-cable of the square knot $T_{2, 3}\#T_{2, -3}$.
By the cable construction, we mean the following. Let $V_1$ be a standard unknotted solid torus around the torus knot $T_{2k+1, 1}$.
Let $V_2$ be a tubular neighborhood of $T_{2, 3}\#T_{2, -3}$.
We fix a homeomorphism $h: V_1 \rightarrow V_2$ and define $K\coloneqq h(T_{2k+1, 1})$.

Let us rewrite the blue box $B[T_{2,-3}]$ in Example \ref{exmpl_nontriv1} by $B_{2k+1}[T_{2,-3}]$, and consider that all $2k+1$ copies of $T_{2, -3}$ lie in $B_{2k+1}[T_{2, -3}]$.
In addition, for simplicity we assume that $K\setminus B_{2k+1}[T_{2, -3}]$ consists of $2k+1$ parallel strands. In particular, the twisting part of the embedded torus knot $T_{2k+1, 1}$ is completely located inside of the box $B_{2k+1}[T_{2, -3}]$.
Then, a longitudinal turn of the blue box $B_{2k+1}[T_{2, -3}]$ around $V_2$, which is the same as in Example \ref{exmpl_nontriv1}, induces a loop of the cable knot.
We remark that for such a loop, an analogue of Proposition \ref{lemma_loop_lift} also holds.
We take an ambient isotopy $(F_t)_{t\in [0,1]}$ to be the same one as in Example \ref{exmpl_nontriv1}, and set $F^{\mathrm{box}}\coloneqq F_1$.
Our aim is to relate the action of $(F^{\mathrm{box}})_\ast$ on $\Cord^{\mathrm{ab}}(K)$ to the action of $(F^{\mathrm{box}})_\ast$ on $\Cord^{\mathrm{ab}}(T_{2, 3}\#T_{2, -3})$, which we already understand from Example \ref{exmpl_nontriv1}.

By Seifert-van Kampen Theorem, we have an isomorphism
\begin{equation}\label{eqn_kampen_cable}
\pi_1(\mathbb{R}^3\setminus K)\cong\pi_1\big(\mathbb{R}^3\setminus T_{2, 3}\#T_{2, -3}\big)\ast_{\pi_1(\partial V_2)} \pi_1\big(V_2\setminus h(T_{2k+1, 1})\big),
\end{equation}
which comes from the following diagram
\[ \xymatrix{
\pi_1(\R^3\setminus K)  & \pi_1(\R^3\setminus \mathrm{int} V_2) \ar[r]^-{\cong} \ar[l] & \pi_1(\R^3\setminus T_{2,3}\# T_{2,-3}) \\
\pi_1(V_2\setminus h(T_{2k+1,1})) \ar[u]  & \pi_1(\partial V_2)  . \ar[u]^-{i_*} \ar[l]_-{j_*} &
}\]
Here, all homomorphisms are induced by the inclusion maps.

The meridian and the longitude of $T_{2,3}\# T_{2,-3}$ fixed in Example \ref{exmpl_nontriv1} determines elements $m_2,\ell_2\in \pi_1(\partial V_2)$.
Let $m\in \pi_1 (V_2\setminus h(T_{2k+1,1}))$ be the longitude of $h(T_{2k+1})$.
Then, the first homology group $H_1(V_2\setminus h(T_{2k+1,1}))$ is a free abelian group generated by $\ell_2$ and $m$.
Since $\pi_1(\partial V_2)\cong H_1(\partial V_2)$, there is a unique group homomorphism
\[ \psi \colon \pi_1 (V_2\setminus h(T_{2k+1,1})) \to \pi_1(\partial V_2) \colon [\ell_2] \to [\ell_2], \ [m] \mapsto [m_2] , \]
which is reduced to an isomorphism on the first homology groups.


$\psi \circ j_*\colon \pi_1(\partial V_2) \to \pi_1(\partial V_2)$ maps $[\ell_2]$ to $[\ell_2]$, and $[m_2]$ to $[m_2]^{2k+1}$.
Therefore, by (\ref{eqn_kampen_cable}), the map $i_*\circ \psi\colon \pi_1(V_2\setminus h(T_{2k+1,1})) \to \pi_1 (\R^3\setminus T_{2,3}\# T_{2,-3})$ induces a group epimorphism 
\begin{equation}\label{map_group_hom}
\widetilde{\psi}:\pi_1(\bR^3\setminus K)\to \frac{\pi_1(\bR^3\setminus T_{2,3}\# T_{2,-3})}{\langle[e]=[m_2]^{2k}\rangle}.
\end{equation} 
This is the unique map which makes the following diagram commute:
\[ \xymatrix{
 \frac{\pi_1(\bR^3\setminus T_{2,3}\# T_{2,-3})}{\langle[e]=[m_2]^{2k}\rangle} & \pi_1(\R^3\setminus K) \ar[l]^-{\widetilde{\psi}} & \pi_1(\bR^3\setminus T_{2,3}\# T_{2,-3}) \ar[l] \ar@/_15pt/[ll]_-{q} \\
 & \pi_1(V_2,h(T_{2k+1,1})) \ar[u] \ar[ul]^-{q\circ i_*\circ \psi} & \pi_1(\partial V_2) \ar[l]_-{j_*} \ar[u]_-{i_*},
} \]
where $q$ is the quotient map.
Note that modulo $[m_2]^{2k}$, $\widetilde{\psi}$ maps the meridian (resp. the longitude) of $K$ to the meridian (resp.the $(2k+1)$-th power of the longitude) of $T_{2,3}\# T_{2,-3}$.
Consequently, using (\ref{map_group_hom}), we obtain a surjective $\bZ_2$-algebra map
$$\Cord^{\mathrm{loop}}(K)\rightarrow \frac{\Cord^{\mathrm{loop}}(T_{2, 3}\#T_{2, -3})}{\langle [\gamma_a\bullet\gamma_b]=[\gamma_a \bullet m_2^{2k}\bullet \gamma_b]\rangle}.$$
In fact, the relation $[\gamma_a\bullet\gamma_b]=[\gamma_a \bullet m_2^{2k}\bullet \gamma_b]$ for all $[\gamma_a],[\gamma_b]\in \pi_1(\bR^3\setminus T_{2,3}\# T_{2,-3})$ is vacuous in $\Z_2$-coefficients. Indeed, we have
\begin{align*}
   [\gamma_a \bullet m^{2k}_2 \bullet \gamma_b]&=[\gamma_a\bullet m^{2k-1}_2\bullet \gamma_b] +[\gamma_a\bullet m^{2k-1}_2]\cdot [ \gamma_b]\\
   &=[\gamma_a\bullet m^{2k-2}_2\bullet \gamma_b] +[\gamma_a\bullet m^{2k-2}_2]\cdot [ \gamma_b]+[\gamma_a\bullet m^{2k-1}_2]\cdot [ \gamma_b]\\
   &=[\gamma_a\bullet m^{2k-2}_2\bullet \gamma_b]+(1+1)[\gamma_a]\cdot [ \gamma_b].
\end{align*}
For the last equality, we use (\ref{meridian-vanish}).
Hence, we obtain a well-defined surjective $\bZ_2$-algebra map
$$\Psi:\Cord(K)\rightarrow \Cord(T_{2, 3}\#T_{2, -3}).$$

Now, let us describe the action of $(F^{\mathrm{box}})_\ast$ on $\pi_{1}(\mathbb{R}^3\setminus K)$ using the amalgamated product (\ref{eqn_kampen_cable}).
First, on generators from $\pi_1(\mathbb{R}^3\setminus T_{2, 3}\#T_{2, -3}) \cong \pi_1(\R^3\setminus \mathrm{int} V_2)$, the action of $(F^{\mathrm{box}})_\ast$ was described in (\ref{eqn_autom_gener}).
Next, on generators from $\pi_1\big(V_2\setminus h(T_{2k+1, 1})\big)$, it is immediate that the action of $(F^{\mathrm{box}})_\ast$ is the identity map. 
Consequently, we obtain the following commutative diagram
\[ \xymatrix@C=40pt{
\Cord(K) \ar@{->>}[d]_-{\Psi} \ar[r]^-{(F^{\mathrm{box}})_\ast} & \Cord(K) \ar@{->>}[d]_-{\Psi} \\
\Cord(T_{2, 3}\#T_{2, -3}) \ar[r]^-{(F^{\mathrm{box}})_\ast} & \Cord(T_{2, 3}\#T_{2, -3}). 
}\]

Finally, recall that from Example \ref{exmpl_nontriv1} we know that the map $F^{\mathrm{box}}$ induces a non-trivial automorphism on $\Cord(T_{2, 3}\#T_{2, -3})$.
Since $\Psi$ is a surjection, $F^{\mathrm{box}}$ induces a non-trivial automorphism on $\Cord(K)$, too.
\end{example}

\begin{remark}\label{rem-potential}
Even though Theorem \ref{thm-1} is so far proved only in $\bZ_2$-coefficients, now we would like to briefly discuss potential advantage of the fully non-commutative cord algebra $\Cord^{\mathrm{nc}}(K)$  in $\bZ_2[\lambda^{\pm 1}, \mu^{\pm1}]$-coefficients on the level of automorphisms on it.
For the precise definition of $\Cord^{\mathrm{nc}}(K)$, one can refer to \cite[Definition 2.6]{CELN} or  Definition \ref{def-fully-nc} below.

For this, we consider the action of a Gramain's loop (Remark \ref{rem_embedding_spaces}) on a knot $K\subset \bR^3$.
Here, $K$ is a Seifert framed knot with a base-point $\ast$ on the framing $K'$. 
We realize a Gramain's loop as follows. Let $B[\ast]$ be a small box in $\R^3$ that contains a small part of $K$ with the base-point $\ast$ on $\partial B[*]$. Then we can take a convention that Gramain's loop twists $K$ around $B[\ast]$ once in the negative meridian direction.
Let $(F_t)_{t\in[0, 1]}$ be an ambient isotopy induced by the Gramain's loop and set $F^{\mathrm{Gram}}\coloneqq F_1$.
Then, we obtain
$$(F^{\mathrm{Gram}})_\ast[\gamma]=[m\bullet\gamma\bullet m^{-1}],\hspace{0.4cm}\hbox{ where }[\gamma]\in\pi_1(\bR^3\setminus K, \ast).$$
In addition, similarly to the blue-box loop, this action naturally translates to $\Cord(K)$ and $\Cord^{\mathrm{nc}}(K).$
It is immediate from (\ref{meridian-vanish}) that the action $(F^{\mathrm{Gram}})_\ast$ is trivial on $\Cord(K)$, however as we will see, this is not the case for $\Cord^{\mathrm{nc}}(K)$.

Let us put $K\coloneqq T_{2, 3}$. Then, by \cite[Example 2.24]{CELN}, $\Cord^{\mathrm{nc}}(K)$ is a non-commutative ring generated by $\lambda^{\pm}$, $\mu^{\pm}$ and one more generator $s$ modulo a two-sided ideal generated by the following four elements
$$\lambda\mu+\mu\lambda ,\  \lambda\mu^6 s+ s\lambda\mu^6,\  1+\mu+s+\lambda\mu^5 s\mu^{-3}s \mu^{-1},\  1+\mu+\lambda\mu^4 s\mu^{-2}+\lambda\mu^5 s\mu^{-2}s\mu^{-1}.$$
Moreover, as in \cite[Example 2.24]{CELN}, there is a well-defined ring homomorphism $\varphi$ from $\Cord^{\mathrm{nc}}(K)$ to a non-commutative $\Z_2$-algebra $\Z_2\la  \mu^{\pm},  a^{\pm}\ra / ( \mu a \mu + a \mu a ) \cong\bZ_2[\pi_1(\R^3\setminus T_{2,3})] $ given by
\[\begin{array}{ccc}
\mu\mapsto \mu, & \lambda\mapsto  a\mu a^{-1}\mu a\mu^{-3}, & s\mapsto (1+\mu)a\mu^{-1}a^{-1}.
\end{array}
\]
Then, we have
\begin{align*} 
\varphi(s)+\varphi\circ (F^{\mathrm{Gram}})_* (s) &= (1+\mu)a\mu^{-1}a^{-1} + \mu (1+\mu)a\mu^{-1}a^{-1} \mu^{-1} \\
&= (1+\mu)a\mu^{-1}a^{-1} + (1+\mu)\mu a (a^{-1}\mu^{-1}a^{-1}) \\
& = (1+\mu)(a+\mu)\mu^{-1}a^{-1}.
\end{align*}
Actually, $(1+\mu)(a+\mu)\neq 0$. To see this, consider a well-defined homomorphism to the ring of $2\times 2$ matrices with entries in $\Z_2$
\[ \Z_2\la \mu^{\pm}, a^{\pm}\ra / ( \mu a \mu + a \mu a ) \to M_{2\times 2}(\Z_2) \colon a\mapsto \begin{pmatrix}0 & 1 \\ 1 & 0 \end{pmatrix},\ \mu \mapsto \begin{pmatrix}1 & 1 \\ 0 & 1 \end{pmatrix}, \]
which maps $(1+\mu)(a+\mu)$ to $\begin{pmatrix}0 & 1 \\ 0 & 1 \end{pmatrix}$.
Therefore, the action of $(F^{\mathrm{Gram}})_\ast$ is non-trivial on $\Cord^{\mathrm{nc}}(K)$, though it is trivial on $\Cord(K)$.
\end{remark}

\appendix

\section{Proof of Theorem \ref{thm-str-cord}}\label{subsec-proof-isom}

\subsection{Review of the string homology in degree $0$ from \cite{CELN}}\label{subsec-full-string}

For any $q\in K$ and $v\in T_q \bR^3 = T_qK\oplus (T_qK)^{\perp} $, let $v^{\mathrm{normal}}$ denote the $(T_qK)^{\perp}$-component of $v$.
Consider the tubular neighborhood $N$ of $K$ given by (\ref{tub-nbd}).
Fix $x_0\in \partial N$ and a unit vector $v_0\in T_{x_0}N$.

\begin{definition}[Definition 2.1 of \cite{CELN}]\label{def-string-CELN}
Let $l\in \Z_{\geq 0}$. A \textit{broken string with $2l$ switches} on $K$ is a tuple 
\[s=(s_1,\dots ,s_{2l+1})\]
consisting of $C^1$ paths
\[ \begin{cases} 
s_{2i+1}  \colon [a_{2i},a_{2i+1}]\to N & \text{ for }i=0,\dots ,l , \\
s_{2i} \colon [a_{2i-1},a_{2i}]\to \bR^3 & \text{ for }i=1,\dots ,l,
\end{cases}\text{ where }0=a_0<a_1<\dots < a_{2l+1},
\] 
satisfying the following conditions:
\begin{itemize}
\item $s_0(0)=s_{2l+1}(a_{2l+1})=x_0$ and $\dot{s}_0 (0) = \dot{s}_{2l+1}(a_{2l+1})=v_0$.
\item $s_{j}(a_j)=s_{j+1}(a_j)\in K$ for every $j\in \{1,\dots ,2l\}$.
\item $\dot{s}_{2i}(a_{2i})^{\mathrm{normal}} = -\dot{s}_{2i+1}(a_{2i})^{\mathrm{normal}}$ and $\dot{s}_{2i-1}(a_{2i-1})^{\mathrm{normal}} = \dot{ s}_{2i}(a_{2i-1})^{\mathrm{normal}}$ for every $i\in \{1,\dots ,l\}$.
\end{itemize}
$s_{2i}$ and $s_{2i+1}$ are referred to as a \textit{$Q$-string} and an \textit{$N$-string} respectively.
We let $\tilde{\Sigma}^l_K$ denote the space of broken strings with $2l$ switches on $K$.
\end{definition}

\begin{definition}
For $l\in \Z_{\geq 0}$, we define $C_0(\tilde{\Sigma}^l_K)$ to be the $\bZ_2$-vector space freely generated by $\bold{s}=(s_1,\dots ,s_{2l+1})\in \tilde{\Sigma}^l_K$ such that:
\begin{itemize}
\item[(0a')] $(s_j)^{-1}(K) = \{a_{j-1},a_{j}\}$ for every $j\in\{1,\dots ,2l+1\}$. Here, $[a_{j-1},a_j]$ is the domain of $s_j$.
\item[(0b')] For every $j\in\{2,\dots ,2l\}$, $s_{j}$ is not tangent to $K$ at the endpoints $\{ a_{j-1}, a_{j}\}$.
\end{itemize}
\end{definition}

To define the string homology of $K$, we need to introduce:
\begin{itemize}
\item a $\Z_2$-vector space $C_1(\tilde{\Sigma}^l_K)$.
\item $\Z_2$-linear maps  $\partial \colon C_1(\tilde{\Sigma}^l_K) \to C_0(\tilde{\Sigma}^l_K)$ and $\delta_Q,\delta_N \colon C_1(\tilde{\Sigma}^l_K) \to C_0(\tilde{\Sigma}^{l+1}_K)$.
\end{itemize}
In order to write down the proof of Theorem \ref{thm-str-cord}, we do not need to refer to their precise definitions, so let us briefly recall them. (For the precise definitions, see \cite[Section 2.1]{CELN} and reduce the coefficient ring from $\Z$ to $\Z_2$.)
$C_1(\tilde{\Sigma}^l_K)$ is a counterpart of $C_1(\Sigma^l_K)$ in Definition \ref{def-gen-chain} and consists of generic $1$-chains of $\tilde{\Sigma}^l_K$ satisfying certain transversality conditions.
$\partial$ is the boundary operator for singular $1$-chains.
Finally, by splitting a $Q$-string (resp. $N$-string) at a point where it intersects $K$, and by inserting a short $N$-string (resp. $Q$-string) called a \textit{spike}, $\delta_Q$ (resp. $\delta_N$) is defined in a similar way as $\delta\colon C_1(\Sigma^l_K)\to \bar{C}_0(\Sigma^{l+1}_K)$.

Let us set $C_0(\tilde{\Sigma}_K)\coloneqq \bigoplus_{l=0}^{\infty} C_0(\tilde{\Sigma}^l_K)$ and $C_1(\tilde{\Sigma}_K)\coloneqq \bigoplus_{l=0}^{\infty} C_1(\tilde{\Sigma}^l_K)$.
Then, a $\Z_2$-linear map
\[ \partial +\delta_Q+\delta_N \colon C_1(\tilde{\Sigma}_K)\to C_0(\tilde{\Sigma}_K) \]
is defined.
The \textit{string homology} $H^{\str}_0(K)$ in degree $0$ is defined by
\[ H^{\str}_0(K) \coloneqq C_0(\tilde{\Sigma}_K)/ \Image(\partial +\delta_Q +\delta_N). \]
In addition, $C_0(\tilde{\Sigma}_K)$ has a product structure induced by the map
\[ \tilde{\Sigma}^l_K \times \tilde{\Sigma}^{l'}_K \to \tilde{\Sigma}^{l+l'}_K \colon \left( (s_1,\dots ,s_{2l+1}),(s'_1,\dots ,s'_{2l'+1})\right) \mapsto ( s_1,\dots ,s_{2l}, s_{2l+1}\bullet s'_{1}  , s'_2,\dots ,s'_{2l'+1} ), \]
where $s_{2l+1}\bullet s'_{1} $ is the concatenation of $s_{2l+1}$ and $s'_{1}$ at $x_0$.
This determines a product structure on $H_0^{\str}(K)$.

\subsection{Fully non-commutative cord algebra}
\label{sec:noncom cord}

We fix an orientation and a framing of $K$.
We also fix a base point $q_0\in K$, then a point $q'_0\in K'$ is determined by the shift.
Let $\mathcal{P}_{K,q_0}$ be the set of homotopy classes of paths $\gamma\colon[0,1] \to \bR^3 \setminus K$ such that $\gamma(0),\gamma(1)\in K'\setminus \{q'_0\}$.
\begin{definition}[Definition 2.6 of \cite{CELN}]\label{def-fully-nc}
Let $\mathcal{A}^{\nc}$ be the non-commutative $\bZ_2$-algebra generated by $\mathcal{P}\cup \{ \lambda , \lambda^{-1},\mu,\mu^{-1}\}$ modulo the relations
\begin{align}\label{rel-lambda-mu} 
\begin{array}{cc} \lambda\cdot \lambda^{-1} = \lambda^{-1}\cdot \lambda = \mu\cdot \mu^{-1}= \mu^{-1}\cdot \mu =1, & \lambda \cdot \mu = \mu\cdot \lambda. \end{array}
\end{align}
The \textit{fully non-commutative cord algebra} of $K$ is defined as
\[ \Cord^{\nc}(K) \coloneqq \mathcal{A}^{\nc}/\mathcal{I}^{\nc},\]
where $\mathcal{I}^{\nc}$ is a two-sided ideal of $\mathcal{A}^{\nc}$ generated by the four types of relations as in Figure \ref{fig-cord-relation}.
\begin{figure}
\centering
\begin{overpic}[height=5cm]{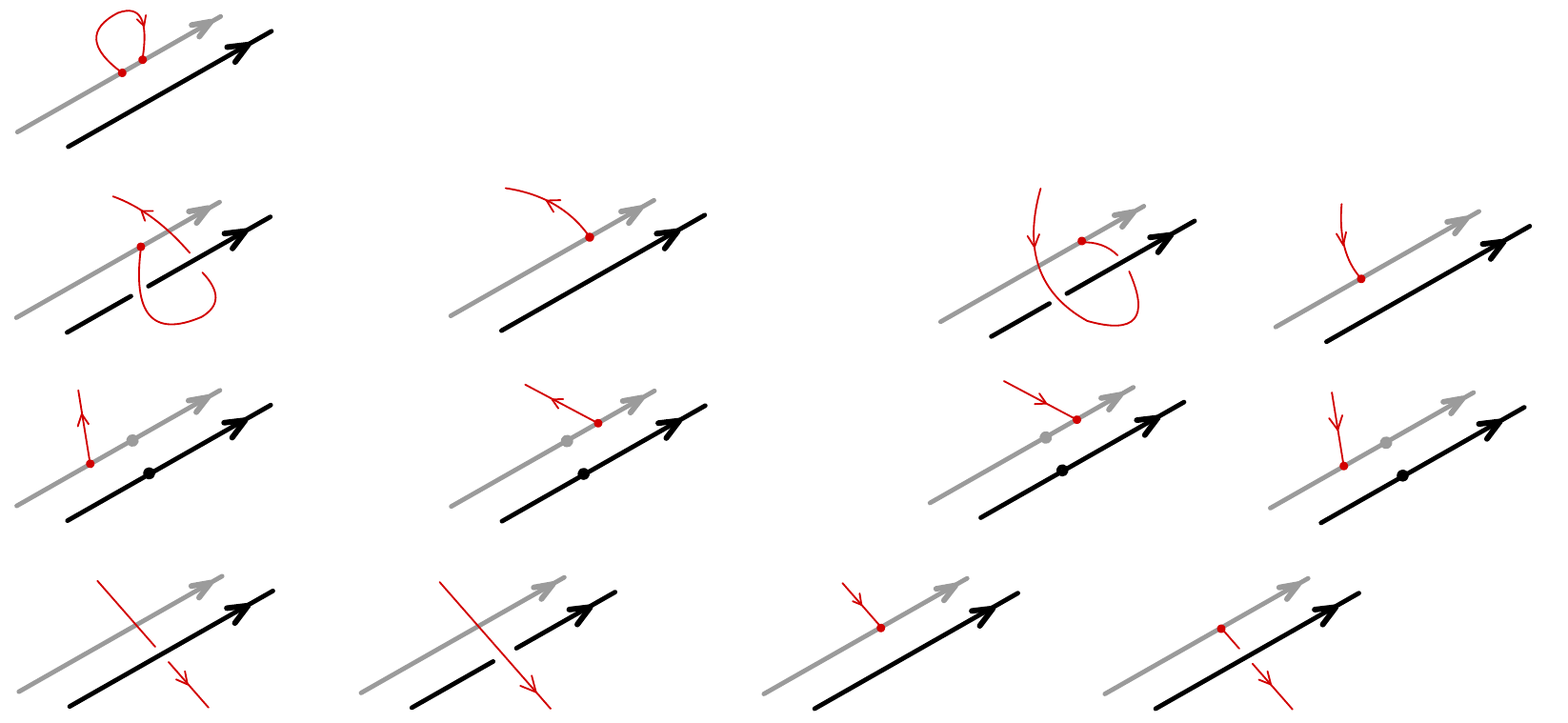}
\put(-5,40){(i)}
\put(-6,28){(ii)}
\put(-7,15){(iii)}
\put(-7,4){(iv)}
\put(19,40){$= 1-\mu$ }
\put(19,28){$=\mu\ \cdot$}
\put(49,28){and}
\put(77.5,28){$=$}
\put(100,28){$\cdot \ \mu$ }
\put(19,15){$=\lambda\ \cdot$}
\put(49,15){and}
\put(77.5,15){$=$}
\put(100,15){$\cdot \ \lambda$ }
\put(19,4){$-$}
\put(43,4){$=$}
\put(67,4){$\cdot$}
\end{overpic}
\caption{The black line is the oriented knot $K$ and the gray line is its shift $K'$ with respect to the framing. The bullets on them in (iii) are the base points.}\label{fig-cord-relation}
\end{figure}
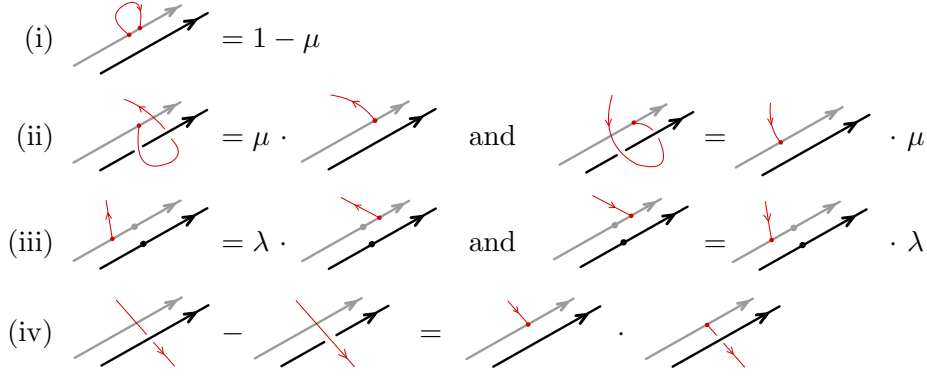
\end{definition}

$\Cord^{\nc}(K)$ is a $\Z_2[\lambda^{\pm},\mu^{\pm}]$-NC-algebra in the sense of \cite[Definition 2.12]{CELN}: It is a ring equipped with a ring homomorphism $\Z_2[\lambda^{\pm},\mu^{\pm}]\to \Cord^{\nc}(K)$. 
It should be noted that $\lambda$ and $\mu$ correspond to the longitude and the meridian of $K$ respectively
via the isomorphism fixed by the orientation and the framing of $K$
\begin{align}\label{peripheral}
\{\lambda^a\mu^b\mid a,b\in \Z \} = \Z \cdot \lambda \oplus \Z \cdot \mu \cong \pi_1(\partial N,x_0).
\end{align}

Fix a point $q_1\in K\setminus \{q_0\}$ near $q_0$. 
Take paths $\gamma_Q$ and $\tilde{\gamma}_Q$ with endpoints in $\{q_1,q'_1\}$ as in Section \ref{subsubsec-cord-Z2}.
We assume that $K'$ lies in $\partial N$ and take $q'_1\in K'$ as the base point $x_0\in \partial N$ of Definition \ref{def-string-CELN}. 
By \cite[Proposition 2.9]{CELN}, there is an isomorphism of algebra
\begin{align}\label{isom-I-nc}  I_K^{\mathrm{nc}} \colon \Cord^{\nc}(K) \to H^{\str}_0(K) .
\end{align}
Let us describe how to construct this isomorphism.
Take $l \in \Z_{\geq 0}$, $x_1,\dots ,x_l\in\mathcal{P}_{K,q_0}$ and $\alpha_1,\dots ,\alpha_{l+1}\in \{\lambda^a\mu^b \mid a,b\in \bZ\}$.
Then, $I^{\nc}_K$ maps $\alpha_1 x_1\alpha_2 \cdots x_l\alpha_{l+1} \in \Cord(K)$
to an element of $H^{\str}_0(K)$ represented by
\begin{align}\label{alpha-x-alpha} (\bar{\alpha}_1,\bar{x}_1,\bar{\alpha}_2, \dots ,\bar{x}_l, \bar{\alpha}_{l+1}) \in C_0(\tilde{\Sigma}^l_K),
\end{align}
where $\bar{x}_1,\dots ,\bar{x}_l$ are $Q$-strings and $\bar{\alpha}_1,\dots ,\bar{\alpha}_{l+1}$ are $N$-strings such that:
\begin{itemize}
\item Up to parametrization, the path $\bar{x}_{i}$ agrees with $\hat{\gamma}_{x_i}$ in Proposition \ref{prop-isom-IK}.
\item  Via (\ref{peripheral}), $\alpha_i$ determines  a homotopy class of a loop in $\partial N$ based at $x_0$, and
$\bar{\alpha}_i$ is defined in a similar way as $\bar{x}_i$.
(Its precise definition will not be needed in the later discussion.
For more details, see the proof of \cite[Proposition 2.9]{CELN}.)
\end{itemize}

\subsection{Relation between $\Cord^{\nc}(K)$ and $\Cord (K)$}\label{subsubsec-cord-nc}

Let $A$ be a $\Z_2[\lambda^{\pm},\mu^{\pm}]$-NC-algebra.
First, we define $A^{\#}$ to be the ring $A$ modulo the two-sided ideal generated by
\[ k \cdot x - x\cdot k \]
for all $x\in A$ and $k \in \bZ_2[\lambda^{\pm},\mu^{\pm}]$.
Then, $\mathcal{A}^{\#}$ becomes a $\bZ_2[\lambda^{\pm},\mu^{\pm}]$-algebra.
Next, for the $\Z_2[\lambda^{\pm},\mu^{\pm}]$-algebra $A^{\#}$,
we use the ring homomorphism $\bZ_2[\lambda^{\pm},\mu^{\pm}] \to \bZ_2\colon \lambda\mapsto 1, \mu \mapsto 1$ to define a $\bZ_2$-algebra
\begin{align}\label{A-reduction}
\rest{A}{\lambda=1,\mu=1}^{\#} \coloneqq A^{\#}\otimes_{\Z_2[\lambda^{\pm},\mu^{\pm}]} \Z_2 
\end{align}
together with a $\Z_2$-algebra map
$j_A \colon A \to  \rest{A}{\lambda=1,\mu=1}^{\#} \colon x\mapsto x\otimes 1.$

In the case of $A=\Cord^{\nc}(K)$, 
we can see from the relations in Definition \ref{def-cord-Z/2} and Definition \ref{def-fully-nc} that there is a natural isomorphism of $\bZ_2$-algebra
$ \rest{\Cord^{\mathrm{nc}}(K)}{\lambda=1,\mu=1}^{\#} \to \Cord(K)\colon [\gamma]\otimes 1\mapsto [\gamma]$,
so we have a $\Z_2$-algebra map
\[ j_K\coloneqq j_{\Cord^{\nc}(K)} \colon \Cord^{\nc}(K) \to \Cord(K) \]
which maps $\lambda$ and $ \mu$ to $1$ and $[\gamma]\in \mathcal{P}_{K,q_0}$ to $[\gamma]\in\mathcal{P}_K$.

Let us observe the composition
\begin{align}\label{I-1-j-I}
\xymatrix{
H^{\str}_0(K) \ar[r]^-{(I^{\nc}_K)^{-1}} & \Cord^{\nc}(K) \ar[r]^-{j_K} & \Cord(K) \ar[r]^-{I_K} & \bar{H}^{\str}_0(K) .
} \end{align}
By the composite map, $\bar{\alpha}$ for any $\alpha \in \{\lambda^a\mu^b \mid a,b\in \Z\}$ is mapped to the unit of $\bar{H}_0^{\str}(K)$, and any element represented by (\ref{alpha-x-alpha}) for $l\geq 1$ is mapped to an element of $\bar{H}_0^{\str}(K)$ represented by $(\hat{\gamma}_{x_1},\dots , \hat{\gamma}_{x_l})\in \bar{C}_0(\Sigma^l_K)$.
Therefore, any element of $H^{\str}_0(K)$ represented by $(s_1,\dots ,s_{2l+1})\in \tilde{\Sigma}^l_K$ satisfying (0a') and (0b') is mapped to an element of $\bar{H}^{\str}_0(K)$  represented by
\[\begin{cases} 1 \in \bar{C}_0(\Sigma^0_K) =\Z_2 & \text{ if }l=0, \\
(s'_1,\dots ,s'_l)\in \bar{C}_0(\Sigma^l_K) & \text{ if }l\geq 1,
\end{cases}\]
such that $s'_i$ agrees with $s_{2i}$ up to parametrization for $i\in \{ 1,\dots ,l\}$,

\subsection{Proof of isomorphism}\label{app-isom}

The unit conormal bundle $\Lambda_K$ is identified with $\partial N$ by the map $\Lambda_K\to \partial N\colon (q,p)\mapsto q+ 2\epsilon \cdot p$. We choose a base point $x_0\in \partial N \cong \Lambda_K$.
In addition, fix a diffeomorphism $L_k \to N\setminus \partial N \colon (q,p) \mapsto (q, \frac{2\epsilon}{1+|p|^2}p)$.

For any two points $x,y\in \Lambda_K$, let $P_{x,y}$ denote the space of continuous paths $\gamma\colon [0,T]\to \Lambda_K$ such that $T>0$, $\gamma(0)=x$ and $\gamma(T)=y$. The inverse path of $\gamma\in P_{x,y}$ is denoted by $\gamma^{-1}\in P_{y,x}$.
The concatenation of $\gamma_1\in P_{x,y}$ and $\gamma_2\in P_{y,z}$ gives $\gamma_1\bullet \gamma_2 \in P_{x,z}$.
For each $a\in \mathcal{R}(\Lambda_K)$, fix paths $\gamma^0_a\in P_{x_0,a(0)}$ and $\gamma^1_a\in P_{x_0,a(T_a)}$.

First, referring to \cite[Section 6.2]{CELN}, let us introduce the fully non-commutative Legendrian contact homology of $\Lambda_K$.

For every $u\in \mathcal{M}_{\Lambda_K,J}(a;a_1,\dots ,a_m)$, let $\bar{u}\colon \partial D_{m+1}\to \Lambda_K$ denote the $\Lambda_K$-component of the map $\rest{u}{\partial D_{m+1}} \colon \partial D_{m+1}\to \R\times \Lambda_K$.
On the $k$-th boundary component, $\rest{\bar{u}}{\partial_k D_{m+1} }$ is extended to a path in $\Lambda_k$ on $\overline{\partial_k D_{m+1}}\cong [0,1]$.
Let us denote it by $\gamma^u_k$.
Then, we define a homology class $\partial_k \bar{u} \in H_0(P_{x_0,x_0};\Z_2) $ represented by loops as follows:
When $m=0$, we take $\gamma^1_a \bullet \gamma_k^u \bullet (\gamma^0_a)^{-1} \in P_{x_0,x_0}$.
When $m\geq 1$, we take
\[ P_{x_0,x_0} \ni \begin{cases}
\gamma^1_a \bullet \gamma_1^u \bullet (\gamma^1_{a_1})^{-1}  & \text{ if }k=1, \\
\gamma^0_{a_k} \bullet \gamma_k^u \bullet (\gamma^1_{a_{k+1}})^{-1}  & \text{ if } 2\leq k\leq m , \\
\gamma^0_{a_m}\bullet \gamma^u_{m+1} \bullet (\gamma^0_a)^{-1}  & \text{ if } k=m+1.
\end{cases}\]
Moreover, via the isomorphism
\[ H_0(P_{x_0,x_0};\Z_2) = \bZ_2[\pi_1(\Lambda_K,x_0)] \cong \Z_2[\pi_1(\partial N,x_0)] = \bZ_2[\lambda^{\pm},\mu^{\pm}]. \]
we regard $\partial_k \bar{u}$ as an element of $\Z_2[\lambda^{\pm},\mu^{\pm}]$.

Let $\mathcal{A}^{\nc}_*(\Lambda_K)$ be the unital $\Z_2$-graded algebra generated by $\mathcal{R}(\Lambda_K)\cup \{\lambda,\lambda^{-1},\mu,\mu^{-1}\}$ modulo the same relations as (\ref{rel-lambda-mu}).
Here, the degrees of $\lambda$ and $\mu$ are $0$.
Then, $\mathcal{A}^{\nc}_*(\Lambda_K)$ is a $\Z_2[\lambda^{\pm},\mu^{\pm}]$-NC-algebra.
We define a derivation $d^{\nc}_J\colon \mathcal{A}_*(\Lambda_K)\to \mathcal{A}_{*-1}(\Lambda_K)$ so that
$d_J(\lambda^{\pm})=d_J(\mu^{\pm})=0$ and
\[ d^{\nc}_J(a) = \sum_{m=0}^{\infty} \sum_{|a_1|+\dots +|a_m|=|a|-1} \left(  \sum_{u\in \bar{\calM}_{\Lambda_K,J}(a;a_1,\dots ,a_m)} (\partial_1 \bar{u}) a_1 (\partial_2 \bar{u}) \cdots a_m (\partial_{m+1} \bar{u}) \right) \]
for every $a\in \mathcal{R}(\Lambda_K)$,
and extend it by the Leibniz rule.
The pair $(\mathcal{A}^{\nc}_*(\Lambda_K),d^{\mathrm{nc}}_J)$ is isomorphic to a DGA in \cite[page 706]{CELN} (see also Remark \ref{rem-nc-LCH} below), and $d^{\mathrm{nc}}_J\circ d^{\mathrm{nc}}_J =0$ holds.
Let us denote the homology of $(\mathcal{A}^{\nc}_*(\Lambda_K),d^{\mathrm{nc}}_J)$ by $\LCH^{\nc}_*(\Lambda_K)$.
Note that in degree $0$,
\[ \LCH_0^{\nc}(\Lambda_K) = \mathcal{A}^{\nc}_0(\Lambda_K)/ d^{\nc}_J(\mathcal{A}^{\mathrm{nc}}_1(\Lambda_K)). \]

\begin{remark}\label{rem-nc-LCH}
A graded algebra $\mathscr{A}$ considered in \cite[page 706]{CELN} is generated by elements $\alpha \in H_0(P_{x_0,x_0})$ and words
$\alpha_1a_1\alpha_2 a_2 \cdots \alpha_m a_m \alpha_{m+1}$
for $m\in \bZ_{\geq 1}$, where $a_1,\dots ,a_m \in \mathcal{R}(\Lambda_K)$ and $\alpha_1 \in H_0(P_{x_0,a_1(T_{a_1})})$, $\alpha_k \in H_0(P_{ a_{k-1}(0),a_{k}(T_{a_k})} )$ for $2\leq k\leq m$ and $\alpha_{m+1} \in H_0(P_{a_m(0),x_0})$.
(Here, we reduce the coefficient ring from $\Z$ in \cite[page 706]{CELN} to $\Z_2$.)
By concatenation of paths, we have isomorphisms
\[ \bZ_2[\lambda^{\pm},\mu^{\pm}] \cong H_0(P_{x_0,x_0}) \to \begin{cases} H_0(P_{x_0,a_1(T_{a_1})}) \colon \gamma \mapsto \gamma\bullet \gamma^1_{a_1}, \\
H_0(P_{ a_{k-1}(0),a_{k}(T_{a_{k}})} ) \colon \gamma \mapsto (\gamma^0_{a_{k-1}})^{-1}\bullet \gamma\bullet \gamma^1_{a_{k}} , \\
H_0(P_{a_m(0),x_0}) \colon \gamma \mapsto (\gamma^0_{a_m})^{-1}\bullet \gamma .
\end{cases} \]
(The coefficient rings are $Z_2$.)
Via these isomorphisms, the DGA $(\mathscr{A}, \partial^{\mathrm{sy}})$ in \cite[page 706]{CELN} agrees with $(\mathcal{A}^{\nc}_*(\Lambda_K),d^{\mathrm{nc}}_J)$.
\end{remark}

\begin{remark}
In \cite[Section 6.2]{CELN}, a version of Chekanov-Eliashberg DGA is considered, which we denote here by
$(\widetilde{\mathscr{A}},\partial^{\mathrm{sing}}+\partial^{\mathrm{sy}})$.
It is described as follows (see also the construction of a DGA in \cite[Section 3]{EL} with coefficients in chains of a based loop space):
$\widetilde{\mathscr{A}}$ is generated by singular chains in the spaces
\[P_{x_0,a_1(T_{a_1})} \times P_{ a_1(0),a_{2}(T_{a_{2}})} \times \dots \times P_{ a_{m-1}(0),a_{m}(T_{a_{m}})} \times P_{a_m(0),x_0}\]
for all $a_1,\dots ,a_m\in \mathcal{R}(\Lambda_K)$. (When $m=0$, we consider $P_{x_0,x_0}$.)
The differential is given by $\partial^{\mathrm{sing}}+\partial^{\mathrm{sy}}$, where $\partial^{\mathrm{sing}}$ is the boundary operator for singular chains and $\partial^{\mathrm{sy}}$ is defined by using the moduli spaces $\bar{\calM}_{\Lambda_K,J}(a;a_1,\dots ,a_m)$.
The chain complex $(\widetilde{\mathscr{A}}, \partial^{\mathrm{sing}} + \partial^{\mathrm{sy}})$ is bigraded by the degree of singular chains and the degree of Reeb chords.
Viewing $\partial^{\mathrm{sing}}$ as the horizontal differential, we obtain a spectral sequence and its first page is concentrated in the $0$-th column, which agrees with
\[ (\mathscr{A}=H_0( \widetilde{\mathscr{A}} ,\partial^{\mathrm{sing}}) , \partial^{\mathrm{sy}}) \]
considered in Remark \ref{rem-nc-LCH}.
Here, we use the fact that $\Lambda_K \cong S^1 \times S^1 = K(\Z^2,1)$ and $H_*(P_{x,y}) = H_0(P_{x,y})$ for any $x,y\in \Lambda_K$.
Therefore, we have an isomorphism in degree $0$
\begin{align}\label{isom-loop-nc}
H_0 (\widetilde{\mathscr{A}},\partial^{\mathrm{sing}}+\partial^{\mathrm{sy}}) \cong  H_0 (\mathscr{A},\partial^{\mathrm{sy}}) \cong \LCH^{\nc}_0 (\Lambda_K) .
\end{align}
\end{remark}

\cite[Theorem 1.2]{CELN} asserts that there exists an isomorphism from $H_0 (\widetilde{\mathscr{A}},\partial^{\mathrm{sing}}+\partial^{\mathrm{sy}})$ to $H^{\str}_0(K)$.
It is induced by a chain map of \cite[Proposition 6.12]{CELN}.
Combining with (\ref{isom-loop-nc}), we obtain an isomorphism
\[ \tilde{\Psi}_K\colon \LCH^{\nc}_0(\Lambda_K) \to H^{\str}_0(K) , \]
and by composing it with the quotient map $\mathcal{A}^{\nc}_0(\Lambda_K)\to \LCH^{\nc}_0(\Lambda_K)$, we have
\[ \tilde{\Psi}'_K\colon \mathcal{A}^{\nc}_0(\Lambda_K) \to H^{\str}_0(K), \]
which preserves the product structures.
Let us describe the images under $\tilde{\Psi}'_K$ of the generators of $\mathcal{A}^{\nc}_0(\Lambda_K)$.
First, $\tilde{\Psi}'_K(\lambda)$ and $\tilde{\Psi}_K'(\mu)$ are given by loops in $\partial N$ based at $x_0$ which belong to $C_0(\tilde{\Sigma}^0_K)$ and represent the longitude and the meridian respectively. 
Next, for every $a\in \mathcal{R}(\Lambda_K)$ with $|a|=0$ and for $l\in \Z_{\geq 0}$,
$\tilde{\Psi}'_{K}(a)$ is represented by $(\tilde{\Psi}'_{K,l}(a))_{l=0,1,\dots}\in C_0(\tilde{\Sigma}_K)$ given by
\[ \tilde{\Psi}'_{K,l}(a) =  \sum_{u\in \calM_{L_K,l}(a)} \tilde{\bold{s}}^u \in C_0(\tilde{\Sigma}^l_K).\]
Here, $\tilde{\bold{s}}^u = (\tilde{s}^u_1,\tilde{s}^u_2,\dots ,\tilde{s}^u_{2l+1}) \in \tilde{\Sigma}^l_K$ for $u\in \calM_{L_K,l}(a)$ is determined as follows:
\begin{itemize}
\item For every $i\in \{1,\dots ,l\}$, $\rest{u}{\partial_{2i}D_{2l+1}}$ is extended to a curve on $\overline{\partial_{2i}D_{2l+1}}$ in $\bR^3$.
By pre-composing it with a parametrization of $\overline{\partial_{2i}D_{2l+1}}$, the path $\tilde{s}_{2i}$ in $\R^3$ is defined.
\item For every $i\in \{0,\dots ,l\}$, $\rest{u}{\partial_{2i+1}D_{2l+1}}\colon \partial_{2i+1}D_{2l+1} \to L_K \cong N\setminus \partial N$ is extended to a curve on $\overline{\partial_{2i}D_{2l+1}}$ in $N$.
When $i\neq 0,l$, by pre-composing it with a parametrization of $\overline{\partial_{2i+1}D_{2l+1}}$, the path  $\tilde{s}_{2i+1}$ in $N$ is defined.
When $i=0$ (resp. $i=l$), we in addition concatenate it with $\gamma^1_a$ (resp. $(\gamma^0_a)^{-1}$), then we obtain $\tilde{s}_{0}$ and $\tilde{s}_{2l+1}$.
\end{itemize}
For the parametrization of $\tilde{s}^u_1,\dots ,\tilde{s}^u_{2l+1}$, see the proof of \cite[Proposition 6.12]{CELN}.
In particular, up to parametrization, $\tilde{s}^u_{2i}$ agrees with $s^u_i$ defined by (\ref{path-si}).

Recall the map $\Psi'_K\colon \mathcal{A}_0(\Lambda_K)\to \bar{C}_0(\Sigma_K)$ of (\ref{Psi'K}).
By the same notation, let $\Psi'_K\colon \mathcal{A}_0(\Lambda_K)\to \bar{H}_0^{\str}(K)$ denote the composite map with the projection $\bar{C}_0(\Sigma_K) \to \bar{H}_0^{\str}(K)$.

\begin{lemma}
The following diagram commutes:
\begin{align}\label{diagram-A-H-C}
\begin{split} \xymatrix{
\mathcal{A}^{\mathrm{nc}}_0(\Lambda_K) \ar@{->>}[d]  \ar[r]^-{\tilde{\Psi}'_K} & H^{\str}_0(K)  \ar[r]^-{(I^{\nc}_K)^{-1}}_-{\cong} & \Cord^{\mathrm{nc}}(K) \ar[d]_-{j_K} \\
\mathcal{A}_0(\Lambda_K)  \ar[r]^-{\Psi'_K} & \bar{H}^{\str}_0(K) & \Cord (K) \ar[l]_-{I_K}^-{\cong} ,
}\end{split}
\end{align}
where the left vertical arrow is a $\Z_2$-algebra map which sends $\lambda$ and $\mu$ to $1$ and $a\in \mathcal{R}(\Lambda_K)$ to $a$.
\end{lemma}
\begin{proof}
From the previous observation on the map $ I_K \circ j_K \circ (I^{\nc}_K)^{-1}$ of (\ref{I-1-j-I}),
for every $a\in \mathcal{R}(\Lambda_K)$,
$ I_K \circ j_K \circ (I^{\nc}_K)^{-1} \circ \tilde{\Psi}'_K(a)$ is represented by $(\bold{x}_l)_{l=0,1,\dots}\in \bar{C}_0(\Sigma_K)$, where
\[\bold{x}_l= 
\sum_{u\in \calM_{L_K,l}(a)} (s^u_{1},s^u_2,\dots ,s^u_{l})  
\]
for $l\geq 1$ (here, we use the coincidence of $\tilde{s}^u_{2i}$ and $s^u_i$ up to parametrization),
and $\bold{x}_0 $ is equal to $ \#_{\bZ_2} \calM_{L_K,0}(a)  \in \bar{C}_0(\Sigma^0_K) = \bZ_2$ since
\[ (I_K\circ j_K\circ (I^{\nc}_K)^{-1})\left(\sum_{u\in \calM_{L_K,0}(a) } (\tilde{s}^u_1 ) \right) = \#_{\bZ_2} \calM_{L_K,0}(a). \]
In addition, $I_K \circ j_K \circ (I^{\nc}_K)^{-1} \circ \tilde{\Psi}'_K$ maps $\lambda$ and $\mu$ to $1$.
Comparing with the definition of $\Psi'_K$ in Section \ref{subsubsec-isom-LCH-str}, these computations show the commutativity of the diagram.
\end{proof}

\begin{proof}[Proof of Theorem \ref{thm-str-cord}]
The left vertical map of (\ref{diagram-A-H-C}) maps $ d^{\nc}_J(\mathcal{A}^{\nc}_1(\Lambda_K))$ onto $d_J(\mathcal{A}_1(\Lambda_K))$. Therefore, the commutativity of the diagram (\ref{diagram-A-H-C}) implies that $\Psi'_K(y) =0 $ for any $y\in d_J(\mathcal{A}_1(\Lambda_K))$.

Now, from (\ref{diagram-A-H-C}), we obtain a commutative diagram
\[\xymatrix{
\LCH^{\mathrm{nc}}_0(\Lambda_K) \ar[d]  \ar[rr]_-{\cong}^-{ (I^{\nc}_K)^{-1} \circ \tilde{\Psi}_K} & &\Cord^{\mathrm{nc}}(K) \ar[d]^-{j_K} \\
\LCH_0(\Lambda_K)  \ar[r]^-{\Psi_K} & \bar{H}^{\str}_0(K)  & \Cord (K) \ar[l]_-{I_K}^-{\cong} .
}\]
The left vertical map induces an isomorphism $\rest{\LCH^{\nc}_0(\Lambda_K)}{\lambda=0,\mu=0}^{\#} \to \LCH_0(\Lambda_K)$ of $\Z_2$-algebra.
(The notation $\rest{A}{\lambda=1,\mu=1}^{\#}$ comes from (\ref{A-reduction}).)
In addition, since
\[ (I^{\nc}_K)^{-1}\circ \tilde{\Psi}_K \colon \LCH^{\nc}_0(\Lambda_K) \to \Cord^{\nc}(K) \]
is an isomorphism as $\Z_2[\lambda^{\pm},\mu^{\pm}]$-NC-algebra, it induces an isomorphism of $\Z_2$-algebra 
\[\rest{\LCH^{\nc}_0(\Lambda_K)}{\lambda=1,\mu=1}^{\#}  \to \rest{\Cord^{\nc}(K)}{\lambda=1,\mu=1}^{\#} .\]
Let us denote this isomorphism by $\wh{\Psi}_K$.
Then, the following diagram commutes:
\[\xymatrix{
\rest{\LCH^{\mathrm{nc}}_0(\Lambda_K)}{\lambda=1,\mu=1}^{\#} \ar[d]^-{\cong}  \ar[rr]_-{\cong}^-{\wh{\Psi}_K} & & \rest{ \Cord^{\mathrm{nc}}(K)}{\lambda=1,\mu=1}^{\#} \ar[d]^-{\cong} \\
\LCH_0(\Lambda_K)  \ar[r]^-{\Psi_K} & \bar{H}^{\str}_0(K)  & \Cord (K) \ar[l]_-{I_K}^-{\cong} .
}\]
It follows that $\Psi_K \colon \LCH_0(\Lambda_K) \to \bar{H}^{\str}_0(K)$ is an isomorphism.
\end{proof}

\printbibliography

@article {Ng,
    AUTHOR = {Ng, Lenhard},
     TITLE = {Framed knot contact homology},
   JOURNAL = {Duke Math. J.},
  FJOURNAL = {Duke Mathematical Journal},
    VOLUME = {141},
      YEAR = {2008},
    NUMBER = {2},
     PAGES = {365--406},
      ISSN = {0012-7094},
   MRCLASS = {53D40 (53D12 53D35 57M27)},
  MRNUMBER = {2376818},
MRREVIEWER = {Tobias Ekholm},
       DOI = {10.1215/S0012-7094-08-14125-0},
       URL = {https://doi.org/10.1215/S0012-7094-08-14125-0},
}

@article {CELN,
    AUTHOR = {Cieliebak, Kai and Ekholm, Tobias and Latschev, Janko and Ng,
              Lenhard},
     TITLE = {Knot contact homology, string topology, and the cord algebra},
   JOURNAL = {J. \'{E}c. polytech. Math.},
  FJOURNAL = {Journal de l'\'{E}cole polytechnique. Math\'{e}matiques},
    VOLUME = {4},
      YEAR = {2017},
     PAGES = {661--780},
      ISSN = {2429-7100},
   MRCLASS = {57M27 (53D42 55P50 57R17)},
  MRNUMBER = {3665612},
MRREVIEWER = {Daniel V. Mathews},
       DOI = {10.5802/jep.55},
       URL = {https://doi-org.utokyo.idm.oclc.org/10.5802/jep.55},
}

@article {DR,
    AUTHOR = {Dimitroglou Rizell, Georgios},
     TITLE = {Legendrian ambient surgery and {L}egendrian contact homology},
   JOURNAL = {J. Symplectic Geom.},
  FJOURNAL = {The Journal of Symplectic Geometry},
    VOLUME = {14},
      YEAR = {2016},
    NUMBER = {3},
     PAGES = {811--901},
      ISSN = {1527-5256,1540-2347},
   MRCLASS = {53D42 (53D35 57R17)},
  MRNUMBER = {3548486},
MRREVIEWER = {Janko\ Latschev},
       DOI = {10.4310/JSG.2016.v14.n3.a6},
       URL = {https://doi.org/10.4310/JSG.2016.v14.n3.a6},
}

@article {DR-lift,
    AUTHOR = {Dimitroglou Rizell, Georgios},
     TITLE = {Lifting pseudo-holomorphic polygons to the symplectisation of
              {$P\times\Bbb{R}$} and applications},
   JOURNAL = {Quantum Topol.},
  FJOURNAL = {Quantum Topology},
    VOLUME = {7},
      YEAR = {2016},
    NUMBER = {1},
     PAGES = {29--105},
      ISSN = {1663-487X},
   MRCLASS = {53D42 (53D12 53D40)},
  MRNUMBER = {3459958},
MRREVIEWER = {Weiwei Wu},
       DOI = {10.4171/QT/73},
       URL = {https://doi.org/10.4171/QT/73},
}

@article {EES-R,
    AUTHOR = {Ekholm, Tobias and Etnyre, John and Sullivan, Michael},
     TITLE = {The contact homology of {L}egendrian submanifolds in {${\Bbb
              R}^{2n+1}$}},
   JOURNAL = {J. Differential Geom.},
  FJOURNAL = {Journal of Differential Geometry},
    VOLUME = {71},
      YEAR = {2005},
    NUMBER = {2},
     PAGES = {177--305},
      ISSN = {0022-040X},
   MRCLASS = {53D35 (57R17)},
  MRNUMBER = {2197142},
MRREVIEWER = {Joshua M. Sabloff},
       URL = {http://projecteuclid.org/euclid.jdg/1143651770},
}

@article {Cha,
    AUTHOR = {Chantraine, Baptiste},
     TITLE = {Lagrangian concordance of {L}egendrian knots},
   JOURNAL = {Algebr. Geom. Topol.},
  FJOURNAL = {Algebraic \& Geometric Topology},
    VOLUME = {10},
      YEAR = {2010},
    NUMBER = {1},
     PAGES = {63--85},
      ISSN = {1472-2747,1472-2739},
   MRCLASS = {57R17 (53D12 57M50)},
  MRNUMBER = {2580429},
MRREVIEWER = {Lenhard\ L.\ Ng},
       DOI = {10.2140/agt.2010.10.63},
       URL = {https://doi.org/10.2140/agt.2010.10.63},
}

@article {EES,
    AUTHOR = {Ekholm, Tobias and Etnyre, John and Sullivan, Michael},
     TITLE = {Legendrian contact homology in {$P\times\Bbb R$}},
   JOURNAL = {Trans. Amer. Math. Soc.},
  FJOURNAL = {Transactions of the American Mathematical Society},
    VOLUME = {359},
      YEAR = {2007},
    NUMBER = {7},
     PAGES = {3301--3335},
      ISSN = {0002-9947},
   MRCLASS = {53D35 (53D40)},
  MRNUMBER = {2299457},
MRREVIEWER = {Michael J. Usher},
       DOI = {10.1090/S0002-9947-07-04337-1},
       URL = {https://doi.org/10.1090/S0002-9947-07-04337-1},
}

@article {EHK,
    AUTHOR = {Ekholm, Tobias and Honda, Ko and K\'{a}lm\'{a}n, Tam\'{a}s},
     TITLE = {Legendrian knots and exact {L}agrangian cobordisms},
   JOURNAL = {J. Eur. Math. Soc. (JEMS)},
  FJOURNAL = {Journal of the European Mathematical Society (JEMS)},
    VOLUME = {18},
      YEAR = {2016},
    NUMBER = {11},
     PAGES = {2627--2689},
      ISSN = {1435-9855},
   MRCLASS = {53D42 (53D10 53D40 57M27 57R90)},
  MRNUMBER = {3562353},
MRREVIEWER = {Georgios Dimitroglou Rizell},
       DOI = {10.4171/JEMS/650},
       URL = {https://doi.org/10.4171/JEMS/650},
}

@article {CEL,
    AUTHOR = {Cieliebak, K. and Ekholm, T. and Latschev, J.},
     TITLE = {Compactness for holomorphic curves with switching {L}agrangian
              boundary conditions},
   JOURNAL = {J. Symplectic Geom.},
  FJOURNAL = {The Journal of Symplectic Geometry},
    VOLUME = {8},
      YEAR = {2010},
    NUMBER = {3},
     PAGES = {267--298},
      ISSN = {1527-5256,1540-2347},
   MRCLASS = {53D40 (53D12 53D42 53D45)},
  MRNUMBER = {2684508},
MRREVIEWER = {Lenhard\ L.\ Ng},
       URL = {http://projecteuclid.org/euclid.jsg/1283865584},
}

@article{Kalman,
  title={Contact homology and one parameter families of Legendrian knots},
  author={T{\'a}m{\'a}s Kalm{\'a}n},
  journal={Geometry \& Topology},
  year={2004},
  volume={9},
  pages={2013-2078},
  url={https://api.semanticscholar.org/CorpusID:8307055}
}

@article{Casals,
  title={Braid loops with infinite monodromy on the Legendrian contact DGA},
  author={Roger Casals and Lenhard L. Ng},
  journal={Journal of Topology},
  year={2021},
  volume={15},
  url={https://api.semanticscholar.org/CorpusID:230799429}
}

@article {Ng2003,
    AUTHOR = {Ng, Lenhard},
     TITLE = {Knot and braid invariants from contact homology. {II}},
      NOTE = {With an appendix by the author and Siddhartha Gadgil},
   JOURNAL = {Geom. Topol.},
  FJOURNAL = {Geometry and Topology},
    VOLUME = {9},
      YEAR = {2005},
     PAGES = {1603--1637},
      ISSN = {1465-3060,1364-0380},
   MRCLASS = {57M27 (20F36 53D35 57R17)},
  MRNUMBER = {2175153},
MRREVIEWER = {Joshua\ M.\ Sabloff},
       DOI = {10.2140/gt.2005.9.1603},
       URL = {https://doi.org/10.2140/gt.2005.9.1603},
}

@article {O24,
    AUTHOR = {Okamoto, Yukihiro},
     TITLE = {Legendrian non-isotopic unit conormal bundles in high
              dimensions},
   JOURNAL = {J. Topol.},
  FJOURNAL = {Journal of Topology},
    VOLUME = {18},
      YEAR = {2025},
    NUMBER = {3},
     PAGES = {Paper No. e70039},
      ISSN = {1753-8416,1753-8424},
   MRCLASS = {53D42 (53D12 55P50 57R17)},
  MRNUMBER = {4960403},
       DOI = {10.1112/topo.70039},
       URL = {https://doi.org/10.1112/topo.70039},
}

@article {O22,
    AUTHOR = {Okamoto, Yukihiro},
     TITLE = {Toward a topological description of {L}egendrian contact
              homology of unit conormal bundles},
   JOURNAL = {Algebr. Geom. Topol.},
  FJOURNAL = {Algebraic \& Geometric Topology},
    VOLUME = {25},
      YEAR = {2025},
    NUMBER = {2},
     PAGES = {951--1027},
      ISSN = {1472-2747,1472-2739},
   MRCLASS = {53D42 (55P50 57R17)},
  MRNUMBER = {4930557},
       DOI = {10.2140/agt.2025.25.951},
       URL = {https://doi.org/10.2140/agt.2025.25.951},
}

@article {EL,
    AUTHOR = {Ekholm, Tobias and Lekili, Yank\i},
     TITLE = {Duality between {L}agrangian and {L}egendrian invariants},
   JOURNAL = {Geom. Topol.},
  FJOURNAL = {Geometry \& Topology},
    VOLUME = {27},
      YEAR = {2023},
    NUMBER = {6},
     PAGES = {2049--2179},
      ISSN = {1465-3060,1364-0380},
   MRCLASS = {57R17},
  MRNUMBER = {4634745},
MRREVIEWER = {Alexander\ Fel\cprime shtyn},
       DOI = {10.2140/gt.2023.27.2049},
       URL = {https://doi.org/10.2140/gt.2023.27.2049},
}

@article{ES,
  title={Skeins on branes},
  author={Ekholm, Tobias and Shende, Vivek},
  journal={arXiv preprint arXiv:1901.08027},
  year={2019}
}

@article {APS,
    AUTHOR = {Abbondandolo, Alberto and Portaluri, Alessandro and Schwarz,
              Matthias},
     TITLE = {The homology of path spaces and {F}loer homology with conormal
              boundary conditions},
   JOURNAL = {J. Fixed Point Theory Appl.},
  FJOURNAL = {Journal of Fixed Point Theory and Applications},
    VOLUME = {4},
      YEAR = {2008},
    NUMBER = {2},
     PAGES = {263--293},
      ISSN = {1661-7738,1661-7746},
   MRCLASS = {53D40 (53D12)},
  MRNUMBER = {2465553},
MRREVIEWER = {Tobias\ Ekholm},
       DOI = {10.1007/s11784-008-0097-y},
       URL = {https://doi.org/10.1007/s11784-008-0097-y},
}

@article {ENS,
    AUTHOR = {Ekholm, Tobias and Ng, Lenhard and Shende, Vivek},
     TITLE = {A complete knot invariant from contact homology},
   JOURNAL = {Invent. Math.},
  FJOURNAL = {Inventiones Mathematicae},
    VOLUME = {211},
      YEAR = {2018},
    NUMBER = {3},
     PAGES = {1149--1200},
      ISSN = {0020-9910,1432-1297},
   MRCLASS = {57M27 (53D42)},
  MRNUMBER = {3763406},
MRREVIEWER = {Janko\ Latschev},
       DOI = {10.1007/s00222-017-0761-1},
       URL = {https://doi.org/10.1007/s00222-017-0761-1},
}

@article {Shende,
    AUTHOR = {Shende, Vivek},
     TITLE = {The conormal torus is a complete knot invariant},
   JOURNAL = {Forum Math. Pi},
  FJOURNAL = {Forum of Mathematics. Pi},
    VOLUME = {7},
      YEAR = {2019},
     PAGES = {e6, 16},
      ISSN = {2050-5086},
   MRCLASS = {57M25 (55N30 57M27)},
  MRNUMBER = {4010558},
MRREVIEWER = {Mohamed\ Elhamdadi},
       DOI = {10.1017/fmp.2019.1},
       URL = {https://doi.org/10.1017/fmp.2019.1},
}

@article {Asp,
    AUTHOR = {Asplund, Johan},
     TITLE = {Fiber {F}loer cohomology and conormal stops},
   JOURNAL = {J. Symplectic Geom.},
  FJOURNAL = {The Journal of Symplectic Geometry},
    VOLUME = {19},
      YEAR = {2021},
    NUMBER = {4},
     PAGES = {777--864},
      ISSN = {1527-5256,1540-2347},
   MRCLASS = {53D40},
  MRNUMBER = {4371551},
MRREVIEWER = {Jun\ Zhang},
       DOI = {10.4310/JSG.2021.v19.n4.a1},
       URL = {https://doi.org/10.4310/JSG.2021.v19.n4.a1},
}

@article {FMP20,
    AUTHOR = {Fern\'andez, Eduardo and Mart\'inez-Aguinaga, Javier and
              Presas, Francisco},
     TITLE = {Fundamental groups of formal {L}egendrian and horizontal
              embedding spaces},
   JOURNAL = {Algebr. Geom. Topol.},
  FJOURNAL = {Algebraic \& Geometric Topology},
    VOLUME = {20},
      YEAR = {2020},
    NUMBER = {7},
     PAGES = {3219--3312},
      ISSN = {1472-2747,1472-2739},
   MRCLASS = {57K33 (53D10 58A17)},
  MRNUMBER = {4194282},
MRREVIEWER = {Laura\ P.\ Starkston},
       DOI = {10.2140/agt.2020.20.3219},
       URL = {https://doi.org/10.2140/agt.2020.20.3219},
}

@article {SS,
    AUTHOR = {Sabloff, Joshua M. and Sullivan, Michael G.},
     TITLE = {Families of {L}egendrian submanifolds via generating families},
   JOURNAL = {Quantum Topol.},
  FJOURNAL = {Quantum Topology},
    VOLUME = {7},
      YEAR = {2016},
    NUMBER = {4},
     PAGES = {639--668},
      ISSN = {1663-487X,1664-073X},
   MRCLASS = {57R17 (57Q45)},
  MRNUMBER = {3593565},
MRREVIEWER = {Klaus\ Mohnke},
       DOI = {10.4171/QT/82},
       URL = {https://doi.org/10.4171/QT/82},
}

@article {BEHWZ,
    AUTHOR = {Bourgeois, F. and Eliashberg, Y. and Hofer, H. and Wysocki, K.
              and Zehnder, E.},
     TITLE = {Compactness results in symplectic field theory},
   JOURNAL = {Geom. Topol.},
  FJOURNAL = {Geometry and Topology},
    VOLUME = {7},
      YEAR = {2003},
     PAGES = {799--888},
      ISSN = {1465-3060},
   MRCLASS = {53D45 (53D35 53D40 57R17)},
  MRNUMBER = {2026549},
MRREVIEWER = {Kai Cieliebak},
       DOI = {10.2140/gt.2003.7.799},
       URL = {https://doi.org/10.2140/gt.2003.7.799},
}

@inproceedings{Hatcher2002TopologicalMS,
  title={Topological Moduli Spaces of Knots},
  author={Allen Hatcher},
  year={2002},
 url={https://api.semanticscholar.org/CorpusID:8472234},
 options={url=true}
}

@article{Budney2010,
author = {Budney, Ryan},
title = {Topology of knot spaces in dimension 3},
journal = {Proceedings of the London Mathematical Society},
volume = {101},
number = {2},
pages = {477-496},
year = {2010}
}

@incollection {Budney2006,
    AUTHOR = {Budney, Ryan},
     TITLE = {A family of embedding spaces},
 BOOKTITLE = {Groups, homotopy and configuration spaces},
    SERIES = {Geom. Topol. Monogr.},
    VOLUME = {13},
     PAGES = {41--83},
 PUBLISHER = {Geom. Topol. Publ., Coventry},
      YEAR = {2008},
   MRCLASS = {57R40 (55P48 55Q45 57R50)},
  MRNUMBER = {2508201},
MRREVIEWER = {Brian\ A.\ Munson},
       DOI = {10.2140/gtm.2008.13.41},
       URL = {https://doi.org/10.2140/gtm.2008.13.41},
}

@article{Budney2007,
title = {Little cubes and long knots},
journal = {Topology},
volume = {46},
number = {1},
pages = {1-27},
year = {2007},
issn = {0040-9383},
doi = {https://doi.org/10.1016/j.top.2006.09.001},
url = {https://www.sciencedirect.com/science/article/pii/S0040938306000541},
author = {Ryan Budney}
}

@article{Hatcher83,
 ISSN = {0003486X, 19398980},
 URL = {http://www.jstor.org/stable/2007035},
 author = {Allen E. Hatcher},
 journal = {Annals of Mathematics},
 number = {3},
 pages = {553--607},
 publisher = {[Annals of Mathematics, Trustees of Princeton University on Behalf of the Annals of Mathematics, Mathematics Department, Princeton University]},
 title = {A Proof of the Smale Conjecture,
          $\operatorname{Diff}(S^3) \simeq O(4)$},
 urldate = {2025-12-17},
 volume = {117},
 year = {1983}
}

@article{martinez2024legendrian,
  title={On the Legendrian realisation of parametric families of knots},
  author={Mart{\'\i}nez-Aguinaga, Javier},
  journal={arXiv preprint arXiv:2406.04293},
  year={2024}
}

@article {CG,
    AUTHOR = {Casals, Roger and Gao, Honghao},
     TITLE = {Infinitely many {L}agrangian fillings},
   JOURNAL = {Ann. of Math. (2)},
  FJOURNAL = {Annals of Mathematics. Second Series},
    VOLUME = {195},
      YEAR = {2022},
    NUMBER = {1},
     PAGES = {207--249},
      ISSN = {0003-486X,1939-8980},
   MRCLASS = {57K33 (53D10)},
  MRNUMBER = {4358415},
MRREVIEWER = {Youlin\ Li},
       DOI = {10.4007/annals.2022.195.1.3},
       URL = {https://doi.org/10.4007/annals.2022.195.1.3},
}

@article {Golovko,
    AUTHOR = {Golovko, Roman},
     TITLE = {A note on the infinite number of exact {L}agrangian fillings
              for spherical spuns},
   JOURNAL = {Pacific J. Math.},
  FJOURNAL = {Pacific Journal of Mathematics},
    VOLUME = {317},
      YEAR = {2022},
    NUMBER = {1},
     PAGES = {143--152},
      ISSN = {0030-8730,1945-5844},
   MRCLASS = {53D12 (53D42)},
  MRNUMBER = {4440942},
MRREVIEWER = {Toru\ Yoshiyasu},
       DOI = {10.2140/pjm.2022.317.143},
       URL = {https://doi.org/10.2140/pjm.2022.317.143},
}

@inproceedings {Ng-plane,
    AUTHOR = {Ng, Lenhard},
     TITLE = {Plane curves and contact geometry},
 BOOKTITLE = {Proceedings of {G}\"okova {G}eometry-{T}opology {C}onference
              2005},
     PAGES = {165--174},
 PUBLISHER = {G\"okova Geometry/Topology Conference (GGT), G\"okova},
      YEAR = {2006},
      ISBN = {1-57146-152-3},
   MRCLASS = {53D35 (57M27 57R17 57R42)},
  MRNUMBER = {2282015},
MRREVIEWER = {Joshua\ M.\ Sabloff},
}

\noindent Yukihiro Okamoto:
Department of Mathematical Sciences,
Tokyo Metropolitan University,
Minami-osawa, Hachioji, Tokyo, 192-0397, Japan

\noindent
\textit{E-mail address}: 
\texttt{yukihiro@tmu.ac.jp}

\medskip

\noindent Mari\'{a}n Poppr

\noindent \textit{E-mail address}: \texttt{marianpoppr@centrum.cz}

\end{document}